\documentclass[12pt]{article}
\usepackage{amsmath}
\usepackage{amssymb}
\usepackage{amsfonts}
\usepackage{amsthm}  % for theorem style
\usepackage{mathrsfs}   	
\usepackage{amsbsy}		% for bold \pmb{}
\usepackage{graphicx}
\usepackage{makeidx}
\usepackage[leftbars]{changebar}
\usepackage{bm}
\usepackage[all]{xy}
\usepackage[left,modulo]{lineno}
\usepackage{tkz-base}
\usepackage{circuitikz}	% for electricity
\usepackage{mathdots}
\usepackage[ps2pdf]{hyperref} % ps2pdf needed for compilation via Latex -> dvips -> ps2pdf
\hypersetup{ colorlinks=true,linkcolor=blue, citecolor=blue, urlcolor=blue  }

\theoremstyle{plain}
\newtheorem{theorem}{Theorem}[section]
\newtheorem{theorem-definition}[theorem]{Theorem and Definition}
\newtheorem{definition}[theorem]{Definition}
\newtheorem{corollary}[theorem]{Corollary}
\newtheorem{lemma}[theorem]{Lemma}

\newtheorem{notation}[theorem]{Notation}
\newtheorem{proposition}[theorem]{Proposition}

\newtheoremstyle{myexstyle}
{}{}{}{}{\bfseries}{.}{ }{}
\theoremstyle{myexstyle}
\newtheorem{example}[theorem]{Example}

\newtheorem{remark}[theorem]{Remark}
\newtheorem{notations}[theorem]{Notations}
\newtheorem{warning}[theorem]{Warning}
\numberwithin{equation}{section}

\newcommand{\N}{\mathbb{N}}
\newcommand{\R}{\mathbb{R}}

\newcommand{\pr}{\operatorname{pr}}

\begin{document}
%\frontmatter
\title{Partial Dirac Structures\\
 and Dynamical Systems}

\date{}
\author{Fernand Pelletier \& Patrick Cabau}

\maketitle

\begin{abstract}
In a previous paper (\cite{PeCa24}), the notion of Dirac structure in finite dimension was extended to the convenient setting. In particular, we introduce the notion of \emph{partial Dirac structure on a convenient manifold} and look for which all geometrical results in finite dimension which are still true in this infinite dimensional framework.  Note that  this context is justified by many mechanical infinite dimensional examples which recover all the classical ones, such as Hilbert, Banach, Fr\'{e}chet context or direct limits of Banach spaces. Another reason is that if we want to extend the variational technics, as in \cite{YoMa06II}, to the infinite dimensional   setting, the category of convenient vector spaces is cartesian closed, which is not the case for the category of locally convex vector spaces and so this variational approach does not work in this last setting. 
In this second part, first, we look for an adaptation of the results obtained in \cite{YoMa06II} for a variational approach of constraint Lagrangian on a subbundle of a Banach manifold. Then we study the same type of problem but for constraint Lagrangians on a \textit{singular} distribution. Theses results are finally applied to the  characterization of normal geodesics for a conical Finsler metric on a Banach manifold.
\end{abstract}

% \tableofcontents


\noindent
\textbf{Classification MSC 2020.--}  
22E65, %Infinite-dimensional Lie groups and their Lie algebras: general properties 
46T05, %Infinite-dimensional manifolds
58B20, %Riemannian, Finsler and other geometric structures on infinite-dimensional manifolds
70G45,  %Differential geometric methods (tensors, connections, symplectic, Poisson, contact, Riemannian, nonholonomic, etc.) for problems in mechanics
70H45  %Constrained dynamics, Dirac’s theory of constraints
.

\section{Introduction}
\label{__Introduction}

An algebraic theory of Dirac structures first appeared in \cite{Dor87} in the Hamiltonian framework of integrable evolution equations.  
The geometric approach was realized in \cite{Cou90} and \cite{CoWe88}, in the finite dimensional setting, as a subset $D$ of the Pontryagin bundle $TM \oplus T^\ast M$ satisfying certain conditions.  This allows to treat both weak symplectic  and Poisson structures in a unified framework.\\
The notion was intensively developed, in finite dimension, in connection with the study of non holonomic mechanical systems or mechanical systems with constraints.   The Hamiltonian framework refers to the principle of least action and Legendre transformation through Euler equations  and a variational approach (cf. \cite{JoRa12} and \cite{YJM10}).\\

The purpose of \cite{PeCa24} and this paper is to propose three directions:
\begin{enumerate}
\item[(1)]   
{\it From section~\ref{__PartialAlmostDiracStructuresAndImplicitHamiltonians} to section~\ref{__VariationalApproach} with  applications in section~\ref{__DynamicSystemsImplicitLagrangiansOnDirectLimitsOfFiniteDimensionalManifolds}},\\ 
taking into account the context of \cite{JaZw23} and \cite{SB-KZB17} (see also all references of the same authors inside), we  propose a generalization to the convenient setting of the finite dimensional  geometrical context of Dirac structure and its associated Hamiltonian framework (among of a lot of  references, cf., for example, \cite{JoRa12}, \cite{YJM10}), a generalization of  such  approaches to the infinite dimensional setting for constraint mechanical system in restriction to a subbundle of the configuration manifold. This context is justified by the fact that some infinite mechanical systems and some infinite mathematical control theory problems sit in different infinite dimensional contexts such as Hilbert, Banach, Fr\'echet or even convenient settings: 
for instance the Korteweg de Vries hierarchy, the non-linear Schr\"odinger equation or geodesic problems on groups of diffeomorphisms or sub-Riemanian geometry (cf. for example  \cite{DuNo84}, \cite{Fad80}, \cite{Gra94}, \cite{GKMS18}, \cite{KMM20},  \cite{GMV15}, \cite{Arg20}, \cite{Pel19}), \cite{SB-KZB17}).
\item[(2)] 
{\it In section~\ref{__DiracStructuresOnSingularConstraintsInTheBanachSetting}},\\
 extension of the previous results to infinite mechanical systems constraint by a singular distribution on the configuration manifold.
\item[(3)]
{\it In  section~\ref{__NormalGeodesicsConicSubHilbertFinsler}},\\
 application of the  previous results to "conic Finsler metrics".
\end{enumerate} 

Recall that, {\bf  in a  first part  (\cite{PeCa24}}), we have proposed  a version of partial Dirac structures on convenient manifolds and studied their geometrical properties:\\
 according to the framework of partial Poisson structures (cf. \cite{CaPe23}, Chapter~7), we introduce the notion on the partial  bundle
\[
{T}^{\mathfrak{p}}M = TM \oplus T^\flat M
\]
where $T^\flat M$ is a weak sub-bundle of the  cotangent bundle $T'M$ of a convenient manifold\footnote{i.e. the inclusion of $T^\flat M$ in $T^\prime M$ is an injective convenient morphism.}.

As in finite dimension, this bundle can be provided with a Courant bracket and an anchor 
$\rho^D: {T}^{\mathfrak{p}}(M)\to TM$ which is called a partial Dirac structure if the Courant bracket in restriction to local sections of  ${T}^{\mathfrak{p}}(M)$ takes values in such a set of sections and if the induced bracket satisfies the Jacobi identity, which means, by the way, that 
$({T}^{\mathfrak{p}}(M),\rho^D)$ is provided with a Lie algebroid structure.\\
In finite dimension, any Dirac structure is integrable, that is the distribution  defines a (singular) foliation  provided with a canonical pre-symplectic structure on each leaf. Unfortunately, even in the Hilbert setting, this result in no more true. The interested reader  will find sufficient conditions in the Banach setting under which  a singular Banach foliation is associated to a Banach partial Dirac structure and conditions under which such a foliation can be  provided with a pre-symplectic structure.  All along the text, the reader will find examples of  partial Dirac structures.\\

{\bf  In this second part}, all these geometrical results will give rise to applications on:
\begin{enumerate} 
\item[1.]
Dynamical systems associated with implicit Hamiltonians, non degenerated and degenerated Lagrangians (as a generalization to our context of  \cite{YoMa06II});
\item[2.]  
Dynamical systems of non degenerate Lagrangians on "singular non  holonomic constraints" with application, in the Banach setting, to "conic  sub-Finsler" metric and, in particular, (singular) sub-Riemannian  metric with characterization of their geodesics as a generalization of \cite{Arg20}.
\end{enumerate}

This paper is organized as follows.\\  Section~\ref{__DiracStructuresOnFiniteDimensionalManifolds}  summaries  most  classical results about Dirac structures in finite dimension. It  can  be considered as a  basis of comparison between such results in finite dimension  and  all the adaptations proposed in this work. In particular,  {\it this allows to point out  all  the obstructions we have met in this work  for such a generalization in the infinite dimensional setting.}. 
In section~\ref{__PartialAlmostDiracStructuresAndImplicitHamiltonians}, after having recalling the notion of (weak) Pontryagin bundle $T^\mathfrak{p}M=TM\oplus T^\flat M \to M$ in the convenient setting, we define various concepts such as the Courant bracket, partial almost Dirac structures  and study its properties. We end this section with the core concept of implicit Hamiltonian system. 
In section~\ref{__ImplicitLagrangians}, we extend to the convenient setting the notion of implicit Lagrangians and some results obtained in finite dimension by H. Yoshimura, H. L. Jacobs and J. E. Marsden in \cite{YJM10}. 
Section~\ref{__VariationalApproach} is devoted to the link, in the convenient setting, between the variational approach of implicit Lagrangian systems and implicit Hamiltonian systems.    
Section~\ref{__DynamicSystemsImplicitLagrangiansOnDirectLimitsOfFiniteDimensionalManifolds} is dedicated to the properties of the direct limit (resp. projective  limit) of an ascending  sequence (resp. projective sequence) of Banach almost Dirac structures and sufficient conditions under which  direct limit  of a ascending  sequence of Hilbert  integrable Dirac structures is a convenient integrable Dirac structure. 
In section~\ref{__DiracStructuresOnSingularConstraintsInTheBanachSetting},
we present, in the Banach setting, a generalization  to a singular distribution on $M$ of  the induced  partial almost Dirac structure associated to a closed subbundle $\Delta_M$ of $TM$.  
The last section is devoted to the study of normal geodesics  of a  conic sub Hilbert-Finsler structure on a Banach manifold.
\section{Dirac Structures on Finite Dimensional Manifolds}
\label{__DiracStructuresOnFiniteDimensionalManifolds}

Let $M$ be a manifold of dimension $n$, $TM$ its tangent bundle and $T^\ast M$ its cotangent bundle. 

\subsection{Dirac Structures}
\label{___DiracrStructures}

The Pontryagin bundle\index{Pontryagin bundle} 
${T}^{\mathfrak{p}}(M) = TM \oplus T^\ast(M)$ is the Whitney sum bundle over $M$, i.e. it is the bundle over the base $M$ where the fibres over the point $x \in M$ equal to $T_x M \oplus T_x^\ast M$. This bundle is equipped with 
%the natural projections
%\[
%\operatorname{pr}_T : {T}^{\mathfrak{p}}(M) \to TM
%\textrm{  et  }
%\operatorname{pr}_{T^\ast} : {T}^{\mathfrak{p}}(M) \to T^\ast(M)
%\]
%as well as
\begin{description}
\item[$\bullet$]
a nondegenerate, symmetric fibrewise bilinear form $<.,.>$ defined for any $x \in M$, any pair $(X_x,Y_x)$ of $T_xM$ and any pair $(\alpha_x,\beta_x)$ of $T_x^\ast M$ by
\begin{eqnarray}
\label{eq_DirectSymmetricBilinearFormDiracStructures}
<(X_x,\alpha_x),(Y_x,\beta_x)>
= 
\beta_x(X_x) + \alpha_x(Y_x)
\end{eqnarray}
\item[$\bullet$]
the \emph{Courant bracket}\index{Courant bracket}\index{bracket!Courant} defined for any pair $(X,Y)$ of vector fields and any pair $(\alpha,\beta)$ of $1$-forms by
\begin{eqnarray}
\label{eq_CourantBracketDiracStructures}
[[
(X,\alpha),(Y,\beta)
]]
=
\left( 
[X,Y], 
L_X \beta - L_Y \alpha 
+\dfrac{1}{2} d \left( \alpha(Y) - \beta(X) \right) 
\right)
\end{eqnarray}
\end{description}
The Courant bracket satisfies the identity:
\[
\circlearrowleft [[[[a1, a2]], a3]]  
=
-
\circlearrowleft \dfrac{1}{3} d( < [[a1, a2]], a3 > )
\]
where $\circlearrowleft$ stands for the sum of cyclic permutations.

\begin{notation}
\label{N_OrthogonalDiracStructure}
For a vector sub-bundle $D$ of ${T}^{\mathfrak{p}}M$, we denote $D^\perp$ the orthogonal of $D$ with respect to $<.,.>$:
\[
D^\perp 
= 
\{ (U,\eta) \in {T}^{\mathfrak{p}}M :\;
\forall (X,\alpha) \in D, < (X,\alpha),(U,\eta) > =0  \}
\]
\end{notation}
\begin{notation}
\label{N_SetOfSectionsOfAsub-bundle}
We denote $\Gamma(D)$ the set of smooth sections of a vector sub-bundle $D$ of ${T}^{\mathfrak{p}}M$.
\end{notation}

\begin{definition}
\label{D_DiracStructures}
An \emph{almost Dirac structure}\index{almost!Dirac structure} on $M$ is a vector sub-bundle $D$ of ${T}^{\mathfrak{p}}M$ satisfying
\begin{description}
\item[\textbf{(ADS)}]
{\hfil 
$D = D^\perp$
}
\end{description}
\end{definition}
Moreover, if $\Gamma(D)$\index{GammaL@$\Gamma(D)$ (distribution)} fulfils the condition
\begin{description}
\item[\textbf{(DS)}]
{\hfil 
$[[\Gamma(D), \Gamma(D)]] \subset \Gamma(D)$
}
\end{description}
$D$ is a \emph{Dirac structure}\footnote{In \cite{CoWe88}, the condition \textbf{(DS)} is replaced by
\[
\forall 
\left( 
\left( X_1,\alpha_1 \right) , 
\left( X_2,\alpha_2 \right) , 
\left( X_3,\alpha_3 \right) 
\right)
\in \left( \Gamma(D) \right) ^3,\;
\circlearrowleft \left\langle L_{X_1} \alpha_2 , X_3
 \right\rangle = 0 .
\]
}
\index{Dirac structure} on $M$.\\
%\begin{eqnarray}
%\label{eq_BracketDiracStructures}
%[[
%(X,\alpha),(Y,\beta)
%]]
%=
%\left( [X,Y], L_X \beta - i_Y d\alpha \right)
%\end{eqnarray}
Since the pairing $<.,.>$ has split signature, \textbf{(ADS)} is equivalent to
\begin{enumerate}
\item[$\bullet$]
$<.,.>_{|D} = 0$
\item[$\bullet$]
$\operatorname{rank} D = n$ 
\end{enumerate}
For any Lagrangian sub-bundle $L \subset TM$, the expression
\[
\mathbf{T}_{L} \left( a_1, a_2, a_3 \right) 
:= 
\circlearrowleft <[[a_1, a_2]], a_3>
\]
for any $a_1$, $a_2$ and $a_3$ in  $\Gamma(L)$, defines a $3$-tensor called the \emph{Courant tensor}\index{Courant!tensor} of $L$.\\ 
For a  sub-bundle $D$, the condition of involutivity  \textbf{(DS)} is equivalent to $\mathbf{T}_D = 0$. 

\subsection{Induced Dirac Structures on the Cotangent Bundle}
\label{___InducedDiracStructuresOnTheCotangentBundle}

An important class of Dirac structures is the induced Dirac structure on the cotangent bundle $T^\ast Q$ of a finite dimensional manifold $Q$ (cf. \cite{LeOh11}, 2.2). \\
Let $p_{T^\ast Q} : T^\ast Q \to Q$ be the cotangent bundle and $\Omega^\flat : T T^\ast Q \to T^\ast T^\ast Q$ the morphism associated to the Liouville $2$-form $\Omega$ (canonical  symplectic structure on $T^\ast Q$).\\
Given a constant-dimensional distribution $\Delta \subset TQ$, one defines the lifted distribution
\[
\Delta_{T^\ast Q} 
:= 
\left( T p_{T^\ast Q} \right) ^{-1} (\Delta) \subset T T^\ast Q
\] 
%The annihilator $\Delta_{T^\ast Q}^{\operatorname{o}}$ is also given by $\Delta_{T^\ast Q}^{\operatorname{o}} 
%= 
%p_Q^\ast \left( \Delta \right) $.\\
Then the sub-bundle $\mathcal{D}_\Delta$ of $T T^\ast Q \oplus T^\ast T^\ast Q$ defined by
\[
\mathcal{D}_\Delta 
=
\{
(u,\alpha) \in T T^\ast Q \oplus T^\ast T^\ast Q: \;
u \in \Delta_{T^\ast Q},\; \alpha - \Omega^\flat (u) 
\in \Delta_{T^\ast Q}^{\operatorname{o}}
\}
\]
is a Dirac structure on $T^\ast Q$. 

\subsection{Integrable Dirac Structures as Lie Algebroids}
\label{___IntegrableDiracStructuresAsLieAlgebroids}

A \emph{Lie algebroid}\index{Lie algebroid} is a finite dimensional smooth vector bundle $\pi : E \to M$  over $M$ with a vector bundle homomorphism $ \rho : E \to TM$, called the \emph{anchor}\index{anchor}, and a Lie algebra bracket $[.,.]_E : \Gamma(E) \times \Gamma(E) \to \Gamma(E)$ satisfying:
\begin{description}
\item[\textbf{(LA1)}]
$\rho$ is a Lie algebra homomorphism;
\item[\textbf{(LA2)}]
for all $f \in C^\infty(M)$ and all $X$ and $Y$ in $\Gamma(E)$,
\[
[X,fY]_E = f[X,Y]_E + \left( L_{\rho(X)}f \right) Y.
\]
\end{description}
For a Dirac structure $L$ on $M$, the vector bundle $\pi_L : L \to M$ inherits a Lie algebroid structure with bracket on $\Gamma(L)$ given by the restriction of the Courant bracket and  anchor  given by the restriction of the projection $T^{\mathfrak{p}}(M) \to TM$ to $L$ (cf. \cite{JoRa12}, 2.2 and \cite{Bur13}, 4.1).

\section{Partial Almost Dirac Structures and Implicit Hamiltonians} 

\label{__PartialAlmostDiracStructuresAndImplicitHamiltonians}

\subsection{Pontryagin Bundle}\label{___PontyaginBundle}

In all this section, $M$ is a fixed  convenient manifold  modelled on a convenient space $\mathbb{M}$ and we denote by $p_{TM}:TM\to M$ its kinematic bundle (\cite{KrMi97}, 28.12)  and $p_{T^\prime M}:T^\prime M\to M$ its associated convenient dual bundle (\cite{KrMi97}, 33.1). \\

A \emph{weak subbundle}\index{weak!subbundle} $p_{T^\flat M}:T^\flat M \to M$ of $T^\prime M$  is a linear subbundle $T^\flat M$ of $T^\prime M$ modelled on a convenient space $\mathbb{M}^\flat$ such that the inclusion $T^\flat M \to T^\prime M$ is a an injective convenient bundle morphism (cf. \cite{CaPe23}, Chap.~3). 
By the way, we can assume that $\mathbb{M}^\flat $ is a subspace of $\mathbb{M}^\prime$ and the inclusion $\mathbb{M}^\flat \to \mathbb{M}^\prime$ is bounded.  
Note that $p_{T^\flat M}$ is nothing but the restriction of $p_{T^\prime M}$ to $T^\flat M$.\\  
When $T^\flat M$ is fixed,  if $\Delta$ is a closed subbundle of $TM$, we consider
\[
\Delta^0_x
=
\{\alpha\in T_x^\flat M:\; \forall u\in \Delta_x,\alpha(u)=0
\}.
\] 
Then $\Delta^0=\displaystyle\bigcup _{x\in M} \Delta_x^0$\index{$\Delta^0$} is a closed subbundle of $T^\flat M$.\\

A \emph{Pontryagin bundle over $M$}\index{Pontryagin bundle} is a bundle of type
\begin{center}  
$\pi^\mathfrak{p}_M:T^\mathfrak{p}M=TM\oplus T^\flat M \to M$\index{TPM@$T^\mathfrak{p}M$ (Pontryaguin bundle)}.
\end{center}

Let $(U,\phi)$ be a chart  such that $U$ is simply connected.  Then, as classically, we may assume that $U$ is a $c^\infty$ open set in $M$;  we have the following trivializations:
 $TM_{|U}\cong U\times\mathbb{M}$ and $T^\prime M_{| U}\cong U\times \mathbb{M}^\prime$;

  $T^\flat M_{| U}\cong U\times\mathbb{M}^\flat$ and $T^\mathfrak{p}M\cong U\times \mathbb{M}\times\mathbb{M}^\flat$

\noindent and associated local coordinates:

$(\mathsf{x},\mathsf{v})\in U\times\mathbb{M}$ and $(\mathsf{x},\mathsf{p})\in U\times \mathbb{M}^\prime$;

 $(\mathsf{x},\mathsf{p})\in U\times \mathbb{M}^\flat$ and $(\mathsf{x},\mathsf{v},\mathsf{p})\in U\times \mathbb{M}\times\mathbb{M}^\flat$

\noindent which gives rises to the following local expressions of the projections:

$p_{TM}\cong (\mathsf{x},\mathsf{v})\mapsto \mathsf{x} $;

$p_{T^\prime M}\cong (\mathsf{x},\mathsf{p})\mapsto \mathsf{x} $;

$\pr :T^\mathfrak{p}M\to TM 
 \cong  
 (\mathsf{x},\mathsf{v},\mathsf{p})\mapsto (\mathsf{x},\mathsf{v})$;

$\pr^\flat :T^\mathfrak{p}M\to T^\flat M
\cong 
(\mathsf{x},\mathsf{v},\mathsf{p})\mapsto (\mathsf{x},\mathsf{p})$.
 
% Consider the canonical weak symplectic form $\Omega$ on  $T^\prime M$. Then $\Omega^\flat: T T^\prime M\to T^\primeT^\prime M$ is an injective morphism whose range $T^\flat T^\prime M$ is a weak subbundle of $T^\prime T^\prime M$.   Thus $(\Omega^\flat )^{-1}$ is an isomorphism from      $T^\flat T^\prime M$ to $TT^\prime M$. In local coordinates we have:
% \begin{description}
%\Omega_{(\mathsf{x},\mathsf{p})(
% For the tangent  bundle $pr_{TT^\prime M}:TT^\prime M\to T^\prime M$ and $pr_{T^\prime  TM}:T^\prime T M:\to T^\prime M$ in local coordinates are then:
% \begin{description}
% \item $pr_{TT^\prime M}:TT^\prime M\to T^\prime M$:  $\; (\mathsf{x},\mathsf{p},\dot{\mathsf{x}},\dot{\mathsf{p}})\mapsto  (\mathsf{x},\mathsf{p})$
%\item  $pr_{T^\prime TM}:TT^\prime M\to T^\prime M$: $\; (\mathsf{x},\mathsf{v},\dot{\mathsf{x}},\dot{\mathsf{p}})\mapsto (\mathsf{x},\mathsf{v})$
%\end{description}

\subsection{Some Canonical Maps}\label{___CanonicalMaps}  
We consider the following projections:\\
$\pi_{T^\prime TM}: T^\prime T M\to TM$ and $Tp_{T^\prime M}:TT^\prime M\to TM$. \\
In local coordinates, we have:

$\pi_{T^\prime TM}\cong (\mathsf{x},\mathsf{p},\dot{\mathsf{x}},\dot{\mathsf{p}})\mapsto  (\mathsf{x},\dot{\mathsf{x}})$;

$Tp_{T^\prime M}\cong (\mathsf{x},\dot{\mathsf{x}},\mathsf{p},\dot{\mathsf{p}})\mapsto (\mathsf{x},\dot{\mathsf{x}})$.\\
As in finite dimension, we have a unique smooth diffeomorphism  map 
\[
\kappa_M: TT^\prime M\to T^\prime T M
\]
which satisfies the following relation:
$\pi_{T^\prime TM}\circ \kappa_M
=Tp_{T^\prime M}$.\\ 
In local coordinates, $\kappa_M$  is given by:

$\kappa_M \cong (\mathsf{x},\mathsf{p},\dot{\mathsf{x}},\dot{\mathsf{p}})
\mapsto  (\mathsf{x},\dot{\mathsf{x}},\mathsf{p},\dot{\mathsf{p}})$\\
(see for instance \cite{CMR04}).\\

Let $\Omega$  be the weak canonical symplectic form on $T^\prime M$.\\
Then $\Omega^\flat: T T^\prime M \to T^\prime T^\prime M$ is an injective convenient morphism whose range is a weak subbundle $T^\flat T^\prime M$ of $T^\prime T^\prime M$ which is  isomorphic to $T T^\prime M$ (cf. [KrMi97], 48).  Since $T^\flat T^\prime M$  is  isomorphic to $T T^\prime M$, over a chart domain $U$, we have

$T^\flat T^\prime M\cong (U\times\mathbb{M}^\prime)\times (\mathbb{M}^\prime\times\mathbb{M})$\\
where $\mathbb{M}^\prime\times\mathbb{M}$ is considered as a subspace of  $\mathbb{M}^\prime\times\mathbb{M}^{\prime\prime}$.\\

Denote by $(\mathsf{x},\mathsf{p},\mathsf{a},\mathsf{w})$ the local coordinates  on $T^\prime T^\prime M \cong (U\times\mathbb{M})\times (\mathbb{M}^\prime\times\mathbb{M}^{\prime\prime})$. Thus, in local coordinates, $\Omega^\flat \cong TT^\prime M \to T^\prime T^\prime M$ is given by:
	
$\Omega^\flat \cong (\mathsf{x},\mathsf{p},\dot{\mathsf{x}},\dot{\mathsf{p}})\mapsto (\mathsf{x},\mathsf{p},\mathsf{a}=-\dot{\mathsf{p}},\mathsf{w}=\dot{\mathsf{x}})$

\noindent where, in the range, the term $\dot{\mathsf{x}}$  is considered as an element of $\mathbb{M}\subset \mathbb{M}^{\prime\prime}$.\\
By the way, we obtain a smooth diffeomorphism 
\[
\gamma_M=\Omega^\flat\circ (\kappa_M)^{-1}:  T^\prime TM \to T^\flat T^\prime M.
\]
In local coordinates, $\gamma_M$ is given by:

$\gamma_M:(\mathsf{x},\dot{\mathsf{x}},\mathsf{p},\dot{\mathsf{p}})\mapsto(\mathsf{x},\mathsf{p},-\dot{\mathsf{p}},\dot{\mathsf{x}})$.\\
In particular, $(\mathsf{x},\mathsf{p},-\dot{\mathsf{p}},\dot{\mathsf{x}})$ is a \emph{local coordinates system} on $T^\flat T^\prime M$.

\begin{remark}
\label{R_CanonicalMapFinite} 
In  this convenient context, we have formally the same type of diffeomorphism which generalizes        Tulczyjew's isomorphism in finite dimension (cf. \cite{Tul77}). Note that it is not in general a    diffeomorphism $T^\prime TM \to T^\prime  T^\prime M$ but onto a weak subbundle of $T^\prime  T^\prime M$. However, if $\mathbb{M}$ is reflexive, then $\Omega^\flat $ is surjective and we recover the result as in finite dimension.
\end{remark}

\subsection{Partial Almost Dirac Structures}
\label{___PartialAlmostDiracStructures}
\emph{In this section, the bundle $T^\flat M$ is fixed and so the Pontryagin bundle is $T^\mathfrak{p}M=TM\oplus T^\flat M$.}\\

If $<.,.>$ denotes the canonical pairing between $T^\prime M$ and $TM$, its restriction to $T^\flat M\times TM$ is a bilinear form which is bounded since the injection of $T^\flat M$ in $T^\prime M$ is bounded. If there is no ambiguity, this restriction will be also denoted by $<.,.>$. \\
 
The Pontryagin bundle is  provided with a pairing  $<<.,.>>$ defined as follows:\\
for any $x \in M$, any pair $(X_x,Y_x)$ of $T_xM$ and any pair $(\alpha_x,\beta_x)$ of $T_x^\flat M$,
\begin{eqnarray}
\label{eq_pairringTM}
%\label{eq_DirectSymmetricBilinearFormAlmostPartialDiracStructures}
<<(X_x,\alpha_x),(Y_x,\beta_x)>>
= 
\beta_x(X_x) + \alpha_x(Y_x)
\end{eqnarray}

Let $D$ be a closed subbundle of  $T^\mathfrak{p}M$. We set\index{DPerp@$D^\perp$}
\[
D^\perp=
\left\lbrace
(u,\alpha) \in T_x^\mathfrak{p}M:\; \forall x\in M, \forall (v,\beta)\in D_x,\;
 <<(u,\alpha),(v,\beta)>>=0
\right\rbrace
. 
\]
\emph{A partial almost Dirac structure}\index{partial!almost Dirac structure} on $M$ is a closed subbundle $D$ of $T^\mathfrak{p}M$ such that  $D^\perp=D$.\\

For each open set $U$ in $M$, we can provide the set of local sections $\Gamma \left( T^\mathfrak{p}M_{| U} \right)$ with a \emph{Courant bracket}\index{Courant!bracket}\index{bracket!Courant} defined by (cf. \cite{PeCa24}, 3.3)
\begin{equation}
\label{eq_CourantTM}
[(X,\alpha),(Y,\beta)]_C=\left([X,Y],L_X\beta-L_Y\alpha-\displaystyle\frac{1}{2}d(i_Y\alpha -i_X\beta)\right)
\end{equation}
for all sections $(X,\alpha), (Y,\beta)\in \Gamma \left( T^{\mathfrak{p}}M_{|U} \right) $, where $[\;,\;]$ is the usual Lie bracket of vector fields.\\
Note that 
$[(X,\alpha),(Y,\beta)]_C$ always belongs to $\Gamma \left( T^{\mathfrak{p}}M_{|U} \right) $ (cf. \cite{CaPe23}, 3.2).\\

\emph{A partial almost Dirac structure} on $M$ is called \emph{involutive}\index{involutive} if the  the restriction of $[.,.]_C$ to $\Gamma(D_{| U})$ takes values in $\Gamma(D_{| U}) $ for any open set $U$ in $M$. A partial almost involutive Dirac structure on $M$ will be simply called a \emph{partial Dirac Structure} on $M$. In the particular case where $T^\flat M=T^\prime M$, a partial (almost) Dirac structure is simply called a \emph{(almost) Dirac structure}.

For examples of partial (almost) Dirac structure on $M$  in the convenient setting, the reader can see \cite{PeCa24}. \\

Given a partial almost Dirac structure, then $\pr(D)$ is a distribution on $M$ called the \emph{characteristic distribution}\index{characteristic distribution}. A partial almost Dirac structure is called \emph{integrable}\index{integrable} if, there exists an immersed submanifold $L$ contained in $M$ such that $T_x L=\pr(D_x)$ for any $x\in L$. Such a manifold is called a \emph{characteristic leaf}\index{characteristic leaf}. In finite dimension, if $D$ is a Dirac structure, then the distribution is always integrable. Unfortunately, in infinite dimension, this result is no more true in general without more assumptions, even in the Banach setting. However, there exist partial almost Dirac structures which are integrable but not involutive.\\

From \cite{PeCa24}, Theorem~4.24, we  have the following general results:

\begin{theorem}
\label{T_GeometricStructuresOnLeaves} 
Let $D$ be a partial almost  integrable Dirac structure on a convenient manifold $M$. Assume that $\ker \pr_{| D}$  and $\ker \pr^\flat_{| D}$ are supplemented  in each fibre. Consider a  leaf $L$ of the characteristic foliation. Then we have:
\begin{enumerate}
\item[1.] 
$T^\flat L:=\displaystyle\bigcup_{x\in L}\pr^\flat(D_x)$ is a closed weak convenient subbundle of $T^\prime M_{| L}$. 
\item[2.] 
There exists a  skew symmetric convenient bundle morphism $P_L:TL\to T^\prime L$  whose kernel is $(D\cap TM)_{| L}\subset TL$ and whose range is  the  weak subbundle $T^\flat L\cap T^\prime L$. In particular, $\Omega_L(u,v)=<P_L(v), u>$ is a  $2$-form on $L$ whose kernel is  $(D\cap TM)_{| L}$ and  we have $\Omega^\flat =P_L$. Moreover, if $D$ is a partial Dirac structure, $\Omega_L$ is closed and so is strong pre-symplectic.
\item[3.] 
There exists a  skew symmetric convenient bundle morphism $P_L^\flat$ from $T^\flat L$ to its dual bundle $T ^{\flat\prime} L$  whose kernel is $(D\cap T^\flat M)_{| L}\subset T^\flat L$ and whose range is  the  weak subbundle $T^{\flat\prime} L\cap TL$ of $T^{\flat\prime} L$. Moreover, if $D$ is a partial Dirac structure,  $P^\flat_L: T^\flat L \to (T^{\flat\prime} L\cap TL)\subset TL$ is a partial Poisson structure on $L$.
\item[4.]  
$P_L$ (resp. $P^\flat_L$) induces a convenient isomorphism $\widetilde{P}_L$ (resp. $\widehat{P}^\flat_L$) from $T^\flat L\cap T^\prime L$ to $TL\cap T^{\flat\prime}$ (resp. 
$TL\cap T^{\flat\prime}$ to $T^\flat L\cap T^\prime L$) and $P_L^{-1}=P^\flat_L$.\\
Moreover, $D_L\cap TM_{L}=\{0\}$ if and only if $D_L\cap T^\flat M=\{0\}$ and then $P_L$ is an isomorphism from $TL$ to $T^\flat L$ and, if $D$ is a partial Dirac structure,   $\Omega_L$ is weak symplectic  form on $L$.
\end{enumerate}
\end{theorem}

\begin{remark}
\label{R_AssumptionsSatisfied}
In finite dimension, all the previous  assumptions are satisfied and so Theorem~\ref{T_GeometricStructuresOnLeaves} is a generalization of classical results on Dirac manifolds (cf. \cite{Cou90}, \cite{YoMa06I}, \cite{Bur13} and \cite{Marc16}). \\
In the Hilbert setting, the assumptions of Theorem~\ref{T_GeometricStructuresOnLeaves} are always satisfied.\\
In the Banach setting,  Theorem~\ref{T_GeometricStructuresOnLeaves} is always true.
\end{remark}

\subsection{Induced Partial Almost Dirac Structures}
\label{___InducedPartialAlmostDiracStructures}
\emph{In this subsection, we will define a partial almost Dirac structure on the cotangent bundle which will play an important r\^ole for implicit Lagrangians as in finite dimension (cf.\cite{YJM10} for instance). We will give an adaptation of such a situation  to our context.}\\

Let  $\Delta_M$ be  a closed  subbundle of $TM$ which will be called a \emph{constraint distribution}\index{constraint distribution}\index{distribution!constraint}. On the cotangent bundle $T^\prime M$, we define a distribution $\Delta_{T^\prime M}$ by 
\[
\Delta_{T^\prime M}
=
(Tp_{T^\prime M})^{-1}(\Delta_M)\subset TT^\prime M
\]
 where $Tp_{T^\prime M}:TT^\prime M\to TM$ is the tangent map of $p_{T^\prime M}:T^\prime M\to M$. Considering the weak canonical symplectic form  $\Omega$ on  $T^\prime M$, we  associate to $\Delta_{T^\prime M}$ the set  $D_{\Delta_M}$  of $T^\mathfrak{p} T^\prime M= TT^\prime M\oplus T^\flat  T^\prime M$ by:
\[
(D_{\Delta_M})_{(x,p)}=
\left\lbrace
(u,\alpha)\in T^\mathfrak{p}T^\prime M:\;
 u\in (\Delta_{T^\prime M})_{(x,p)};\;
   \alpha- \Omega_{(x,p)}^\flat (u)\in (\Delta_{T^\prime M}^0)_{(x,p)} 
\right\rbrace .
\]

We set 
$D_{\Delta_M}:=
\displaystyle\bigcup_{(x,p)\in T^\prime M}(D_{\Delta_M})_{(x,p)}$.\\
As in    \cite{YoMa06I}, Theorem 2.3, we have:
\begin{proposition}
\label{L_CloseneesOfOmegaDelta}  
Let $\widetilde{\Omega}^\flat$ be  the restriction of $\Omega^\flat$ to $\Delta_{T^\prime M}$.  Then $D_{\Delta_M}$ is the graph of the morphism $\widetilde{\Omega}^\flat : \Delta_{T^\prime M}\to T^\flat M$ and  is a partial almost Dirac structure on $T^\prime M$
\end{proposition}
 
\begin{proof}
On the one hand, from Lemma 3.6 (2), 
$(D_{\Delta_M})_{(x,p)}$ is the the graph of $\widetilde{\Omega}^\flat$ and is a partial linear Dirac structure on $T_{(x,p)} (T^\prime M$). On the other hand, since  $\Delta_{T^\prime M}  $
is a closed  subbundle of $TT^\prime M$  and  since $\Omega^\flat$ is an isomorphism,  so, $\Omega^\flat \left( D_{\Delta_{T^\prime M}} \right)
=\widetilde{\Omega}^\flat \left( D_{\Delta_{T^\prime M}} \right) $ is a closed subbundle of $T^\flat T^\prime M$. Therefore from the Definition 3.16, in \cite{PeCa24}, $D_{\Delta_M}$ is a partial almost Dirac structure on $T^\prime M$.
\end{proof}

\begin{remark}
\label{R_DOmegaM}${}$
\begin{enumerate}
\item[1.]  
 In local coordinates relative to  a chart $(U,\phi)$, as in finite dimension,  as in  \cite{YoMa06I}),  we have 
\begin{equation}
\label{Eq_DOmegaM}
\begin{split}
D_{\Delta_M} \cong
\{
\left( (\mathsf{x},\mathsf{p},\dot{\mathsf{x}},\dot{\mathsf{p}}),(\mathsf{x},\mathsf{p},\mathsf{a},\mathsf{w})\right):\hspace{4.8cm}\\
\mathsf{x}\in\mathbb{U}=\phi(U),\; \mathsf{p}\in \mathbb{M}^\prime ,\;\dot{\mathsf{x}}\in \mathbb{E}, \mathsf{w}=\dot{\mathsf{x}}, \; \mathsf{a}+\dot{\mathsf{p}}\in \mathbb{E}^0
\}
\end{split}
\end{equation}

 \item[2.]
According to \cite{PeCa24}, Proposition 3.27, since $\Omega$ is closed, $D_{\Delta_M} $ will be a partial Dirac structure  if and only if $\Delta_M$ is involutive.
\end{enumerate}
\end{remark}

\subsection{Implicit Hamiltonian Systems}
\label{__ImplicitHamiltonianSystems}
The concept  of implicit Hamiltonian system relative to a Dirac structure  in finite  dimension is now classical. In this section, we give an adaptation to our context  of partial almost  Dirac structure on a convenient manifold  and  we essentially refer to  A. J. Van der Schaft's papers   \cite{VdS94}, \cite{VdS95} and  \cite{VdS98}.

\begin{definition}\label{D_ImplicitHamiltoniansystem} Consider a partial almost Dirac structure $D$ on a convenient manifold $M$ and relative to a Pontryaguin bundle  $T^\mathfrak{p} M=TM\oplus T^\flat M$. Let $H:M\to \mathbb{R}$ be a smooth function called a \emph{Hamiltonian}\index{Hamiltonian}. 
\begin{enumerate}
\item 
The \emph{implicit Hamiltonian system}\index{implicit Hanimtonian system} corresponding to $(M,D,H)$
 is the dynamical system characterized by the differential inclusion
 \begin{equation}
 \label{Eq_IH}
\left( \dot{x}(t), d_{x(t)}H \right) \in D_{x(t)} 
\end{equation}
\item 
A smooth curve $\gamma:]-\varepsilon,\varepsilon[\to M$ such that $ \left( \dot{\gamma}(t),d_{\gamma(t)}H \right) $ belongs to $D_{\gamma(t)}$ is called an \emph{integral curve}\index{integral curve} of the implicit Hamiltonian system (\ref{Eq_IH}) through $x=\gamma(0)$. We also say that the implicit Hamiltonian system associated to $(M,F,H)$ is \emph{integrable} at $x$.
\item 
Let $S$ be a weak   immersed  submanifold of $M$ and assume that $H_{| S}$ is a smooth function on $S$.
\begin{enumerate} 
\item 
The induced implicit  Hamiltonian system  $ \left( M,D,H_{| S} \right) $ on $S$ is the dynamical system:
\begin{equation}
\label{eq_InducedMDH}
\textrm{ find a smooth curve } \gamma:I \to S\;:\;
   \left( \dot{\gamma}(t), d_{\gamma(t)}H_{| S} \right) \in D_{\gamma(t)} 
\end{equation}
\item 
The induced implicit Hamiltonian system is called \emph{integrable}\index{integrable} on $S$ if there exists a vector field  $X$ on $S$ such that  $X$ has a local flow and any integral curve of $X$ is an integral curve of the induced  implicit Hamiltonian system $(M,H_{| S},D)$. If $S=M$, we say that the Hamiltonian is \emph{completely integrable}.
\end{enumerate}
\end{enumerate}
\end{definition}
 
\begin{remark}
\label{R_InducedImplicitHamilton} 
Note that, according to Definition~\ref{D_ImplicitHamiltoniansystem}, 3. (a), since the inclusion $\iota$ of $S$ in $M$ is a weak immersion, it is a smooth map from $S$ to $M$. By the chain rule, $t\mapsto 
  \iota\circ \gamma(t)$ is a smooth curve in $M$ and so the inclusion in (\ref{eq_InducedMDH}) makes sense. Note that in general the converse is not true in the convenient setting:  if $t\mapsto \gamma(t)$ is a smooth curve in $M$ such that $\gamma(t)$ belongs to $S$, in general, $t\mapsto \gamma(t)$ is not a smooth curve in $S$. For the same reason, the assumption that $H_{|S}$ is smooth is not true in general  for any Hamiltonian $H$ on $M$.  In particular, any smooth curve $\gamma$ from an  interval $I$ to $M$ such that $\gamma(I)\subset S$ and which satisfies  $ (\dot{\gamma}(t), d_{\gamma(t)}H_{| S})\in D_{\gamma(t)}$ could not satisfy the condition in Point 2 (a) since $\gamma$ could be not a smooth curve in $S$.  
\end{remark}

Recall that the function $H$ defined  on an open set $U$ in $M$ is admissible (for $D$) (cf.\cite{PeCa24}, 4.2) if there exists a smooth vector field $X$ on $U$ such that $(X,df)$ is a section of $D$ over $U$. Therefore, if a Hamiltonian $H$ is admissible on an open set $U$, the implicit Hamiltonian system $(M,D,H)$ will be integrable on $U$ if  $X$ has a local flow on $U$. This situation always occurs in the Banach setting.

\begin{remark}
\label{R_EnergyConservation} 
Let $(M,D,H)$ be an implicit Hamiltonian system. Since, for any $(u,\alpha)\in D_x$, we have $<\alpha, u>=0$, it follows that  for any integral curve $\gamma$, we have
\[
\displaystyle\frac{dH}{dt}=<dH,\dot{\gamma}>=0.
\]
Therefore, any implicit Hamiltonian system has the "energy-conservation" property (cf. \cite{VdS98}).
\end{remark}

\begin{remark}
\label{R_MDHintegral}
If the implicit Hamiltonian system associated to $(M,D,H)$ is integrable  at $x$, the germ of integral curves through $x$ is not necessary unique. However, from Remark \ref{R_EnergyConservation}, $d_xH$ must belong to $\pr^\flat(D_x)$. But  if $D$ is integrable and if ${L}$ is a leaf through $x$, then if $ (X, dH)$ is a section of $D$ over  a connected open set $U$, then $X$ must be tangent to $L$ on $U\cap L$.  As  $U\cap L$ is an open set of $L$,  any integral curve $\gamma:]-\varepsilon,\varepsilon[\to U$  through $x$ must satisfy $\dot{\gamma}(t)\in \pr(D_{(\gamma(t))})$ and so  $\gamma(]-\varepsilon,\varepsilon[)$ must be contained in ${L}$, which implies that the germ of an integral curve through $x$ is unique.
\end{remark} 

As in finite dimension\footnote{See for instance  \cite{VdS98}, \cite{BlRa04} and all references inside.} 
and using the same notations as in \cite{VdS98}, to an implicit Hamiltonian system $(M,D,H)$, we can associate\\
\noindent the distributions 
 
 ${G_0}:=D\cap TM$\index{G0@$G_0$};

 $G_1=\pr(D)$\index{G1@$G_1$}

\noindent and co-distributions:

 $P_0=D\cap T^\flat M$\index{P0@$P_0$};
 
 $P_1=\pr^\flat(D)$\index{P1@$P_1$}\\
where  $\pr:T^\mathfrak{p}M\to TM$ and 
$\pr^\flat: T^\mathfrak{p}M\to T^\flat M$ are the canonical projections.\\ 

We have the following relations between these distributions  and co-distributions:
\[
G_0\subset G_1, \; P_0\subset P_1, \; G_0=\ker p^\flat_{| D} \text{~~and  ~~} P_0=\ker p_{| D}.
\]

Thus an implicit Hamiltonian system $(M,D,H)$  can be written in the following way in local coordinates\footnote{In local coordinates, the function $H$ will be written $\mathsf{H}$ and, for short, we write $H\cong \mathsf{H}$ according to the notations at the beginning of $\S$~\ref{__ImplicitHamiltonianSystems}.}:

\begin{equation}
\label{eq_ImplicitPb}
 \dot{\mathsf{x}} =\displaystyle\frac{\partial \mathsf{H}}{\partial  \mathsf{p}} \in G_1,\; \;\;\dot{\mathsf{p}}=\displaystyle\frac{\partial \mathsf{H}}{\partial  \mathsf{x}}\in P_1.
\end{equation}

\begin{remark}
\label{R_ExistenceSol} 
The previous problem  is linked with the general problem of existence of solutions of an ordinary differential equation in the convenient setting. In general, such an integral curve need not exist. However, in the Banach setting, this problem has in general some solution.
\end{remark}

If  $H$ is an admissible Hamiltonian on $U$ and  if $(X, dH)$ is a section of $D$ over $U$, then any integral curve $\gamma$ of $X$ and $\alpha(t)=d_{\gamma(t)}H$ is a 
solution of (\ref{eq_ImplicitPb}).\\

In finite dimension, when $D$ is involutive, all these distributions and co-distributions are integrable.  Of course, this result is not true in general in our context. Theorem \ref{T_GeometricStructuresOnLeaves}  gives sufficient conditions under which an analog result is true  and as we have seen (Remark \ref{R_AssumptionsSatisfied}), these results are always true in the Hilbert setting.

From \cite{CaPe23}, Theorem~1.9,  as in \cite{VdS98}, Theorem~6,   we have the following results:  

\begin{theorem}
\label{T_DiracPoisson} 
Let ${D}$ be a partial almost  Dirac structure on $M$. 
\begin{enumerate}
\item[(1)]  
Assume that $P_1$ is a (closed) subbundle of $T^\flat M$. 
For any $x\in M$, there exists  a bounded skew symmetric linear map  
\[
\begin{array}{cccc}
J_x:& (P_1)_x &\to& (G_1)_x \cap (G_1)_x^\prime \\
	& \alpha  &\mapsto & \alpha_{| (P_1)_x}
\end{array}.
\] 
The  kernel of $J_x$ is  $(P_0)_x$ and its  range is isomorphic to $(P_1)_x/(P_0)_x$. In particular $D_x$ is the graph of $J_x$.  Moreover,  for any fixed  $x_0\in M$ there exists an open neighbourhood $U$ of $x_0$ such that:
$$D_{| U}=\{(X,\alpha)\in \Gamma(T^\mathfrak{p}M_{| U})\;/ \;X_x-J_x(\alpha_x)\in(G_0)_x, \;x\in U,\; \alpha\in\Gamma(( P_1)_{| U})\}$$
Moreover,  $J_x$ gives rises to a bundle morphism $J$ from $P_1$ to $P_1^\prime$.

\item[(2)]  
Assume that $G_1$ is a subbundle of $TM$. 
There exists a surjective  linear   bounded skew symmetric map
\[
\begin{array}{cccc}
Q_x:& (G_1)_x &\to& (G_1)_x^\prime \cap (P_1)_x\\
	&  u   	  &\mapsto&  u_{| (G_1)_x}
\end{array}
\]
where $u\in T_xM$ is considered as an element of $T^{\prime\prime} M$. The kernel of $Q_x$   is  $(G_0)_x $, its  range isomorphic is to   $ (G_1)_x/(G_0)_x$,  and  $\mathbb{D}_x$ is the graph of $Q_x$.
Moreover,  for any $x_0\in M$, there exists an open neighbourhood $U$ of $x_0$ such that:
\[
D_{| U}=\{(X,\alpha)\in \Gamma(T^\mathfrak{p}M_{|U}):\;
\alpha_x-Q_x(X_x)\in (P_0)_x,\; x\in U, \;X\in \Gamma((G_1)_{| U})\}
\]
\end{enumerate}
Moreover, $Q_x$ gives rise to a bundle morphism $Q$ from $G_1$ to $G_1^\prime$.
\end{theorem}

\begin{example}
\label{Ex_2Form}
Consider a $2$-form $\Omega$ on a convenient manifold $M$ and assume that  $\Omega^\flat(TM)$ is a closed subbundle $T^\flat M$ of $T^\prime M$. 
Then the  graph 
\[
D_\Omega
=
\left\lbrace
(u,\Omega^\flat(u))\in T_x M\times T^\flat M, u\in T_xM,\; x\in M
\right\rbrace
\]
is a partial almost  Dirac structure (cf. \cite{PeCa24}, Example~4.16).  According to the previous notations, we have $G_1=TM$,  $G_0=\ker\Omega^\flat$, $P_1=\Omega(TM)$ and $P_0=\{0\}$. So the assumptions of Theorem \ref{T_DiracPoisson} are satisfied; in particular, we have  $Q=\Omega^\flat$ .\\
A Hamiltonian $H$ on an
open set $U$ is \emph{admissible}\index{admissible Hamiltonian} if there exists a vector field $X$ on $U$ such that $\Omega^\flat(X)=df$.\\
If $M$ is a Banach manifold and  $H$ is admissible over an open set $U$, then the implicit Hamiltonian system $ \left( M,D_\Omega,H \right) $ is integrable on $U$. In particular, if  $\Omega$ is strong pre-symplectic, for any function $H$ such that $dH$ induces a section of $T^\flat M$, then the implicit Hamiltonian system $ \left( M,D_\Omega, H \right) $ is  completely integrable.
\end{example}

\begin{example}
\label{Ex_PartialPoisson} 
Consider a partial almost Poisson anchor ${P}:T^\flat M\to M$ (cf. \cite{CaPe23},  Example~3.18). Under some assumptions, the graph of $P$ is a partial almost Dirac structure $D_P$. In this case, we have $G_1=P(T^\flat M)$, $G_0=\{0\}$,  $P_1=T^\flat M$ and $P_0=\ker P$. So the assumptions of Theorem \ref{T_DiracPoisson} are satisfied. In particular, $J(\alpha)=P(\alpha)_{| T^\flat M}$ where $P(\alpha)$ is considered as an element of $T^{\prime\prime} M$.\\
A Hamiltonian $H$ defined on an open set $U$ is  admissible if $dH$ induces a section of $T^\flat M$ over $U$. If $M$ is a Banach manifold and $H$ is admissible, the implicit Hamiltonian system $\left( M,D_P, H \right) $ is integrable on $U$. In particular, if $P$ defines  a Poisson structure, for any admissible Hamiltonian $H$, then $ (M, H,D)$ is completely integrable.
\end{example}

\begin{example}
\label{Ex_InducedDirac}
We consider the context of the induced Dirac structure $D_{\Delta_M}$ on $T^\prime M$ (cf. subsection~\ref{___InducedPartialAlmostDiracStructures}) where $\Delta_M$ is a closed subbundle of $TM$. Then $D_{\Delta_M}$ is a partial almost Dirac structure on $T^\prime M$ which is integrable if and only if $\Delta$ is involutive. Assume that $H$ is a smooth 
function on $T^\prime M$.\\
We have $G_1=\Delta_{T^\prime M}$, $G_0=\{0\}$, $P_1=\Omega^{\flat}(\Delta_{T^\prime M})\oplus \Delta_{T^\prime M}^0$ (cf.  Lemma 
\ref{L_CloseneesOfOmegaDelta}) and $P_0=\Delta_{T^\prime M}^0$. So the assumptions of Theorem \ref{T_DiracPoisson} are satisfied.  Since $\Omega^\flat $ is an injective morphism,  there 
exists a right inverse $R: \Omega^{\flat}(\Delta_{T^\prime M})\to G_1$ and if $p_1$ is the canonical projection $p_1: \Omega^{\flat}(\Delta_{T^\prime M})\oplus \Delta_{T^\prime M}^0\to 
\Omega^{\flat}(\Delta_{T^\prime M})$, we have $J_x=R_x\circ p_1$ and $J=R\circ P_1$ is a bundle morphism from $G_1$ to $P_1$.  We also have $Q=\Omega^\flat : G_1\to P_1$.\\
A Hamiltonian $H$  is admissible on an open  set $p_{T^\prime M }^{-1}(U)$ of $T^\prime M$ if and only there exists a vector field $X$ on $p_{T^\prime M }^{-1}(U)$ and a section $\nu$ of $\Delta_{T^\prime M}^0$ such that $dH=\Omega^\flat(X) +\nu$.\\
If $M$ is a Banach manifold and $H$ is admissible on $p_{T^\prime M }^{-1}(U)$, the implicit Hamiltonian system $(M,H,D)$ is integrable on  $p_{T^\prime M }^{-1}(U)$.
\end{example}

\section{Implicit Lagrangians}\label{__ImplicitLagrangians}

In this section, we adapt to the convenient setting context, the results in finite dimension of H. Yoshimura, H. L. Jacobs and J. E. Marsden given in  \cite{YJM10} and all their referenced papers in it.  Without other  assumptions, throughout this section,  $M$ will be a fixed  convenient connected  manifold and $\Delta $ a closed subbundle of $TM$. According to 
$\S$~\ref{___InducedPartialAlmostDiracStructures}, let $\Delta_{T^\prime M}=(Tp_{T^\prime M})^{-1}(\Delta_M)$ be the  associated  closed subbundle of $TT^\prime M$, $\Omega$ the weak canonical symplectic form on $T^\prime M$ and $T^\flat M$ the range of $\Omega^\flat$ which is a weak subbundle $T^\prime M$. 
We  associate to $\Delta_{T^\prime M}$ the induced partial almost Dirac structure  $D_{\Delta_M}$  on  $TT^\prime M$. We will use the local coordinates as defined in $\S$~\ref{___InducedPartialAlmostDiracStructures}.

\subsection{Legendre Transformation}\label{___LegendreTransformation}  

Any smooth map $\mathbf{L}:TM\to \mathbb{R}$ is  classically called  a \emph{Lagrangian}\index{Lagrangian}\footnote{In infinite dimension, for  some physical or mechanical problems, the Lagrangian is not defined on all $TM$, but  there exists a weak immersed  convenient  submanifold $M_0$ modelled on $\mathbb{M}_0$, which is in general  open or dense  in  $M$, so that the  Lagrangian is only defined on the pullback of $TM$ over  $M_0$ (cf. \cite{MaRa99}, p. 108). This situation is studied in \cite{SB-KZB17}. 
For the sake of simplicity, we will consider a Lagrangian $L$ defined on the whole tangent bundle $TM$.}
%Since $M_0$ will be an immersed submanifold of  $M$ for any $x\in M_0$ the inclusion  $i_0$ of $M_0$ in $M$ is continuous we will assume that $M_0\subset M$. Thus, for any $x\in M_0$, there exist a chart domain $U$ around $x$ in $M$ such that the connected component $U_0$ of $U\cap M_0$ which contains $x$ is also a chart domain form $M_0$ and a trivialization of $\Delta_{M}_0$ over $U_0$  which is the restriction of the trivialization of $\Delta_M$ over $U$
 and the associated \emph{Legendre transformation}\index{Legendre transformation} is  defined by:
\begin{equation}
\label{eq_LegendreTransform}
\mathbb{F}{\mathbf{L}}_x(v).w=\displaystyle\frac{d}{ds} \left( \mathbf{L}(v+sw) \right)_{| s=0}
\end{equation}
where $v$ and $w$ belong to $T_x M$ and where  $\mathbb{F}\mathbf{L}_x(v).w$ is the fibre  derivative of $\mathbf{L}$ along the vertical fibre of $TTM\to TM$  at point $(x,v)$, in the direction $w$.  The context of  \cite{AbMa87}, 3.5  can be adapted to the convenient setting, and by the way, , %$\mathbb{F}L_x$ define  a smooth map map from $T_x M \to T_x^\prime M$. 
 as the map $\mathbf{L}$ is smooth, it follows that the induced map 
$\mathbb{F}\mathbf{L}:TM\to T^\prime M$ is a smooth fibre preserving map. In local coordinates in $TM$ and $T^\prime M$, we have 
\[
\mathbb{F}\mathbf{L}\cong(\mathsf{x}, \mathsf{v})
\mapsto 
 \left( \mathsf{x},\mathsf{p}=\displaystyle\frac{\partial \mathsf{L}}{\partial \mathsf{v}} \right).
\]
Classically for mechanical systems in finite dimension, the Legendre transformation is called \emph{non degenerate} if  $\mathbb{F}\mathbf{L}$ is a local diffeomorphism. In the convenient setting, these conditions do   not occur in general. 
Therefore, in our context, we will say that $\mathbb{F}\mathbf{L}$ is \emph{non degenerate}\index{non degenerate!Legendre transformation} if, locally, $\mathbb{F}\mathbf{L}$ is injective,  for all $x$ in some open set $U$ in $M$; otherwise, we will say that  $\mathbb{F}L$ is \emph{degenerate}. 
Aiming at the notion of \emph{implicit Lagrangian}\index{implicit Lagrangian},  we will consider this general case where $\mathbb{F}\mathbf{L}$  could be degenerate.

\subsection{Dirac Differential Operator Associated to a Lagrangian}
\label{___DiracDifferentialOperatorAssociatedToALagrangian} 
Given a Lagrangian $\mathbf{L}$ over $TM$, its differential $d\mathbf{L}$ is a $1$-form on $TM$.\\
As in \cite{YJM10} or \cite{YoMa06I}, the \emph{Dirac differential operator}\index{DL@$\mathfrak{D}L$ (Dirac differential operator of $L$)}\index{Dirac differential operator} of $\mathbf{L}$ is defined by:
\[
\mathfrak{D}\mathbf{L}=\gamma_M\circ d\mathbf{L}
\]
where $\gamma_M$ is the diffeomorphism
\[
\gamma_M=\Omega^\flat\circ  \kappa_M^{-1}:T^\prime TM\to T^\flat T^\prime M
\]
defined in $\S$~\ref{___CanonicalMaps}. In local coordinates, we have: 
\begin{equation}
\label{Eq_DL}
\mathfrak{D}\mathbf{L}(x,v)\cong \left( \mathsf{x},\displaystyle\frac{\partial \mathsf{L}}{\partial \mathsf{v}}, -\displaystyle\frac{\partial \mathsf{L}}{\partial \mathsf{x}},\mathsf{v}\right)
\end{equation}

\subsection{Implicit Lagrangian Systems} \label{___ImplicitLagrangianSystems}

Under the context of $\S$~\ref{___InducedPartialAlmostDiracStructures}, as in  \cite{YJM10}, we introduce:

\begin{definition}
\label{D_ImplicitLagrangian}  
Consider a Lagrangian  $\mathbf{L}: TM\to \mathbb{R}$   and a closed subbundle $\Delta_M $ of $TM$.  
Let $\Delta_{T^\prime M}=(Tp_{T^\prime M})^{-1}(\Delta_M)$ be the  associated  closed subbundle of $TT^\prime M$,   $D_{\Delta_M}$ the induced partial almost Dirac structure  on  $T^\prime M$ and $\mathfrak{D}\mathbf{L}:TM\to T^\flat T^\prime M$ the Dirac differential operator of $\mathbf{L}$.\\
We set  $\mathbb{P}=\mathbb{F}\mathbf{L}(\Delta_M)\subset T^\prime M$. \\
An integral curve of an implicit Lagrangian system  $(M, \Delta_M,\mathfrak{D} \mathbf{L})$ is a curve $t\mapsto \gamma(t):= (x(t),v(t) ,p(t))$ from some interval $I\subset \mathbb{R}$ to $TM\oplus T^\prime M$ such that
\begin{equation}
\label{eq_ImplicitLagrangian}
 \left( x(t),p(t),\dot{x}(t),\dot{p}(t),  \mathfrak{D}\mathbf{L} \left( x(t),v(t) \right)  \right)
  \in D_{\Delta_M}(x(t),p(t))
\end{equation}
for all $t\in I$.
%An implicit Lagrangian system is a triple $(L,\Delta_M, X)$ where $X$ is a vector field defined on a neighbourhood  of $P$ and which satisfies the condition
\end{definition}

\begin{remark}
\label{R_ConsequencesDefinition}
In  local coordinates, as defined in $\S$~\ref{___InducedPartialAlmostDiracStructures}, for $z\in \mathbb{P}$, if   $z\cong (\mathsf{x},\mathsf{p})$, where $\mathsf{p}\cong  \displaystyle\frac{\partial \mathsf{L}}{\partial v}$ and $ v\cong(\mathsf{x},\mathsf{v})$, then according to (\ref{Eq_DOmegaM}), (\ref{Eq_DL}),  the condition  (\ref{eq_ImplicitLagrangian})  can be written:

\begin{equation}
\label{Eq_ImplicitLlocal}
\mathsf{p}=\displaystyle\frac{\partial \mathsf{L}}{\partial \mathsf{v}},\;\;\;\;\; \dot{\mathsf{x}}\in \mathbb{E}\,;\;\;\;\;\dot{\mathsf{x}}=\mathsf{v},\;\; \;\;\;\;\dot{\mathsf{p}}-\displaystyle\frac{\partial \mathsf{L}}{\partial\mathsf{x}} \in \mathbb{E}^0.
\end{equation}

Note that this integral curve   
$ t \mapsto (x(t),v(t),p(t))$ is  such that $t\mapsto(x(t),v(t))$ is contained in  $\Delta_M$ and satisfies  $\mathbb{F}\mathbf{L}(x(t),v(t))=(x(t),p(t))$ and so is contained in $\Delta_M\oplus \mathbb{P}$.  In local coordinates, if $\mathbf{L}(x,v)\cong\mathsf{L}(\mathsf{x},\mathsf{v})$, an integral curve 
\[
t\mapsto (x(t),v(t),p(t))\cong (\mathsf{x}(t),\mathsf{v}(t),\mathsf{p}(t))
\]
is solution of the following implicit  system 
\begin{equation}
\label{Eq_ImplicitLagrangianLoc}
\left\{\begin{aligned}
 &\begin{pmatrix}
\dot{\mathsf{x}}\\
\dot{\mathsf{p}}
\end{pmatrix}=
\begin{pmatrix}
	0&\operatorname{Id}_{\mathbb{M}}\\
	-\operatorname{Id}_{\mathbb{M}^\prime}& 0
\end{pmatrix}
\begin{pmatrix}	-\displaystyle\frac{\partial \mathsf{L}}{\partial \mathsf{x}}\\
			\mathsf{v}
\end{pmatrix}
+
\begin{pmatrix} 
0\\
\alpha(\mathsf{x})
\end{pmatrix} 
\textrm{ for some }  
\alpha:U\to \mathbb{E}^0\hfill{}\\
		&\mathsf{p}=\displaystyle\frac{\partial \mathsf{L}}{\partial \mathsf{v}}\hfill{}\\
		&<\beta(\mathsf{x}),\mathsf{v}>=0,\;\;\;\;\; \forall \beta:U\to \mathbb{E}^0\hfill{}\\
\end{aligned}\right.
\end{equation}
In particular, if $\Delta_M$ is finite $k$-co-dimensional for each  $x_0\in M$, in local coordinates,  there exists an open neighbourhood  $\mathbb{U}$ of $\mathsf{x}_0$ in $\mathbb{M}$ and smooth maps $\omega_i=U\to \mathbb{E}^0$, $i \in \{1,\dots, k\}$ such that 
$ \left( \omega_1(\mathsf{x}),\dots, \omega_k(\mathsf{x}) \right) $ is a basis of $\mathbb{E}^0$. Therefore, there exists "Lagrangian multipliers"\index{Lagrangian multipliers} 
$ \left( \lambda_1(\mathsf{x}),\dots,\lambda_k(\mathsf{x}) \right) $ such that the system (\ref{Eq_ImplicitLagrangianLoc}) can be written as in finite dimension (cf. \cite{YoMa06I})
\begin{equation}
\label{Eq_ImplicitLagrangianLock}
\left\{
\begin{aligned}
 &\begin{pmatrix}
\dot{\mathsf{x}}\\
\dot{\mathsf{p}}
\end{pmatrix}=
\begin{pmatrix}
	0&\operatorname{Id}_{\mathbb{M}}\\
	-\operatorname{Id}_{\mathbb{M}^\prime}& 0
\end{pmatrix}
\begin{pmatrix}	
	-\displaystyle\frac{\partial \mathsf{L}}{\partial \mathsf{x}}(\mathsf{x})\\
\mathsf{v}
\end{pmatrix}
+
\begin{pmatrix} 
		0\\
		\displaystyle\sum_{i=1}^k \lambda_i(\mathsf{x})\omega_i(\mathsf{x})
						\end{pmatrix} \hfill{}\\
		&\mathsf{p}=\displaystyle\frac{\partial \mathsf{L}}{\partial \mathsf{v}}\hfill{}\\
		&<\omega_i(\mathsf{x}),\mathsf{v}>=0,\;\;\;\;\; \forall i \in \{1\dots,k\}
		\hfill{}\\
\end{aligned}
\right.
\end{equation}

\end{remark}

\begin{remark}
\label{R_DeltaM=TM} 
When  $\Delta_M=TM$, the  differential system % (\ref{Eq_ImplicitLlocal}) or 
(\ref{Eq_ImplicitLagrangianLock}) is reduced to:
\begin{equation}\label{Eq_ImplicitLagrangianLocTM}
\left\{
\begin{aligned}
 \dot{\mathsf{x}}=\mathsf{v}\hfill{}\\
\dot{\mathsf{p}}=\displaystyle\frac{\partial \mathsf{L}}{\partial \mathsf{x}}\hfill\\
\mathsf{p}=\displaystyle\frac{\partial \mathsf{L}}{\partial \mathsf{v}}\hfill{}\\
\end{aligned}
\right.\hfill{}\\
\end{equation}
which is equivalent to the Euler-Lagrange conditions (cf. Corollary. \ref{C_EulerLagrangeImplicitL}).
\end{remark}

\begin{remark}
\label{R_ExistenceIntegralCurve}
As for an implicit Hamiltonian (cf. Remark \ref{R_ExistenceSol}), the existence of integral curves of  an implicit Lagrangian    $ \left( M,D_{\Delta_M},\mathfrak{D}\mathbf{L} \right)$ is linked with the general problem of existence of solutions of an ordinary differential equation in the convenient setting; in general, such integral curves need not exist. However, in the Banach setting, this problem has, in general, some solutions. Moreover, under some strong assumptions or particular cases like projective limits or direct limits, this problem has some solutions.
\end{remark}

\subsection{Non-degenerate Constraint Lagrangians}\label{___NonDegenerateConstraintLagrangians}
In this subsection, we assume that $\mathbf{L}$ can be  degenerate. However, the Legendre transformation $\mathbb{F}\mathbf{L}$ is defined and, again, we set  $\mathbb{P}=\mathbb{F}\mathbf{L}(\Delta_M)$. \\
% Now,  if $L$ is non degenerate, then $P=\mathbb{F}L(TM)$ is an immersed submanifold of $T^\prime M$ modeled on $\mathbb{M}$.%and curve  $t\mapsto (x(t),v(t),p(t))$ in $ D_{\Delta_M}$   is an integral curve of  the implicit Lagrangian $(X,\mathfrak{D}L)$ if an only if $(x(t),p(t))$ is a solution of the ordinary differential system of the two first equations in (\ref{Eq_ImplicitLagrangianLocTM}) which is contained  of  is contained in $P$ which an implicit differential system.
We will say that $\mathbf{L}$ is \emph{a non degenerate $\Delta_M$-constraint Lagrangian} if,  for any $x\in M$, there exists an open neighbourhood $U$ of $x$ such  that  the restriction of $\mathbb{F}\mathbf{L}$ to $(\Delta_M)_{| U}$ is a weak  immersion\footnote{f. \cite{CaPe23}, 3.7.}. 
In local coordinates, this  condition is equivalent to:\\ 
$\mathsf{x}\mapsto \displaystyle\frac{\partial^2 \mathsf{L}}{\partial \mathsf{v}^2}(\mathsf{x})$  is a  smooth field, on  some open set $\mathbb{U}$ in $ \mathbb{M}$,  of weak non-degenerate bilinear forms on  $\mathbb{E}$,
 
\noindent and then we have: 
\begin{equation}
\label{eq_localP}
\mathbb{P}\cong 
\left\{
 \left( \mathsf{x},\displaystyle\frac{\partial \mathsf{L}}{\partial \mathsf{v}}(\mathsf{x},\mathsf{v}) \right) , 
(\mathsf{x},\mathsf{v})\in \mathbb{U}\times\mathbb{E}
\right\}
\end{equation}
Indeed, in local coordinates, the tangent map $T\mathbb{F}\mathbf{L}$ from $TM{| U}\cong \mathbb{U}\times \mathbb{M}$ to $T^\prime M_{|U}\cong \mathbb{U}\times\mathbb{M}^\prime$ can be written as a matrix 
\begin{equation}
\label{eq_TFL}
\begin{pmatrix}
\mathsf{Id}&\mathsf{0}\\
            \displaystyle\frac{\partial^2 \mathsf{L}}{\partial\mathsf{x}\partial\mathsf{v}}&  \displaystyle\frac{\partial^2 \mathsf{L}}{\partial\mathsf{v}^2}
\end{pmatrix}
\end{equation}
Therefore, $\mathbb{F}\mathbf{L}$ is an immersion, if and only if $\mathsf{x}\mapsto \displaystyle\frac{\partial^2 \mathsf{L}}{\partial \mathsf{v}^2}(\mathsf{x})$  is a  smooth field of  non-degenerate matrices on $\mathbb{E}$, defined on  some open set $\mathbb{U}$ in $ \mathbb{M}$.\\

\begin{remark} 
\label{R_conicfiber}
An open  set $\mathbb{E}_0$ in a convenient space $\mathbb{E}$ is called \emph{conic}\index{conic} if  for any $u\in \mathbb{E}_0$, $\lambda u$ belongs to $\mathbb{E}_0$ for all $\lambda>0$.\\
We will say that an open set $E_0$ in a convenient bundle $\pi:E\to M$ is a \emph{conic subbundle}\index{conic subbundle} of $E$ 
if the restriction of $\pi$ to $E_0$ is a convenient locally trivial bundle over $M$ whose typical fibre $\mathbb{E}_0$ is a conic open set of the typical fibre $\mathbb{E}$ of $E$. 
More generally,  we can consider a conic subbunbdle $E_0$ of $\Delta_M$ on which  $\mathbf{L}$ is non degenerate. We then say that \emph{$\mathbf{L}$ is $E_0$ non-degenerate}.\\

If $\mathbf{L}$ is $E_0$ non-degenrate, then  $\mathbb{P}$ is a  weak  convenient submanifold of $T^\prime M$ modelled on $\mathbb{M}\times\mathbb{E}$\footnote{$\mathbb{E}$ being the typical fibre of $\Delta_M$.} (\textit{via}  
the projection on the $\mathbb{M}\times \mathbb{E} $ in local coordinates)\footnote{In finite dimension,  when $\mathbb{P}$ is a submanifold of $T^\prime M$, $P$ is called the \emph{primary constraint}\index{primary constraint} in \cite{Dir50}.}.  Such a situation occurs in particular in the following examples.
\end{remark}

% If $\mathbf{L}$ is $K$ non-degenrate, then  $\mathbb{P}$ is a  weak  convenient submanifold of $T^\prime M$ modelled on $\mathbb{M}\times\mathbb{E}$ \footnote{$\mathbb{E}$ being the typical fibre of $\Delta_M$} (\textit{via}  
%the projection on the $\mathbb{M}\times \mathbb{E} $ in local coordinates)\footnote{In finite dimension,  when $\mathbb{P}$ is a submanifold of $T^\prime M$, $P$ is called the \emph{primary constraint} in \cite{Dir50}.}. This 
%situation occurs in particular in the following examples.

\begin{example}
\label{Ex_SubRiemannian}
Let $M$ be a smooth paracompact Banach  manifold and  $\Delta_M$ be a closed subbundle of $TM$  of codimension $k$.  Assume that there exists a weak Riemanniann metric $g$ on $\Delta_M$. Since $\Delta_M$ is supplemented, choose any $k$-dimensional supplement subbundle $F$ of $E$  in $TM$. Since $M$ is smooth paracompact, we can always extend $g$ to a smooth map $\mathbf{L}:TM\to \mathbb{R}$. For instance,  we can take any Riemannian metric $h$ on $F$  and extend it to a weak Riemannian metric $\hat{g}$ on $TM$ defined by

$\tilde{g}_{|F}= h$, $\tilde{g}_{|E}=g$ and $\tilde{g}_x(u,v)=0$ for all $u\in E_x$ and $v\in F_x$. 

\noindent 
 Then, by same arguments as in Example \ref{Ex_SubRiemannian}, 
we can show that $\mathbb{P}=\mathbb{F}\mathbf{L}(\Delta_M)$ is a weak immersed submanifold of $T^\prime M$.  

Therefore, over  the range $\mathbb{U}\subset\mathbb{M}$ of a chart domain  $U$ in $M$, there exist smooth fields $\alpha_i : \mathbb{U}\to\mathbb{E}^0$ such that $ \left( \alpha_1(\mathsf{x}),\dots,\alpha_k(\mathsf{x}) \right) $ is a basis of $\mathbb{E}^0$. Then, in associated local coordinates, there exist Lagrange multipliers $\lambda_i$, $ i \in \{1,\dots,k\}$, such that the associated implicit Lagrangian system can be written (cf. \cite{YoMa07}): 
 \begin{equation}\label{Eq_ImplicitLagrangiansubRiemanniank}
 \left\{
 \begin{matrix}
  \dot{\mathsf{x}}=\mathsf{v}\;\;\; \;\;\dot{\mathsf{p}}=\displaystyle\frac{\partial \mathsf{L}}{\partial \mathsf{x}}(\mathsf{x},\mathsf{v})  +\sum_{i=1}^k\lambda_i(\mathsf{x})\alpha_i(\mathsf{x})\hfill{}\\
  <\alpha_i(\mathsf{x}),\mathsf{v}>=0,\; 
  \forall i \in \{1\dots, k\}\;\;\;\; \;\;{\mathsf{p}}=\displaystyle\frac{\partial \mathsf{L}}{\partial \mathsf{v}}(\mathsf{x},\mathsf{v})\hfill{}
\end{matrix}
\right.
\end{equation}
Note that these equations only depend on the restriction of $\mathbf{L}$ to $\Delta_M$.
\end{example}

\begin{example}
\label{ex_Rope} 
{\sf The heavy rope.}  \\
Consider a non-elastic rope. Such a rope will be modelled as a map\\
$x:[0,\ell]\to \mathbb{R}^2$  such that $x \in H^1([0,\ell],\mathbb{R}^2)$ \footnote{This means that $x$ is continuous and its derivative $x^\prime $ is of class $L^2$.}  with $||x^\prime ||_2=1$  a.e. (cf. \cite{Rei22}).\\ 
The configuration manifold is then the Hilbert space 
$M= \{x \in H^1([0,\ell],\mathbb{R}^2) \} $. 
This space  is isomorphic to $\mathbb{R}^2\times L^2 \left( [0,\ell], \mathbb{R}^2 \right) $. The constraint condition is the vector subbundle\footnote{Here $x^\prime$ denotes the derivative of $x$ relative to $\sigma$.}  
\[
\Delta_M
=
\left\lbrace
(x,v)\in H^1([0,\ell],\mathbb{R}^2)\times H^1([0,\ell],\mathbb{R}^2) \;:\; \displaystyle\int_0^\ell x^\prime(\sigma)v(\sigma)^top d\sigma=0
\right\rbrace.
\]
of $TM$ of codimension $1$. 
The rope has a mass density per unit length $\rho:[0,1] \to ]0,\infty[$  and we assume that $\rho$ is smooth. Since $\rho(\sigma)>0$ for any $\sigma\in [0,\ell]$,  by a compactness argument, it follows that we have  some positive constant $C$ such that $0<||\rho(\sigma)||\leq C$. This rope is subjected to the force of gravity proportional to the constant field $ (0,g)$. Therefore the potential energy is 
$E_{\text{pot}}(x)=\displaystyle \int_0^\ell\rho(\sigma) x_2(\sigma)g d\sigma$ 
and the kinematic energy is 
$E_{\operatorname{kin}}(v)=\displaystyle\frac{1}{2}\int_0^\ell\rho(\sigma)||v||_2^2 d\sigma$.\\
But, in this context,  using the Riesz representation for the Hilbert space $H^1([0,\ell],\mathbb{R}^2)$ according to the inner product

$<(x,v),(y,w)>=<x,y>_{\mathbb{R}^2}+ <v,w>_{L^2}$,\\
the bundle  $\Delta_M$ defined by the constraint $v\mapsto <\alpha(x), v> := <x^\prime, v>_{L^2}$ and the Lagrangian associated to this physical  problem is  
\[
\mathbf{L}(x,v)= <\rho .v,v>_{L^2}-<\rho.x^\prime,(0,g)>_{L^2}.
\]
Thus, according to the assumption  on $\rho$, the restriction of the Hessian $\displaystyle \frac{\partial^2 \mathsf{L}}{\partial v^2}(x)$ associated in local coordinates to the restriction of $\mathbf{L}$ to $(\Delta_M)_x$ is non degenerate since it is positive definite. 

Therefore $\mathbb{P}$ is an immersed submanifold of
 \[
T^\prime H^1([0,\ell],\mathbb{R}^2)\equiv T H^1([0,\ell],\mathbb{R}^2)= H^1([0,\ell],\mathbb{R}^2)\times  H^1([0,\ell],\mathbb{R}^2)
\]
which, under Riesz representation, can be identified with $\Delta_M$. The associated implicit Lagrangian problem is then     
\begin{equation}\label{Eq_ImplicitLagrangiansubRope}
 \left\{\begin{matrix}
  \dot{\mathsf{x}}=\mathsf{v}\;\;\; \;\;\dot{\mathsf{p}}=-\rho\begin{pmatrix} 0\\
g
\end{pmatrix}+\lambda(x)x^\prime\hfill{}\\
<\mathsf{x}^\prime ,\mathsf{v}>=0,\;\;\;\; \;\;\mathsf{p}=\rho.\mathsf{v}\hfill{}
\end{matrix}\right.
\end{equation}
\end{example}

\subsection{Conservation of the Generalized Energy}  
\label{___ConservationOfTheGenaralizedEnergy}
Given an implicit Lagrangian  $(M, D_{\Delta_M},\mathfrak{D}\mathbf{L})$, following  \cite{YJM10}  or \cite{YoMa06I} in finite dimension, \emph{the generalized energy}\index{generalized energy} $\mathbf{E}$ on $TM\oplus T^\prime M$ is: 
\begin{equation}
\label{Eq_GeneralEnergy}
\mathbf{E}(x,v,p)=<p,v>-\mathbf{L}(x,v).
\end{equation}
Along an integral curve 
$ \left( x(t),v(t),p(t) \right) $   of (\ref{eq_ImplicitLagrangian}), 
we have \footnote{From now on, if $\mathbf{F}$  is a smooth function in some variables  $ \left( x_1, \dots, x_m \right)  $,
$\displaystyle\frac{\partial \mathbf{F}}{\partial x_i}$  denotes the partial derivative of $\mathbf{F}$ relative to the $i$-th  variable.}

\[\begin{matrix}
\displaystyle\frac{d}{dt}\mathbf{E}((x(t),v(t),p(t))&=<\dot{p},v>+<p,\dot{v}>-\displaystyle\frac{\partial \mathbf{L}}{\partial x}\dot{x}-\displaystyle\frac{\partial \mathbf{L}}{\partial v}\dot{v}\\
                                                           &=<\dot{p}-\displaystyle\frac{\partial \mathbf{L}}{\partial x},v>\hfill{}\\
                                                             &=0\hfill{}
                                                             \end{matrix}.
                                                             \]
                                            since $\dot{x}=v \in \Delta_M$ and $\dot{p}-\dfrac{\partial \mathbf{L}}{\partial x} \in \Delta_M^0$.\\                                                                                                                     
Note that if we have given a generalized energy $\mathbf{E}$ on $TM\oplus T^\prime M$, we can recover the Lagrangian from the relation  (\ref{Eq_GeneralEnergy})  and, from $\mathbf{E}$,  the implicit  Lagrangian system (\ref{eq_ImplicitLagrangian}) can be also considered as the implicit system
\begin{equation}
\label{eq_ImplicitEnergy}
(x(t),p(t),\dot{x}(t),\dot{p}(t), d\mathbf{E}(x(t),v(t),p(t))\in D_{\Delta_M}(x(t),p(t))
\end{equation}

\subsection{Implicit Hamiltonian System Associated to a Lagrangian}\label{___ImplicitHamiltonianSystemAssociatedToALagrangian}

Consider a non degenerate $\Delta_M$-constraint Lagrangian  $\mathbf{L}$ on $TM$ and again we set $\mathbb{P}=\mathbb{F}\mathbf{L}(\Delta_M)$.  Then $\mathbb{P}$ is a weak immersed submanifold of $T^\prime M$. We will adapt to  the convenient setting,  the construction in  finite dimension of the Hamiltonian $H_\mathbb{P}$  as  it is done in  a series of papers by H. Yoshimura and  J. E. Marsden  (cf. for instance \cite{YoMa06I} or \cite{YJM10}).

\begin{theorem}
\label{T_ImplicitHamiltonianL} %{T_ImplicitHamiltonianL} 
Let $M$ be a  convenient manifold 
%whose model $\mathbb{M}$ is reflexive and let 
and let $\Delta_M$ be  a closed  convenient subbundle of $TM$. Then, for any  $\Delta_M$-constraint regular Lagrangian $\mathbf{L}$, we have:
\begin{enumerate}
\item [(1)]   
Denote by   $H_{\Delta_M}:\Delta_M\to \mathbb{R}$ the smooth function defined by 
\[
H_{\Delta_M}(x,v)=E(x,v,\mathbb{F}\mathbf{L}_{| \Delta_M}(x,v)).
\]
Then the function   $H_\mathbb{P} (x,p)=H_{\Delta_M}(x,v)$ is a well defined smooth function on $\mathbb{P}$.
 \item[(2)]  
We set $H(x,v,p)=<p,v>-L(x,v)$. Then  the restriction $H_\mathbb{P}$ of  $H$ to $\mathbb{P}$ is smooth.  Moreover, 
 the induced   implicit Hamiltonian system $ \left( M,D_{\Delta_M},H_{| T^\prime M} \right) $ on $\mathbb{P}$  is well defined, and, in local coordinates, it can be characterized by   
\begin{equation}\label{D_InducedImplicitHamiltoniansystem}
  \dot{\mathsf{x}}=\displaystyle\frac{\partial \mathsf{H}}{\partial \mathsf{p}}(\mathsf{x},\mathsf{v},\mathsf{p}),\;\;\dot{\mathsf{p}}+\displaystyle\frac{\partial \mathsf{H}}{\partial \mathsf{x}}(\mathsf{x},\mathsf{v},\mathsf{p})\in \mathbb{E}^0,\;\;
     \end{equation} 
with constraints 
     $\displaystyle\frac{\partial \mathsf{H}}{\partial \mathsf{v}}(\mathsf{x},\mathsf{v},\mathsf{p})_{| \mathbb{E}}=0$.
\item[(3)] 
A smooth curve $t\mapsto (x(t),p(t))$ in $\mathbb{P}$, defined on an interval $I$,  is an integral curve of the induced   implicit Hamiltonian system $ \left( M,D_{\Delta_M},H_{| T^\prime M} \right) $ on $\mathbb{P}$ if and only if there exists a curve $t\mapsto (x(t),v(t))$ in $\Delta_M$  such that $\mathbb{F}\mathbf{L}(x(t),v(t))=(x(t),p(t))$ and which is an integral curve of the implicit Lagrangian system
$ \left( M, \Delta_M,\mathfrak{D}\mathbf{L} \right) $.     
\end{enumerate}
\end{theorem}
  
This proposition motives the following definition:

\begin{definition}
\label{D_ImplicitHamiltonianSystem}
The  implicit Hamiltonian system $(M,D_{\Delta_M},H_{| T^\prime M})$ with constraint $\left(\displaystyle\frac{\partial H_{T^\prime M}}{\partial v}\right)_{| \mathbb{P}}=0$ is called \emph{the induced implicit Hamiltonian associated to the $\Delta_M$-constraint regular Lagrangian $L$} and $\mathbb{P}$ the integrability domain of  $(M,D_{\Delta_M},H_{| T^\prime M})$.
\end{definition}

\begin{remark}
\label{R_HPDltaM} ${}$
\begin{enumerate}
\item[1.]  
The terminology "integrability domain of  $(M,D_{\Delta_M},H_{| T^\prime M})$" is justified by Assertion 2 and Assertion 3 of Theorem \ref{T_ImplicitHamiltonianL}. 
\item[2.] 
In the opposite to the context  of finite dimension, in the infinite dimensional convenient setting,  in general, since $\mathbb{P}$ is only a weak immersed submanifold of $TM \oplus T^\prime M$, there does not exist, even locally, a smooth extension $H$ of $H_\mathbb{P}$ to $TM\oplus T^\prime M$. Even more, for the same reason in general, the restriction to $\mathbb{P}$, of any smooth function $H$ on  $TM\oplus T^\prime M$ is not a smooth function on $\mathbb{P}$.
\item[3.] 
In local coordinates,  according to  Proposition~\ref{T_ImplicitHamiltonianL}, 2.,  under the constraint $\displaystyle\frac{\partial \mathsf{H}}{\partial \mathsf{v}}(\mathsf{x},\mathsf{v},\mathsf{p})_{| \mathbb{E}}=0$ and if $H_\mathbb{P}\cong \mathsf{H}_\mathbb{P}$,  the conditions in (\ref{D_InducedImplicitHamiltoniansystem}), 
can be written
\begin{equation}
\label{eq_MHPDeltaM}
 \dot{\mathsf{x}}=\displaystyle\frac{\partial \mathsf{H}_\mathbb{P}}{\partial \mathsf{p}}(\mathsf{x},\mathsf{v},\mathsf{p}),\;\;\dot{\mathsf{p}}+\displaystyle\frac{\partial \mathsf{H}_\mathbb{P}}{\partial \mathsf{x}}(\mathsf{x},\mathsf{v},\mathsf{p})\in \mathbb{E}^0,\;\;
 \end{equation}
 and so only depends on ${H}_\mathbb{P}$ and not on the smooth extension of $H_\mathbb{P}$ to $TM\oplus T^\prime M$. Note that our approach is quite different  from the results of  H. Yoshimura and  J. E. Marsden in finite dimension in the same context.
\end{enumerate}
\end{remark}

\begin{proof}[Proof of Theorem \ref{T_ImplicitHamiltonianL}]${}$\\

1. At first, consider the set:
\begin{equation}
\label{eq_StablE}
     \operatorname{Stab}( \mathbf{E})=\{(x,p, v)\in TM\oplus T^\prime M \;: \displaystyle\frac{\partial\mathbf{E}}{\partial v}(x, p, v)=0\}.
\end{equation}

From  (\ref{eq_StablE}),  for any $(x,p)\in \mathbb{P}$ if $(x,p)=\mathbb{F}L(x,v)$ then $(x,v,p)$ belongs to $ \operatorname{Stab}( \mathbf{E})$.  Therefore,   $ H_{\Delta_M}(x,v)=\mathbf{E}(x,v,\mathbb{F}(x,v)_{| \Delta_M})$ defines  a smooth function  on $\Delta_M$. Moreover, if $\mathbb{F}L(x,v)=\mathbb{F}L(x',v')$ then $x=x'$ and so $ H_{\Delta_M}(x,v)=H_{\Delta_M}(x,v')$. It follows that $H_\mathbb{P} (x,p)=H_{\Delta_M}(x,v)$ does not depend on the choice of $(x,v)$ such that $\mathbb{F}L(x,v)=(x,p)$  and so is a well defined function on $\mathbb{P}$.\\
On the other hand,  since $\mathbb{F}L$ is a local diffeomorphism from $\Delta_M$ to $\mathbb{P}$, there exist open neighbourhoods  $\mathcal{U}$ and $\mathcal{U}'$ of $(x,v)$ and $(x,v')$ in $\Delta_M$ such that $\mathbb{F}L(\mathcal{U})=\mathbb{F}
  L(\mathcal{U}'):=\mathcal{W}$  and its projection on $M$ is an open neighbourhood  $U$ of $x$ in $M$. It follows that $\mathbb{F}L(x, v)=\mathbb{F}L(x, v')$ for  all $(x,v)\in \mathcal{U}$ and $(x,v')\in \mathcal{U}'$ and then  ${H_{\Delta_M  }}_{| \mathcal{U}}={H_{\Delta_M  }}_{| \mathcal{U}'}={H_\mathbb{P}}_{|\mathcal{W}}$. Therefore, $H_\mathbb{P}$ is a  well defined smooth map on $\mathbb{P}$.\\ 
 
2. Let $\iota$   be the inclusion of $\mathbb{P}$ in $TM\oplus T^\prime M$. By construction,  we have  $ H \circ \iota (x, p)= H_\mathbb{P}(x,p)$. Since $H_\mathbb{P}$ is smooth, it follows that $H\circ\iota$ is also a smooth function on $\mathbb{P}$  and so $H$ is a smooth extension of $H_\mathbb{P}$ to $TM\oplus T^\prime M$.\\
 As  $H_{| \mathbb{P}}$ is smooth, the  constraint  $\displaystyle\frac{\partial {H}}{\partial {v}}({x},{v},{p})_{| \mathbb{P}}=0$  makes sense, which means that $dH_{| \mathbb{P}}= \left( \displaystyle\frac{\partial H}{\partial x}, \frac{\partial H}{\partial p} \right) $ and so the induced implicit Hamiltonian system $
 \left( M,D_{\Delta_M},H_{| T^\prime M} \right) $ on $\mathbb{P}$   that is:
 
  find  smooth  curves $\gamma: t\mapsto (x(t),p(t))$ in $\mathbb{P}$ such that  $ ( \dot{\gamma}(t), d_{\gamma(t))} H_{| \mathbb{P}})\in D_{\Delta_M}(\gamma(t))$  makes sense  (cf.  Definition \ref{D_ImplicitHamiltoniansystem}, 2.). 
 
In local coordinates, we have  $ (\mathsf{x},\mathsf{p})=\left(\mathsf{x}, \displaystyle\frac{\partial \mathsf{L}}{\partial \mathsf{v}}(\mathsf{x},\mathsf{v}) \right)$.
A smooth curve $t\mapsto  (x(t),p(t))\cong ( \mathsf{x}(t), \mathsf{p}(t))$ in $\mathbb{P}$  is an integral  curve of the induced   implicit Hamiltonian system $(M,D_{\Delta_M},H_{| T^\prime M})$ on $\mathbb{P}$ with constraint  $\displaystyle\frac{\partial {H}}{\partial {v}}({x},{v},{p})_{| \mathbb{P}}=0$
  if and only if 
\[
\dot{\mathsf{x}}(t)=\displaystyle\frac{\partial \mathsf{H}}{\partial \mathsf{p}}(\mathsf{x}(t),\mathsf{v}(t),\mathsf{p})\in \mathbb{E},\;\;\dot{\mathsf{p}}+\displaystyle\frac{\partial \mathsf{H}}{\partial \mathsf{x}}(\mathsf{x},\mathsf{v},\mathsf{p})\in \mathbb{E}^0
\]
with constraint  $\displaystyle\frac{\partial \mathsf{H}}{\partial \mathsf{v}}(\mathsf{x}(t),\mathsf{v}(t),\mathsf{p}(t))_{| \mathbb{P}}=0$ , this condition on $\mathbb{P}$ being equivalent to
\[
(\mathsf{x}(t),\mathsf{p}(t) )=
  \left( \mathsf{x}(t),\displaystyle\frac{\partial \mathsf{L}}{\partial\mathsf{v}}(\mathsf{x}(t),\mathsf{v}(t)) \right) .
\] 
  This  ends the proof of Point  2.\\
  
3. Consider a  curve   $t\mapsto (x(t),p(t))$ in $\mathbb{P}$, defined on an interval $I$, such that  that there exists a curve  $t\mapsto  \left( x(t),v(t) \right) $ in $\Delta_M$ with $ 
(x(t),p(t))=\mathbb{F}\mathbf{L}(x(t),v(t))$ for all $t\in I$. Moreover, assume that   $t \mapsto (x(t),p(t))$  is an integral curve of the induced   implicit Hamiltonian system $ 
\left( M,D_{\Delta_M},H_{| T^\prime M}  \right) $ on $\mathbb{P}$.  Choose any coordinates system associated to a chart domain which meets $x(I)$. If, in this coordinates system, $
(x(t),p(t))\cong (\mathsf{x}(t),\mathsf{p}(t))$ and $(x(t),v(t))\cong (\mathsf{x}(t),\mathsf{v}(t))$,  from    (\ref{eq_ImplicitPb}) and (\ref{Eq_DOmegaM}), we must have 
\begin{equation}
\label{eq_xvpIntgralH}
\dot{\mathsf{x}}(t)=\displaystyle\frac{\partial \mathsf{H}}{\partial \mathsf{p}}(\mathsf{x}(t),\mathsf{v}(t),\mathsf{p}(t))\in \mathbb{E}\;\;\; \;\;\dot{\mathsf{p}}(t)+\displaystyle\frac{\partial \mathsf{H}}{\partial \mathsf{x}}(\mathsf{x}(t),\mathsf{v}(t),\mathsf{p}(t))\in \mathbb{E}^0
\end{equation}

On the other hand, we also have
\begin{description}
\item [$\bullet$]
$\displaystyle\frac{\partial \mathsf{H}}{\partial \mathsf{p}}(\mathsf{x},\mathsf{v},\mathsf{p})= \mathsf{v}$ and so the left condition in (\ref{eq_xvpIntgralH}),is equivalent to   \\ $\mathsf{v}(t)=\dot{\mathsf{x}}(t)\in \mathbb{E}$.\

\item[$\bullet$]
$\displaystyle\frac{\partial \mathsf{H}}{\partial \mathsf{x}}(\mathsf{x},\mathsf{v},\mathsf{p})=-\displaystyle\frac{\partial \mathsf{L}}{\partial \mathsf{v}}(\mathsf{x},\mathsf{v})$.
Taking into account that   $(x(t),p(t))\in \mathbb{P}$ and $v(t)=\dot{x}(t)$, it follows that  $(x(t),p(t))={\mathbb{F}\mathbf{L}}(x(t),\dot{x}(t))$. Note that   if $\mathbf{L}\cong{\mathsf{L}}$ then ${\mathbb{F}L}\cong \displaystyle\frac{\partial {\mathsf{L}}}{\partial \mathsf{v}}$ and so  $
\mathsf{p}(t)=-\displaystyle\frac{\partial \mathsf{L}}{\partial \mathsf{v}}(\mathsf{x}(t),\mathsf{v}(t))$. Therefore, the right condition in  (\ref{eq_xvpIntgralH}) is then $\dot{\mathsf{p}}(t)-\displaystyle\frac{\partial \bar{\mathsf{L}}}{\partial \mathsf{x}}(\mathsf{x}(t),
\mathsf{v}(t))\in \mathbb{E}^0$.
\end{description}
 This implies that $t\mapsto (x(t),v(t))$   is an integral curve of the implicit Lagrangian system
$(M, \Delta_M,\mathfrak{D}\mathbf{L})$ (cf. (\ref{Eq_ImplicitLlocal})).
For the converse, under the assumption      $ (x(t),p(t))=\mathbb{F}\mathbf{L}(x(t),v(t))$ for all $t\in I$, we can use the previous arguments in reverse.
\end{proof}

\begin{example}
\label{Ex_SubRiemaniannToBeContinued}
According to the context of Example~\ref{Ex_SubRiemannian}, let $g$  be a strong Riemannian metric on a convenient manifold $M$;  so its  model $\mathbb{M}$ is a Hilbert space. Given any closed Banach subbundle $\Delta_M$ of $TM$,  the strong  Riemannian metric induces a strong Riemannian metric on $\Delta_M$. 
  In fact, the typical fibre of $\Delta_M$ is a Hilbert space  $\mathbb{E}$ and so $\mathbb{P}$ can be identified with the dual bundle $\Delta_M ^\prime$ of $\Delta_M$. Since $\bar{g}:=g_{| \Delta_M}$   defines a strong Riemannian metric on $\Delta_M$;   
  in particular, $\bar{g}^\flat$ is a bundle  isomorphism.  We  obtain a strong Riemannian co-metric $\tilde{g}$ on $\mathbb{P}$
\[
({g}_\mathbb{P})_x(p,p')=\bar{g}\left((\bar{g}_x^\flat)^{-1}(p),(\bar{g}_x^\flat)^{-1}(p')\right)
\]
The Hamiltonian is $H_\mathbb{P}=\frac{1}{2}{g}_\mathbb{P}$.  We have 
\begin{equation}
\label{eq_HamitonianSubRiemannian}
H(x,v,p)
=
<p,v>-\frac{1}{2}{g}_x(v,v)
=
\frac{1}{2}{g}_x \left( ({g}_x^\flat)^{-1}(p),({g}_x^\flat)^{-1}(p') \right) .
\end{equation}
In local coordinates, if ${g}_\mathbb{P}\cong \tilde{\mathsf{g}}_\mathbb{P}$,  then the associated induced  implicit Hamiltonian $(M,D_{\Delta_M},H_{| T^\prime M})$ can be written   (cf. (\ref{eq_MHPDeltaM}))
 \begin{equation}\label{eq_ImplicitHamiltonianinftyg}
  \dot{\mathsf{x}}=\displaystyle\frac{\partial (\mathsf{g}_\mathbb{P})_\mathsf{x}}{\partial \mathsf{p}}(\mathsf{p})\in \mathbb{E},\;\;\dot{\mathsf{p}}+\displaystyle\frac{1}{2}\frac{\partial (\mathsf{g}_\mathbb{P})
  _\mathsf{x}}{\partial \mathsf{x}}(\mathsf{p})\in \mathbb{E}^0\;\;\;\;
  \end{equation}
\end{example}

\begin{example}
\label{ex_RopHamiltonian}
Consider the context of Example \ref{ex_Rope}.\\
Taking into account the Riesz representation, the Hamiltonian $H_\mathbb{P}$ is given by

$H_\mathbb{P}(x,v)
=\left\langle\dfrac{1}{\rho}.p,(1-<x^\prime,p>_{L^2}.p) \right\rangle_{L^2}
-<\rho.x^\prime,(0,g)>$.\\
We have the function 

$H(x,p) = \left\langle \dfrac{1}{\rho}.p,p \right\rangle
-
\left\langle \rho.x^\prime,(0,g) 
\right\rangle $.\\ 
Thus the associated  implicit Hamiltonian system  is
\begin{equation}
\label{eq_ImplicitLHamiltonianRope}
 \left\{\begin{matrix}
  \dot{\mathsf{x}}=\frac{1}{\rho}\mathsf{p};\;\; \;\;\dot{\mathsf{p}}=-\rho\begin{pmatrix} 0\\
g
\end{pmatrix}+\lambda(x)x^\prime\hfill{}\\
<\mathsf{x}^\prime ,\mathsf{p}>=0,\;\;\;\; \;\;\mathsf{p}=\rho.v\hfill{}
\end{matrix}\right.
\end{equation}
which is equivalent to the implicit Lagrangian system (\ref{Eq_ImplicitLagrangiansubRope}).
\end{example}

\section[Variational Approach]{Variational Approach}

\label{__VariationalApproach}
  
\subsection{Hamiltonian Principle}\label{___HamiltonPrinciple}
 
In reference to \cite{SB-KZB17}, the purpose of this section is to give a link between the variational approach of implicit Lagrangian systems and implicit Hamiltonian systems in the convenient setting in a same way as in finite dimension (cf. \cite{YoMa06II}). Once more, the general locally convex setting is not well adapted to infinite dimensional problems in physics and fluid mechanics   
%in the optic of generalization of finite dimensional situation.   
(cf. Introduction).
 
At first, we assume that $M$ is  a convenient manifold. Let $C^\infty (I,M)$ be the set of smooth curves from an open  interval with compact closure $I \subset \mathbb{R}$ to $M$.  Note that if $M$ can be provided with  a local addition\footnote{Cf. 
\cite{KrMi97}, 42.4.}, then $C^\infty(I,M)$ has a convenient manifold  structure modelled on the convenient spaces $\Gamma_0(c^!TM)$ of sections with compact support of the pullback bundle $c^!TM$ where $c$ is any curve in  $C^\infty(I,M)$ and 
so $T_c C^\infty(I,M)$ can be identified with $\Gamma_0(c^! TM)$. In this case, this implies that for any $X\in\Gamma_0(c^!TM)$, there exists a smooth curve $\hat{c}:]-\varepsilon,\varepsilon[\to C^\infty(I,M)$ such that $\displaystyle\frac{d}
  {ds}_{| s=0}\hat{c}=X$.
  
We come back to the general situation of a convenient manifold $M$.\\
For any $c\in C^\infty(I,M)$,  its tangent map  $\dot{c}$ is a lift of $c$  in $C^\infty(I,TM)$.\\
To a given Lagrangian $\mathbf{L}:TM\to \mathbb{R}$  is  associated the \emph{functional action}  of $\mathbf{L}$  on $C^\infty(I, M)$ say  
$
{\mathfrak{F}_\mathbf{L}: C^\infty(I,M)\to  \mathbb{R}}$ defined by
\begin{equation}
\label{eq_MatcalL}
 \mathfrak{F}_\mathbf{L}(c)
 =\displaystyle \int_I 
 \mathbf{L} \left( c(t),\dot{c}(t) \right) dt.
\end{equation}
Note that although $\mathbf{L}$ could be not continuous for the given topology of locally convex space, it is continuous for the convenient topology, 
and, since $M$ is modelled on this topology, from the definition of the convenient smoothness, 
$\mathbf{L} \circ  \left( c,\dot{c} \right) $ 
is a smooth curve and so is continuous.  Indeed,  $TM$ has a structure of convenient manifold and for  this topology, from the definition of the convenient smoothness,  since $\mathbf{L}:TM \to \mathbb{R}$ is smooth, $\mathbf{L}$ continuous for the convenient topology.  On the other hand, for the same reason,  if $c$ is a smooth curve  in $M$, then the curve $ \left( c,\dot{c} \right) $ into $TM$ is smooth, so this curve is also continuous for the convenient topology. Therefore, $\mathbf{L} \left( c,\dot{c} \right) $ is a continuous curve in $\mathbb{R}$.  This  implies that  $\mathfrak{F}_\mathbf{L} (c) $ is well defined.\\
 
A \emph{variation}\index{variation} of $c \in C^\infty(I, M)$ is a  $C^\infty$ map  $\hat{c}:]-\varepsilon,\varepsilon[\times I\to M$ such that
\begin{description} 
\item[$\bullet$]
$\displaystyle\frac{\partial \hat{c}}{\partial{s}}(0,t)=\delta c(t)$  has a compact support in $I$; 
%\item 
%the curve $c_s:=\hat{c}(s,.)$ is a smooth curve  smooth  for all $s\in :]-\varepsilon,\varepsilon[$ with compact support in $\mathring{I}$;
% \item $\tilde{c}(s,t_1)=x_1$ and $\tilde{c}(s,t_2)=x_1$ for all $s\in :]-\varepsilon,\varepsilon[$;
\item[$\bullet$] 
$\forall t \in I,\; 
\hat{c}(0,t)=c(t)$.
\end{description}

In these conditions, we have:
\begin{remark}
\label{R_Icompact} ${}$
\begin{enumerate}
\item[1.]
The curve\footnote{One can also define a family of curves $\left( c^t \right) _{t \in I}$ by 
$
\begin{array}{cccc}
c^t: & [-\varepsilon,\varepsilon[ & \to & M 	\\
	& s	& \mapsto & c^t(s) = \hat{c}(s,t)
\end{array}
$.} $c_s:=\hat{c}(s,.)$  belongs to $C^\infty(I, M)$   for all $s\in ]-\varepsilon,\varepsilon[$ and  
$\delta c(t):=\displaystyle\frac{\partial \hat{c}}{\partial{s}}(0,t)$ 
  defines a section of $ \Gamma_0(c^!TM)$.\\
When  $c \in C^\infty(I, M)$, for any $\delta_c\in  \Gamma_0(c^!TM)$, there exists a germ of smooth curves $s\mapsto c_s$ in  $C^\infty(I, M)$ such that $c_0=c$ and such that $\delta c(t)=\displaystyle\frac{\partial \hat{c}}{\partial{s}}(0,t)$.\\
  We will call each such $\delta c$ an \emph{infinitesimal variation}\index{infinitesimal variation}\index{infinitesimal variation} of  $c$. Therefore, in general, the set of infinitesimal variations of $c$ is a strict subset of  $ \Gamma_c(TM)$.
\item[2.]  
Assume that $I=[t_1,t_2]$. Consider   a  $C^\infty$ map  $\hat{c}:]-\varepsilon,\varepsilon[\times I \to M$ such that $\hat{c}(0,t)=c(t)$
  and  $\hat{c}(s, t_1)=c(t_1)$ and  $\hat{c}(s, t_2)=c(t_2)$ for all $s\in ]-\varepsilon,\varepsilon[$. Then $\delta c(t):=\displaystyle\frac{\partial \hat{c}}{\partial{s}}(0,t)$   is an infinitesimal  
  variation of  $c$ such that $ \delta c(t_1)=\delta c(t_2)=0$.
\end{enumerate}
\end{remark}
Now, as classically, we have:
\begin{lemma}
\label{L_DmathcalL} 
Assume that $I=[t_1,t_2]$. 
\begin{enumerate}
\item[(1)] 
The map $s\mapsto \mathfrak{F}_\mathbf{L}(c_s)=\displaystyle \int_{t_1}^{t_2}\mathbf{L}\left(\hat{c}(s,t),\frac{\partial \hat{c}}{\partial t}(s,t)\right)dt$ is well defined.    
\item[(2)]  
If we set, for all $t \in I$, 
$\displaystyle\frac{\partial \hat{c}}{\partial{s}}(0,t)=\delta c(t)$, then we have
\begin{equation}
\label{eq_DFL}
\begin{matrix}
  d_c\mathfrak{F}_\mathbf{L}(\delta c(t)) 
  := 
  \lim\limits_{s\rightarrow 0}\left( \displaystyle\frac{ \mathfrak{F}_\mathbf{L}(c_s)- \mathfrak{F}_\mathbf{L}(c)}{s} \right)\hfill{}\\
= 
\displaystyle\int_{t_1}^{t_2}
\left( \frac{\partial \mathbf{L}}{\partial x}
\left( c(t) ,\dot{c}(t) \right) 
-\frac{d}{dt}\frac{\partial \mathbf{L}}{\partial v}(c(t),\dot{c}(t))\right)(\delta c(t))dt\\
    + \left[
    \dfrac{\partial \mathbf{L}}{\partial v}
    \left( c(t),\dot{c}(t) \right) (\delta c(t) ) \right]_{t_1}^{t_2}
\end{matrix}
\end{equation}
\end{enumerate}
\end{lemma}
\begin{proof}[Sketch of the proof]${}$\\
1. Since the range of $c$ is compact, 
$c(I)$ can be covered by a finite number $\ \left( U_\alpha  \right) _{\alpha \in \{1\dots,N\}}$  of connected chart domains. Then $c^{-1} \left( U_\alpha \right) $ is an open  $J_\alpha$ of $I$ and $ \left( J_\alpha \right) _{\alpha\in \{1\dots,N\}}$ is a covering of $I$. Each $J_\alpha$ is a finite union of open intervals of $I$. Thus, by this way, we obtain a covering $ \left( J_\beta \right) _{\beta \in \{1,\dots,K\}}$ of $I$ such that $c \left( J_\beta \right) $ is contained in some $U_\alpha$. We can choose $\beta_1,\dots, \beta_k$ such that $ \left( J_{\beta_j} \right) _{j \in \{1,\dots,k\} }$ is a minimal covering of $I$ and such that $t_1\in J_{\beta_1}$ and $t_2 \in J_{\beta_k}$.  From the minimality of this covering,  we can find a partition $\tau_0=t_1<\tau_2<\dots<\tau_{n-1}<\tau_n=t_2$
such that $ c([\tau_{i-1},\tau_i])\subset U_{\alpha_i}$ for $i \in \{1,\dots,k\}$. We set $I_i=[\tau_{i-1},\tau_i]$, for $i \in \{1,\dots,k\}$ and denote  $U_{\alpha_i}$, the chart domain which contains $c \left( I_i \right) $.\\

Let $\hat{c}:]-\varepsilon,\varepsilon[\times I \to M$ be a variation of $c$; since $U_i$ is open, without loss of generality, we can assume that $\hat{c}(]-\varepsilon,\varepsilon[\times I_i)\subset U_{\alpha_i}$. Then, in local  coordinates associated to $U_{\alpha_i}$, we have $\hat{c}_{| ]-\varepsilon,\varepsilon[\times I_i}(s,t)\cong \hat{\mathsf{c}}_i(s,t)$ and  $\dfrac{\partial \hat{c}}{\partial t}(s,t)\cong 
\dfrac{\partial \hat{\mathsf{c}}_i}{\partial t}(s,t)$. Now the map $s \mapsto  \dfrac{\partial \hat{\mathsf{c}}_i}{\partial t}(s,t)$ is a curve in $\mathbb{M}^\prime$. Since  the evaluation of the canonical pairing between $\mathbb{M}$ and $\mathbb{M}^\prime$ is continuous in the convenient topology, it follows that the map $s\mapsto   \mathbf{L} \left( \hat{c}(s,t),\dfrac{\partial \hat{c}}{\partial t}(s,t) \right)$ is continuous and so the map $s\mapsto \mathfrak{F}_\mathbf{L}({c_s}_{| I_i})\cong \displaystyle \int_{I_i}\mathsf{L}\left(\hat{\mathsf{c}_i}(s,t),\frac{\partial \hat{\mathsf{c}}_i}{\partial t}(s,t)\right)dt$ is well defined. The result is obtained by juxtaposition of these integrals.\\

2. Under the previous context, in restriction to  $ ]-\varepsilon,\varepsilon[\times I_i$,  using the chain rule (cf. \cite{KrMi97}, Theorem 3.18 ), the map $(s,t) \mapsto   \mathsf{L}\left(\hat{\mathsf{c}}_i(s,t),\dfrac{\partial \hat{\mathsf{c}}_i}{\partial t}(s,t)\right)$ is a  smooth function on $]-\varepsilon,\varepsilon[ \times I_i$ and  is, in particular,  continuous on $\displaystyle \left[ -\dfrac{\varepsilon}{2},\dfrac{\varepsilon}{2} \right] \times I_i$ and so  is bounded on  this compact set.  
Therefore, we can use the permutation of derivation relative to $s$ and integration on variable $t$.

\begin{equation}
\label{eq DFL1}
\begin{matrix}
\lim\limits_{s\rightarrow 0}\left(\displaystyle\frac{ \mathfrak{F}_\mathbf{L}({c_s}_{| I_i})- \mathfrak{F}_\mathbf{L}({c}_{| I_i})}{s} \right)
\cong \displaystyle \int_{ I_i} 
\left(\frac{\partial\mathsf{L}(\hat{\mathsf{c}}_i(s,t),\displaystyle \frac{\partial \hat{\mathsf{c}}_i}{\partial t}(s,t))}{\partial s}\right)_{| s=0}dt\hfill{} \\
=\displaystyle \int_{I_i}\left[\left(\frac{\partial\mathsf{L}}{\partial\mathsf{x}}({\mathsf{c}_i}(t),\dot{\mathsf{c}}_i(t))\right)
 \left( \frac{\partial \hat{\mathsf{c}_i}}{\partial s}(0,t) \right) 
+\left(\frac{\partial\mathsf{L}}{\partial\mathsf{v}}({\mathsf{c}}_i(t),\dot{\mathsf{c}}(t))\right)
\left( \frac{\partial ^2\hat{\mathsf{c}}_i}{\partial s\partial t}(0,t)
\right)
 \right]dt.
\end{matrix}   
\end{equation}

By using an integration by parts, we obtain
\begin{equation}
\label{eq_DFL1}
\begin{matrix}
\lim\limits_{s\rightarrow 0}\left(\displaystyle\frac{ \mathfrak{F}_\mathbf{L}({c_s}_{|  I_i})- \mathfrak{F}_\mathbf{L}({c}_{| I_i})}{s}\right)\hfill{}\\
 \cong \displaystyle \int_{I_i}\left(\frac{\partial \mathsf{L}}{\partial \mathsf{x}}(\mathsf{c}_i(t) ,\dot{\mathsf{c}}_i(t))-\frac{d}{dt}\frac{\partial \mathsf{L}}{\partial \mathsf{v}}(\mathsf{c}_i(t),\dot{\mathsf{c}_i(t}))\right)
 \left( \frac{\partial \hat{\mathsf{c}_i}}{\partial s}(0,t) \right)  
 dt\\
   + \left[\left(\dfrac{\partial \mathsf{L}}{\partial \mathsf{v}}(\mathsf{c}_i(t),\dot{\mathsf{c}}_i(t))\right)(\dfrac{\partial \hat{\mathsf{c}}_i} {\partial s}(0,t))\right]_{\tau_{i-1}}^{\tau_i}\\
      \end{matrix}
\end{equation}
which ends the proof.
\end{proof}

\begin{definition}
\label{D_CriticalPoint} 
A curve $c\in C^\infty(I,M)$ is called a \emph{critical point with fixed ends}\index{critical points}  for $\mathfrak{F}_\mathbf{L}$ in $C^\infty([t_1,t_2],M)$ if we have 
 $ d_c\mathfrak{F}_\mathbf{L}(\delta c(t))=0$ for all infinitesimal horizontal variations $\delta c$ of $c$ associated to variations with fixed ends.
\end{definition}

We then have the following result (cf. \cite{SB-KZB17}):
\begin{theorem}
\label{T_EulerLagrange-extremal} 
Let $\mathbf{L}:TM\to\mathbb{R}$ be a Lagrangian. 
 Then  $c\in C^\infty \left( [t_1,t_2],M \right) $ is  a critical point with fixed ends of the functional  $\mathfrak{F}_L$ if and only if $c$ satisfies the local Euler-Lagrange conditions: 
\begin{equation}
\label{eq_EulerLagrange}
 \left( \frac{\partial \mathsf{L}}{\partial \mathsf{x}}(\mathsf{c}(t) ,\dot{\mathsf{c}}(t))-\frac{d}{dt}\frac{\partial \mathsf{L}}{\partial \mathsf{v}} \right) 
 (\mathsf{c}(t),\dot{\mathsf{c}}(t))=0
\end{equation} 
in  any local coordinates associated to  chart domains which meet the range of $c$.
\end{theorem}

Taking into account Lemma~\ref{L_DmathcalL},  the proof of Theorem \ref{T_EulerLagrange-extremal} uses the same arguments  as in finite dimension, according  to the easy following generalization of the fundamental lemma of the calculus of variations to the convenient setting. However, for the sake of completeness, we will give the proof of Theorem \ref{T_EulerLagrange-extremal}.

\begin{lemma}
\label{L_VariationLemma} 
Let $f:[t_1,t_2]\to \mathbb{M}^\prime$ be a smooth map where $\mathbb{M}^\prime$ is the bornological dual of $\mathbb{M}$.  
Let $\mathbb{E}$ be a closed subspace  of $\mathbb{M}$ and assume that  we have
$\displaystyle\int_{t_1}^{t_2}<f(t),h(t)>dt=0$, for any smooth map $h:[t_1,t_2]\to \mathbb{E}$ such that $h(t_1)=h(t_2)=0$, then $f(t)$ belongs to $\mathbb{E}^0$, for all $t\in[t_1,t_2]$.\\
In particular, if $\mathbb{E}=\mathbb{M}$, then $f\equiv 0$ on $[t_1,t_2]$.
\end{lemma}

\begin{proof} 
Assume that we have a point $\tau\in [t_1,t_2]$ such that $f(\tau)\not \in \mathbb{E}^0$.   By continuity of $f$, it follows that we have a neighbourhood in $[t_1,t_2]$ on which  $f(t)\not\in\mathbb{E}^0$ and so we may assume that $t_1<\tau<t_2$. Since $f(\tau)$ is a linear form on $\mathbb{M}$, there  must exist $u\in \mathbb{E}$ such that $<f(\tau),u>=K>0$. Since $f$  and the evaluation for the canonical pairing between $\mathbb{M}^\prime$ and $\mathbb{M}$ are continuous  for the $c^\infty$-topology, there exists a sub-interval $J$ of $[t_1,t_2]$ around $\tau$ such that $<f(t),u> \geq \dfrac{K}{2}$, for all $t\in J$. On the other hand, there exists a smooth map $\theta:J\to [0,+\infty[$  such that 
$\theta(\tau)=1$  for $t$ in a closed subinterval $J_0\subset J$ around $\tau$ and whose support is contained in $J$. We can extend $\theta$ to a smooth function (again denoted $\theta$) on $[t_1,t_2]$ by setting $\theta(t)=0$ for   $t\in [t_1,t_2]\setminus J$ and we set   $h(t)= \theta(t) u$ on $ [t_1,t_2]$. Then by construction  $h$ is a smooth map from $[t_1,t_2]$ to $\mathbb{E}$ with $h(t_1)=h(t_2)=0$. But 
\[
\displaystyle\int_{t_1}^{t_2}<f(t),h(t)>dt\geq \textrm{length}(J_0) \times \frac{K}{2}>0
\]
which gives rise to a contradiction.
\end{proof}

\begin{remark}
\label{R_ContinuityEvaluation} 
In the proof of Lemma \ref{L_VariationLemma}, the continuity of the evaluation for the canonical pairing between $\mathbb{M}^\prime$ and $\mathbb{M}$  is fundamental, so such a Lemma is not true, in general, in the locally convex setting.
\end{remark}

\begin{proof}[Proof of Theorem \ref{T_EulerLagrange-extremal}] Assume that  $c$ is a critical point with fixed ends of $\mathfrak{F}_\mathbf{L}$. Consider a local chart domain $U$ such that $U\cap 
c([t_1,t_2])\not=\emptyset$ and let  ${J}$ be the open set in $I$ which is the inverse image of $U\cap c(I)$.  If  the closure of ${J}$ is equal to $ \displaystyle\bigcup_{i=1}^N[\tau_1^i,\tau_2^i]$ , denote by  $c_i$  the restriction of $c$ to $[\tau_1^i,\tau_2^i]$. Choose any  infinitesimal  variation $\delta c_i$ of   $c_i$  such that if $\delta c_i(t)=\displaystyle\frac{\partial \hat{c}_i}{\partial{s}}(0,t)$ then $\hat{c}_i(s,\tau^i_1)=\hat{c}(\tau^i_1)$ and $\hat{c}_i(s,\tau^i_2)=\hat{c}(\tau^i_2)$ for all $s$ and  for all $i \in \{1,\dots, N\}$.  If $J=I$, then $\hat{c_i}$ is a variation with fixed ends of $c$. If $J\not=I$, such a  variation $\hat{c}_i$  can be extended to a variation  of $c$ with fixed ends (again denoted $\hat{c}_i$)  by $\hat{c}(s,t)=0$ for all $t\in I\setminus [\tau_1^i,\tau_i^2]$ and all $s$.
 From  Definition~\ref{D_CriticalPoint}, we have $d_c\mathfrak{F}_\mathbf{L}(\delta c_i)=0=d_{c_i}\mathfrak{F}_\mathbf{L}(\delta c_i)$.  
According to (\ref{eq_DFL1}), we must have  
\[
\displaystyle \int_{\tau _1^i}^{\tau_2^i}\left(\frac{\partial \mathsf{L}}{\partial \mathsf{x}}(\mathsf{c}_i(t) ,\dot{\mathsf{c}}_i(t))-\frac{d}{dt}\frac{\partial \mathsf{L}}{\partial \mathsf{v}}(\mathsf{c}_i(t),
 \dot{\mathsf{c}_i(t}))\right)
 \left( \displaystyle\frac{\partial \hat{\mathsf{c}}_i}{\partial s}(0,t) \right) dt=0.
 \]
Therefore, from Lemma~\ref{L_VariationLemma}, it follows that  
\[
\displaystyle \frac{\partial \mathsf{L}}{\partial \mathsf{x}}(\mathsf{c}_i(t) ,\dot{\mathsf{c}}_i(t))-\frac{d}{dt}
  \frac{\partial \mathsf{L}}{\partial \mathsf{v}}(\mathsf{c}_i(t),\dot{\mathsf{c}}_i(t))=0
\]
on $[\tau_1^i,\tau_2^i]$.

But such a result is then true for all $i \in \{1\dots,N\}$ and so the  Euler-Lagrange condition is satisfies in local coordinates associated to $U$.  
   
Conversely, assume that  the Euler-Lagrange conditions are satisfied on each local coordinates systems which meet $c([t_1,t_2])$.\\
Using the context of the proof of Lemma \ref{L_DmathcalL},   consider  a finite partition\\ 
$\tau_0=t_1<\tau_0<\cdots<\tau_i<\cdots<\tau_n=t_2$ and an covering $U_{\alpha_1}, \dots U_{\alpha_n}$ of $c(I)$ such that $c([\tau_{i-1},\tau_i])\subset U_{\alpha_i}$.  From assumption, the Euler-Lagrange conditions  is satisfied in a local coordinates system  associated to each $U_{\alpha_i}$. Thus,  from (\ref{eq_DFL1}), we have
\[
\mathfrak{F}_\mathsf{L}(\delta c_{|[\tau_i,\tau_{i+1}]}) \cong\left[\left(\frac{\partial \mathsf{L}}{\partial \mathsf{v}}(\mathsf{c}(t),\dot{\mathsf{c}}(t))\right)
 \left( \frac{\partial \hat{\mathsf{c}}} {\partial s}(0,t) \right)\right]_{\tau_i}^{\tau_{i+1}}.
\] 
Therefore, if $\delta c(t_1)=\delta c(t_2)=0$,   by juxtaposition of the successive previous terms,   we obtain $d\mathfrak{F}_L(\delta c)=0$ which ends the proof.
\end{proof}

As in finite dimension (cf.  \cite{YoMa06I}),   according to Remark \ref{R_DeltaM=TM} when $\Delta_M=TM$, for any Lagrangian $\mathbf{L}:TM\to \mathbb{R}$, Theorem \ref{T_EulerLagrange-extremal} implies  

\begin{corollary}
\label{C_EulerLagrangeImplicitL} 
For any Lagrangian $\mathbf{L}:TM\to \mathbb{R}$, we have 
 the following equivalences:
\begin{enumerate}
\item[(1)] 
$c$ is a critical point with fixed ends of the action $\mathcal{F}_\mathbf{L}$;
\item[(2)]
$c$ satisfies the Euler-Lagrange  conditions;
\item[(3)]
if $(c(t),p(t))=\mathbb{F}\mathbf{L}(c(t),\dot{c}(t))$, then $(c(t),\dot{c}(t), p(t))$ is a solution of the differential system 
 (\ref{Eq_ImplicitLagrangianLocTM}).
\end{enumerate}
\end{corollary}
\begin{proof} 
We only have to prove the equivalence between (2) and (3) in local coordinates.  
Now we have 
\[
\mathbb{F}\mathbf{L}(c(t),\dot{c}(t)) =(c(t),p(t) )\cong \displaystyle \frac{\partial \mathsf{L}}{\partial \mathsf{v}}(\mathsf{c}(t) ,\dot{\mathsf{c}}(t))=\mathsf{p}(t)
\] 
Now, $(\mathsf{c},\dot{\mathsf{c}},\mathsf{p})$ is a solution of the differential system (\ref{Eq_ImplicitLagrangianLocTM}), if and only if 
\[
\displaystyle\frac{\partial \mathsf{L}}{\partial \mathsf{x}}
\left( 
\mathsf{c}(t),\dot{\mathsf{c}}(t) \right)
=
\dot{\mathsf{p}}(t)=\frac{d}{dt}
\left( \frac{\partial \mathsf{L}}{\partial \mathsf{v}}(\mathsf{c}(t),\dot{\mathsf{c}}(t)
\right)
\] 
which ends the proof.
\end{proof}

%!o $u$ <- $v$
\subsection{Infinite Dimensional Problems for Pontryagin's Principle}
\label{__OnTheInfiniteDimensionalProblemsForPontryaginPrinciple}

\subsubsection{Finite Dimensional Context}
\label{____PMPFiniteDimensionalContext}

\emph{We recall some results of control theory and Pontryagin Maximum Principle in finite dimension.}
%We begin by the following Remark which will be used in $\S$~\ref{__DirectLimitImplicitLagrangians}.

\begin{remark}
\label{R_ReachableSet}
Classically in finite dimension,  given a an anchored  vector bundle $(E,\pi, M, \rho)$, a curve $(c,v):[t_1,t_2]\to E$ is called \emph{admissible}\index{admissible curve}\index{curve!admissible} if we have 
 $\dot{c}(t)=\rho \left( (c(t),v(t) \right) $ where $v$ is of class $L^2$. We denote by  $\mathcal{A}([t_1,t_2], E) $ the set of admissible curves  $(c,v)$.\\
For any $x\in M$, the \emph{reachable set of $x$}\index{reachable set} is the set  $\mathcal{R}(x)$ of points $y$ in $M$ such that there exists  an admissible curve $(c,v)$ such that $c(t_1)=x$ and $c(t_2)=y$. The famous Stephan-Sussmann Theorem asserts that $\mathcal{R}(x)$ is always an immersed submanifold of $M$ (cf. \cite{Sus73} or \cite{Ste74}). By definition, $y\in \mathcal{R}(x)$ if there exists  a smooth piecewise admissible curve $(c,v):[t_1,t_2]\to E$  which joins $x$ to $ y$.  In fact, there always exists a smooth admissible curve which joins $x$ to $y$. This result is well known, but without precise reference, we give a proof.  Indeed, recall that \emph{a cutoff function}\index{cutoff function}\footnote{cf. \cite{CrFe11}.}
is a function $\tau_\epsilon \in C^\infty(\mathbb{R})$ with the following properties:
\begin{enumerate}
\item[(a)] 
$\tau_\epsilon (t) = 1$  for $t \geq \epsilon$ and $\tau (t) = 0$ for $ t\leq 0$;
\item[(b)] 
$\dot{\tau}_\epsilon(t) > 0$ for $t\in ]0,\epsilon[ $.
\end{enumerate}
 If $(c,v)$ is not smooth at  only $t_0\in]t_1,t_2[$,  if we set $v_\epsilon(t)=v(\dot{\tau}_\epsilon(t-t_0))v(t)$, for $\epsilon$ small enough,  then the curve $c_\epsilon$ such that $\dot{c}_\epsilon(t)=\rho({c}_\epsilon, v_\epsilon(t))$ with $c_\epsilon(t_1)=x$ is a smooth admissible curve which joins $x$ to $y$. By induction we can obtain the announced result.
\end{remark}
 
Given a Lagrangian  map $\mathbf{L}: E\to \mathbb{R}$,  for $\nu\in \mathbb{R}^+$,  we can associate the  map
 $H^\nu: E\times_M T^\prime M\to \mathbb{R}$ called the \emph{control Hamiltonian}\index{control Hamiltonian} defined by
\[
H^\nu(x,v,p)= <p, \rho(x,v)>-\nu \mathbf{L}(x,v).
\]
When $E=M\times U$ where $U$ is an open set of $\mathbb{R}^m$, then if an admissible curve $(c_0,v_0)$ is a maximum for $\mathfrak{F}_\mathbf{L}$ among all absolutely continuous admissible curves $(c,v)\in \mathcal{A}([t_1,t_2], E)$ such that $c(t_1)=c_0(t_1)$ and $c(t_2)=c_0(t_2)$, then, by Pontryagin Maximum principle, in each local coordinates on a chart domain which meets $c([t_1,t_2])$, we have  (cf. \cite{AgSa04} for instance):
\begin{equation}
\label{eq_MaximumPrinciple}
\begin{cases}
 \dot{x}(t)
 =\displaystyle\frac{\partial \mathsf{H}^\nu}{\partial \mathsf{p}}(\mathsf{x}(t),\mathsf{v}(t), \mathsf{p}(t))  \\
 \dot{p}(t)
 =-\displaystyle\frac{\partial \mathsf{H}^\nu}{\partial \mathsf{x}}(\mathsf{x}(t),\mathsf{v}(t), \mathsf{p}(t))  \\
 \displaystyle\frac{\partial \mathsf{H}^\nu}{\partial \mathsf{v}}(\mathsf{x}(t),\mathsf{v}(t), \mathsf{p}(t))=0 \;\; a.e.\end{cases}
\end{equation}
A curve $(x(t),v(t),p(t))$ which satisfies the implicit differential system (\ref{eq_MaximumPrinciple}) in each local coordinates on a chart domain which meets $c([t_1,t_2])$ is called a \emph{ bi-extremal of the Pontryagin principle}\index{bi-extremal}. It is well known that there exists two kinds of  bi-extremals of the Pontryagin principle:
\begin{enumerate}
\item
\emph{normal extremals}\index{normal extremal}\index{extremal!normal} for which $\nu>0$ and then  $(x(t), v(t), p(t))$ is smooth with  $p(t_2)\not=0 $;
\item
\emph{abnormal extremals}\index{abnormal extremal}\index{extremal!abnormal} such that $\nu=0 $ which do not depend on $\mathbf{L}$.
\end{enumerate} 
  
\begin{remark}
\label{R_EndxSubmersionInRx}
When $E=\Delta_M$, the anchor  $\rho$ is the inclusion of $\Delta_M$ in $TM$ and so an admissible curve is a tangent curve to $\Delta_M$.  Such admissible curves in $M$ are  classically called \emph{horizontal curves}\index{horizontal curve}\index{curve!horizontal} and their set will be denoted $\mathcal{A}_{\Delta_M}([t_1,t_2],M)$. Note that the control $u$ associated to an horizontal curve $(c,u)$ is unique and equal to $\dot{c}$.  In general,  the set  $\mathcal{A}_{\Delta_M}([t_1,t_2],M)$ is  the subset of horizontal curves in  $\mathcal{H}^1([t_1,t_2],M) $ of Sobolev class $H^1$. In fact, if $M$ is $n$-dimensional, $\mathcal{H}^1([t_1,t_2],M) $ has a Banach structure manifold modelled on $H^1 \left( [t_1,t_2],\mathbb{R}^n \right) $. Moreover, if  the typical fibre is $m$-dimensional, then $\mathcal{A}_{\Delta_M}([t_1,t_2],M)$ has a structure of Banach submanifold of $\mathcal{H}^1([t_1,t_2],M) $ modelled on $H^1([t_1,t_2],\mathbb{R}^m)$  whose tangent space $T_c\mathcal{A}_{\Delta_M}([t_1,t_2],M)$ is the set of $H^1$ sections of $c^!TM$ which  are infinitesimal horizontal variations of $c$  (cf. \cite{AOP01} or \cite{PiTa01} for instance among plenty of papers). Moreover, for any $x\in M$, the set of $\mathcal{A}_{\Delta_M}(x;[t_1,t_2],M)$ of horizontal curves $c$ such that $c(t_1)=x$ is also an $(n-m)$-codimensional Banach submanifold $\mathcal{A}_{\Delta_M}([t_1,t_2],M)$ and the tangent space  $T_c \mathcal{A}_{\Delta_M}(x;[t_1,t_2],M)$ 
  can be identified with the set of $H^1$ sections of  infinitesimal variations $\delta c$  of $c$ such that $\delta c(t_1)=0$. We denote by $\operatorname{End}_x: \mathcal{A}_{\Delta_M}(x;[t_1,t_2],M)\to M$  the \emph{endpoint map}\index{endpoint map} given by  $\operatorname{End}_x(c)=c(t_1)$. Then $\operatorname{End}_x$ is a smooth map.(cf. \cite{AOP01}  or \cite{PiTa01} for instance). In fact, $\mathcal{A}_{\Delta_M}(x;[t_1,t_2], M)$ is a submanifold of $\mathcal{A}(x;[t_1,t_2],\mathcal{R}(x))$ and so $\operatorname{End}_x$ takes values in $\mathcal{R}(x)$ and is a submersion. Unfortunately, this map is not a submersion in general. This is  the essential reason of the existence of abnormal bi-extremals which is well known in the finite dimensional sub-Riemannian context (cf. \cite{Mon02}).
\end{remark}

If  $T_c \operatorname{End}_{x_1}$ is a submersion on $T_{x_2}\mathcal{R}(x_1)$, then the set  
\[
\mathcal{A}_{\Delta_M}(x_1,x_2;[t_1,t_2],M)
:=\{c\in \mathcal{A}_{\Delta_M}([t_1,t_2],M)\;: c(t_i)=x_i, i \in \{1,2\} \}
\]
is a submanifold of $\mathcal{A}_{\Delta_M}([t_1,t_2],M)$ and $\delta c\in T_c\mathcal{A}_{\Delta_M}(x_1,x_2;([t_1,t_2],M)$ is a section of $c^!\Delta_M$ such that $\delta c(t_1)=\delta c(t_2)=0$.
 
On the other hand, given a Lagrangian $\mathbf{\bar{L}}:\Delta_M\to \mathbb{R}$. Then the Legendre transformation $\mathbb{F}\mathbf{\bar{L}}: (\Delta_M)_x\to (\Delta_M)_x^\prime$  is well defined (as in $\S$ \ref {___LegendreTransformation}).
 % Since we are in finite dimension, choose a subbundle $T$ in $TM$ such that $\Delta_M\oplus T=TM$. Then $T^\prime M=\Delta_M^0 \oplus T^0$ and so we can identified $\Delta_M^\prime $ with $T^0$. By the way we will consider $\mathbb{F}\bar{L}$ as a smooth  map from $\Delta_M$ to $T^0$. \
Now, to $\mathsf{\bar{L}}$ is naturally associated  the functional 
$\mathfrak{F}_{\mathbf{\bar{L}}}: \mathcal{A}_{\Delta_M}([t_1,t_2],M) \to \mathbb{R}$ defined by
\begin{equation}
\label{eq_FbarLdimfinie}
\mathfrak{F}_{\mathbf{\bar{L}}}(c)
= \displaystyle\int_{t_1}^{t_2} \mathbf{\bar{L}} \left( x(t),\dot{x}(t) \right) dt.
\end{equation}
Since $\mathcal{A}_{\Delta_M}([t_1,t_2],M)$ is a Banach manifold, each vector $\delta c\in T_c \mathcal{A}_{\Delta_M}([t_1,t_2],M)$ is an infinitesimal variation of $c$ obtained from a variation $\hat{c}:]-\varepsilon,\varepsilon]\times [t_1,t_2] \to \mathcal{A}_{\Delta_M}([t_1,t_2],M)$. Therefore,  using same arguments as in Lemma \ref{L_DmathcalL},  we can show that $\mathfrak{F}_\mathbf{\bar{L}}$ is differentiable on this manifold and we have 
\begin{equation}
\label{eq_DFbarLfini}
\begin{array}{rcl}
 d_c\mathfrak{F}_{\mathbf{\bar{L}}}(\delta c) 
    &=& \displaystyle \int_{t_1}^{t_2}\left(\frac{\partial\mathbf{\bar{L}}}{\partial x}(c(t) ,\dot{c}(t))-\frac{d}{dt}\frac{\partial \mathbf{\bar{L}}}{\partial v}(c(t),\dot{c}(t))\right)(\delta c(t))dt \\
 &&   + \displaystyle \left[ \frac{\partial \mathbf{\bar{L}}}{\partial v}(c(t),\dot{c}(t))(\delta c(t) ) \right] _{t_1}^{t_2}
\end{array}
\end{equation}

Consider $c\in  \mathcal{A}_{\Delta_M}([t_1,t_2],M)$ such that   $\operatorname{End}_x$ is a submersion at $c$ onto $T_{c(t_2)}\mathcal{R}(c(t_1))$.  Then $c$ is  a   critical point with fixed ends of $\mathfrak{F}_{\bar{L}}$  if   the restriction of $d \mathfrak{F}_{\bar{L}}$ to $T_c  \mathcal{A}_{\Delta_M}(x_1,x_2;[t_1,t_2],M) $ is zero.\\ 

\begin{remark}
\label{R_Interpretation P}
Since we are in finite dimension  we can extend $\mathbf{\bar{L}}$ to a Lagrangian $\mathbf{L}:TM\to \mathbb{R}$ such that  $\mathbf{L}_{| \Delta_M}=\overline{\mathbf{L}}(x,v)$. 
 Therefore, $\mathbb{F}\mathbf{L}_{| \Delta_M}=\mathbb{F}\bar{\mathbf{L}}$ takes values in $\mathbb{P}$. Now,  if $\Omega$ is the canonical symplectic form on $T^\prime M$ 
let $\varpi:\Delta_M\oplus T^\prime M \to T^\prime M$ the canonical projection. and denote $\widetilde{\Omega}=\varpi^*\Omega$. Then  $\widetilde{\Omega}$ is a pre-symplectic 
form on $\Delta_M\oplus T^\prime M$ whose kernel is  $\ker T\varpi$. As $\mathbb{P}$ can be  an immersed  submanifold of $ \Delta_M\oplus T^\prime M$, it is clear  that the 
restriction $\Omega_\mathbb{P}$  of $\widetilde{\Omega}$  to $\mathbb{P}$ is a symplectic form.\\

On the other hand, $\mathbf{L}$ gives rise to  an extension $\tilde{H}^\nu$ of the control Hamiltonian $H^\nu$ given by
$\tilde{H}^\nu (x, v, p)= <p,v>-\nu\mathbf{L}(x,v)$.
So the restriction  of $\tilde{H}^1$ to $\mathbb{P}$ is an Hamiltonian $H_\mathbb{P}$ on $\mathbb{P}$.   Consider the Hamitonian $H^1$  on $\Delta_M \oplus T^\prime M$  defined by 
$$H^1(x,v,p)=<p, v>-\overline{\mathbf{L}}(x,v)$$
Thus $H_\mathbb{P}$ is also the restriction  of $H^1$ to $\mathbb{P}$.
Therefore,   $H_{\mathbb{P}}$ depends only on $(x, p)\in \mathbb{P}$.  If $X_\mathbb{P}$ is  the hamiltonian field of $H_\mathbb{P}$ on $\mathbb{P}$, and  this vector field depends only on $\overline{{\mathbf{L} }}$  and not on the choice of such an extension. In 
local coordinates, each integral curve of $X_\mathbb{P}$ is a solution of the the following system.
\begin{equation}
\label{eq_SolXP}
   \left\{\begin{matrix}
 \dot{x}(t)=\displaystyle\frac{\partial \mathsf{H}_\mathbb{P}}{\partial \mathsf{p}}(\mathsf{x}(t),\mathsf{p}(t))\hfill{}\\
 \dot{p}(t)=-\displaystyle\frac{\partial \mathsf{H}_\mathbb{P}}{\partial \mathsf{x}}(\mathsf{x}(t), \mathsf{p}(t))\hfill{}\\
  \end{matrix}\right.
 \end{equation}
\end{remark}
\bigskip

Most of the following results are well known in finite dimension  in sub-Riemanian geometry, but without complete references to our knowledge  for all of them in such a more general setting,  we will give a proof for completeness.

\begin{theorem}
\label{T_ImplicitHamiltonianPMP} 
Fix some  $c \in \mathcal{A}_{\Delta_M}(x;[t_1,t_2],M) $ Under the previous notations, we have the following equivalences:
 \begin{enumerate}
 \item[(1)] 
 We have the following   properties:
   
 $\bullet\;\;$ there exists $q\not=0\in T_{c(t_2)}^* M$ such that $d_c\mathfrak{F}_\mathbf{\bar{L}}=T_c^* \operatorname{End}_x(q)$.
   
 $ \bullet\;\;$  $c$ is a critical point with fixed ends  of $\mathfrak{F}_\mathbf{{\bar{L}}}$.
 \item[(2)] In any local coordinates on a chart domain which meets $c([t_1,t_2])$,\\ we have  the Euler-Lagrange condition:
\begin{equation}
\label{eq_EulerLagrangeDeltaTM}
\displaystyle \frac{\partial\mathbf{\bar{L}}}{\partial x}(c(t) ,\dot{c}(t))-\frac{d}{dt}\frac{\partial \mathbf{\bar{L}}}{\partial v}(c(t),\dot{c}(t))\equiv 0
 \end{equation}
  with  the condition $p((t_2)\not=0$ ,where $p$  is the lift of $c$ in $T^\prime M$ such that 
$\mathbb{F}\mathbf{L}(c(t),\dot{c}(t))=(c(t),p(t))$
\item[(3)]
If  $q(t)$  is a section  of $c^!(T^\prime M)$ such that  $q(t_2)\not=0$,   and  $\mathbb{F}\mathbf{L}(c(t),\dot{c}(t))=(c(t),q(t))$  then  $  (c(t), \dot{c}(t), q(t))$ is a bi-extremal    of the Pontryagin maximum  for  $\nu=1$ with $E=\Delta_M$ and $\rho$ is the inclusion of $\Delta_M$ in $TM$.
\item[(4)]
If  $(c(t),p(t))=\mathbb{F}\mathbf{\bar{L}}(c(t),\dot{c}(t))$, then  the curve  $t\mapsto (c(t),  p(t))$  is an integral curve of the Hamiltonian vector field $X_\mathbb{P}$ and $p(t_2)\not=0$
In particular, $t\mapsto (c(t), p(t))$ is smooth.
\end{enumerate}
\end{theorem}

\begin{definition}
\label{D_NormalExtremal} 
A curve  $c$  in $\mathcal{A}_x \left( [t_1,t_2],\Delta_M \right) $ is called a \emph{normal extremal}\index{normal extremal} if any of the  assertions in Theorem~\ref{T_ImplicitHamiltonianPMP}  is true.
\end{definition}

\begin{remark}
\label{R_ControlHamiltonianFini}${}$
\begin{enumerate}
\item From Assertion (4), any normal  bi-extremal $(c(t),p(t))$ is smooth. In particular, $c$ is smooth.
\item 
The condition: the range of  $T_c \operatorname{End}_x$ is $T_{c(t_2)}\mathcal{R}(x)$ implies  

${}\;\;$  there exists $q\not=0\in T_{c(t_2)}^\prime M$ such that $d_c\mathfrak{F}_{\mathbf{\bar{L}}}=T_c ^* End_x(q)$.
 
\noindent Note that, for $\nu=0$, a bi-extremal $(c(t),\dot{c}(t), p(t))$ is abnormal\\
if and only if $c$ is a singular point of the endpoint map, that is: 
 
 ${}\;\;$ there exists $q\not=0\in T_{c(t_2)}^\prime \mathcal{R}(x)$ such that $T_c ^* \operatorname{End}(q)=0$.
 
 \noindent In this case, we say that $c$ is  an abnormal extremal.
  % \item In finite dimension all the assumptions of Proposition  \ref{T_ImplicitHamiltonianL}  are satisfied and  so  this Proposition can be considered as a generalization  of the analog results in \cite{YoMa06I} or \cite{YoMa07}. %Note that in these paper the definition $H_P$ is equivalent to its definition  in Proposition  \ref{P_ImplicitHamiltonianL} point (1). 
 \item  According to Remark \ref{R_Interpretation P} and Assertion (4)  in Theorem \ref{T_ImplicitHamiltonianPMP}, all the Assertion in this Theorem do not depend on the extension $\tilde{H}^1$ of the control Hamiltonian $H^1$ such that $\tilde{H}^1_{| \Delta_M\oplus T^\prime M}=H^1$ and $\displaystyle\frac{\partial \tilde{\mathsf{H}}^1}{\partial \mathsf{v}}=0$ in local  coordinates.  
% In particular, they do not depend on the choice of the supplement $T$ of $\Delta_M$  in $TM$.\\
\end{enumerate}
\end{remark}
 
%For the proof of Theorem~\ref{T_ImplicitHamiltonianPMP}, we will use the result given in \cite{Arg20}:  Proposition 2
%\begin{lemma}
%\label{L_tEpFL}
%According to the  previous notations,  fix some $c$ of  class  at least  $C^2$in
%$\mathcal{A}_{\Delta_M}(c(t_1)); [t_1,t_2],M)$, and set  $t\mapsto p(t)\in T^* _{c(t)}M$  defined by $\mathbb{F}
% \mathbf{L}(c(t),\dot{c}(t))=p(t)$.\\
% There  exists $q\in T_{c(t_2)}^\prime M $ with $q\not=0$   such that  $\nu d_c\mathfrak{F}_\mathbf{L}- T_c^* End_{c(t_1)}(q)=0$ 
%if and only if $p(t_2)-p(t_1)=q$,
%  which is equivalent to $\displaystyle \frac{\partial H^1}{\partial v}(c(t), \dot{c}(t), p(t))=0$ with $p(t_2)-p(t_1)=q$.
%\end{lemma}

\begin{proof} [Proof of Theorem~\ref{T_ImplicitHamiltonianPMP}] 

 Consider some   $c \in \mathcal{A}_{\Delta_M} \left( x;[t_1,t_2],M \right) $. If there exists 
 $q\not=0\in T_{c(t_2)}M$ such that $d_c\mathfrak{F}_{\mathbf{\bar{L}}}=T_c ^* End_x(p)$,   from Proposition ~\ref{P_tEpFL} for $\nu=1$, and Remark \ref{R_CaseNu=1} applied to $E_0=\Delta_M$
we have
$$p(t)=\mathbb{F}{\mathbf{\bar{L}}}(c(t),\dot{c}(t)).$$
with  $p(t_2)=q$.\\
 
On the other hand,  by same arguments as in the proof of Theorem~\ref{T_EulerLagrange-extremal}, by taking  infinitesimal variations in   $ \mathcal{A}_{\Delta_M}(x;[t_1,t_2],M) $ and according  to (\ref{eq_DFbarLfini}),  we obtain\footnote{See also the proof of Theorem \ref{T_RelationExtremalFbarL}.}:

 $\displaystyle \frac{\partial{\mathsf{L}}}{\partial \mathsf{x}}(\mathsf{c}(t) ,\dot{\mathsf{c}}(t))-\frac{d}{dt}\frac{\partial {\mathsf{L}}}{\partial \mathsf{v}}(\mathsf{c}(t),\dot{\mathsf{c}}(t)) $ belongs to $\mathbb{E}^0$.

But,  we have   $\mathsf{L}(\mathsf{x},\mathsf{v})=\bar{\mathsf{L}}(\mathsf{x},\mathsf{v}) $  since $c$ is tangent to $\Delta_M$.\\
  It follows that  
\[
\frac{\partial{\mathsf{L}}}{\partial \mathsf{x}}(\mathsf{c}(t) ,\dot{\mathsf{c}}(t))-\frac{d}{dt}\frac{\partial {\mathsf{L}}}{\partial \mathsf{v}}(\mathsf{c}(t),\dot{\mathsf{c}}(t))
 =  
 \frac{\partial{\bar{\mathsf{L}}}}{\partial \mathsf{x}}(\mathsf{c}(t) ,\dot{\mathsf{c}}(t))-\frac{d}{dt}\frac{\partial \bar{{\mathsf{L}}}}{\partial \mathsf{v}}(\mathsf{c}(t),\dot{\mathsf{c}}(t)).
\]
  the assertion (2) is satisfied. \\
  
Assume that assertion (2) is true.  Thus, we have 
     
  $\mathsf{p}=\displaystyle\frac{\partial{\mathsf{L}}}{\partial \mathsf{v}}(\mathsf{c}(t),\dot{\mathsf{c}}(t))=\displaystyle\frac{\partial \bar{\mathsf{L}}}{\partial \mathsf{v}}(\mathsf{c}(t),\dot{\mathsf{c}}(t))$ belongs to $\mathbb{T}^0$.
  
Taking into account  (\ref{eq_EulerLagrangeDeltaTM}), this implies that  (in local coordinates) :  
\[
\displaystyle\dot{\mathsf{p}}(t)=\frac{d}{dt}\frac{\partial{\mathsf{L}}}{\partial \mathsf{v}}(\mathsf{c}(t) ,\dot{\mathsf{c}}(t))=\frac{\partial{\mathsf{L}}}{\partial \mathsf{x}}(\mathsf{c}(t) ,\dot{\mathsf{c}}(t))=-\frac{\partial\tilde{\mathsf{H}}^1}{\partial \mathsf{x}}((\mathsf{c}(t),\dot{\mathsf{c}}(t),\mathsf{p}(t)).
\]

Now, since on the range of the chart domain, we have $(\mathsf{c}(t),\dot{\mathsf{c}}(t))\in \mathbb{M}\times \mathbb{E}$,\\
so we have 
  
  $\dot{\mathsf{c}}(t)=\displaystyle\frac{\partial\tilde{\mathsf{H}}^1}{\partial \mathsf{p}}((\mathsf{c}(t),\dot{\mathsf{c}}(t),\mathsf{p}(t))$.
  
Finally, by construction, we have the constraint  $\displaystyle\frac{\partial\tilde{\mathsf{H}}^1}{\partial \mathsf{v}}((\mathsf{c}(t),\dot{\mathsf{c}}(t),\mathsf{p}(t))=0$.  This implies that   
 $(c(t),\dot{c}(t),p(t))$ is a bi-extremal normal of the Pontryagin Maximum principle,
 with the condition  $p(t_2)\not=0$, which is assertion (3).\\
Assume  that the  assertion (3) is true then the curve $(c(t), \dot{c}(t), p(t))$ satisfies the differential system (\ref{eq_MaximumPrinciple}) for $\nu=1$ with $(p(t_2))\not=0$ with $E=\Delta_M$ and $\rho$ is the inclusion in $TM$. This implies that $(c(t),\dot{c}(t))$  belongs to $\Delta_M$ and so $\mathbb{F}\mathbf{L}(c(t), \dot{c}(t))=(c(t) , p(t)))$  and so is an integral curve of $H_\mathbb{P}$. Under these conditions,  the Assertion (4) is 
  a consequence of  Remark \ref{R_Interpretation P}.\\ 
 
Assume that the assertion (4) is satisfied.  Then we have 
\[
\displaystyle\frac{\partial \tilde{\mathsf{H}}^1}{\partial \mathsf{v}}(\mathsf{x},\mathsf{v},\mathsf{p})=0
\]
with $p(t_2)\not=0$. \\
From  Proposition~\ref{P_tEpFL} for $\nu=1$, and Remark \ref{R_CaseNu=1} applied to $E_0=\Delta_M$,\\ we recover Assertion (1).
\end{proof}

\subsubsection{Convenient setting context}\label{____PMPConvenientSetting}
In infinite dimension, it is well known that the Maximum Pontryagin principle is not true. We will give a survey of some reasons for which the situation of  bi-extremal of the Maximum principle in the finite dimensional setting cannot be generalized  to the convenient setting.\\
  
Given a closed subbundle  $\Delta_M$ of $TM$ and consider  a Lagrangian $\bar{L}$ on $\Delta_M$. Let $C^\infty_{\Delta_M}([t_1,t_2], M)$  be the set of smooth curves $c:[t_1,t_2]\to M$ which are tangent to $\Delta_M$.
 In general, of course, this set does not have a convenient manifold structure. \\
   
For the following description , we assume that \emph{$C^k_{\Delta_M} \left( x;[t_1,t_2], M \right) $ is a convenient manifold}\footnote{If $M$ is a Banach manifold,  see Theorem \ref{T_AdmissibleCurves}. if $M$ is a convenient manifold see Theorem \ref{T_DirctLimitSequenceDeltaTMn}}.

Under this assumption, we can define the endpoint map  
\[
\operatorname{End}_x:C^\infty_{\Delta_M}(x;[t_1,t_2], M) \to M
\] 
which is smooth. For  
   $c\in C^\infty_{\Delta_M}(x;[t_1,t_2], M)$, the range of $T_c \operatorname{End}_x$ can be 
\begin{description} \label{De_QualificationCurves}  
\item[$\bullet$]  
equal to $T_{c_2}M$, and such a curve is called a \emph{normal curve}\index{normal curve}\index{curve!normal};
\item[$\bullet$] 
dense in $T_{c(t_2)}M$\footnote{For the existence of such curves, see \cite{Arg20} or \cite{Pel19}.};  such a curve is called \emph{elusive}\index{elusive curve}\index{curve!elusive} (cf. \cite{Arg20})\footnote{Of course, in finite dimension,  elusive curves do not exist.};
\item[$\bullet$]  
of codimension strictly positive; such a curve  is called \emph{abnormal curve} or \emph{semi rigid}\index{semi rigid curve} \index{curve!semi rigid}(cf. \cite{GMV15}).
\end{description}  

If  $\mathbf{\bar{L}}$ is a Lagrangian on $\Delta_M$, we again consider the "reduced   action" $\mathfrak{F}_{\mathbf{\bar{L}}}$ by the same formulae (\ref{eq_FbarLdimfinie}) and and by same argument as in the proof of Lemma~\ref{L_DmathcalL}, we can show that $\mathfrak{F}_{\mathbf{\bar{L}}}$ is differentiable. When the range  of $T_c \operatorname{End}_x$ is $T_{c(t_2)}M$, then there exists   $q\not=0$  in $T_{c(t_2)}^\prime M$ such that $T_c^* End_x(q)=d_c\mathfrak{F}_{\mathbf{\bar{L}}}$ which gives rise to a characterization in terms of a Hamiltonian condition according to Proposition ~\ref{P_tEpFL} for $\nu=1$ with $E_0=\Delta_M$ which is also true under our previous assumptions in the Banach setting. \\         
To recover the notion of bi-extremal normal, according to Remark~\ref{R_ControlHamiltonianFini}, we could  consider the case  where there exists  $q\not=0\in T_{c(t_2)}^\prime M$ such that $T^*_c \operatorname{End}_x(q)=d_c\mathfrak{F}_{\mathbf{\bar{L}}}$ and we will see that we can recover  the notion of  normal bi-extremal  \textit{via} Proposition ~\ref{P_tEpFL} for $\nu=1$ with $E_0=\Delta_M$.\\
When $c$ is elusive, $T_c \operatorname{End}_x$ is not a submersion. Therefore, locally

   $C^\infty _{\Delta_M}(x,y;[t_1,t_2],M):=\operatorname{End}_x^{-1}(\{y\})$ 
   
   \noindent will not be a subbmanifold of $C^\infty_{\Delta_M}(x;[t_1,t_2],M)$ and, in particular, this implies that, in general, there exists  $\delta c\in T_c C^\infty_{\Delta_M}(x;[t_1,t_2],M)$ such that $\delta c(t_2)=0$, but  there exists no horizontal variation $\hat{c}$ with fixed ends and such that 
   $\displaystyle \frac{\partial \hat{c}}{\partial s}(0,t)=\delta c$. Note that in a dual way,  if $T_c End_x$  is a submersion or if its range is dense, then $T^*_c \operatorname{End}_x$ is injective, and so, in both cases, this dual approach does give no distinction.  Finally, note that a curve  $c$ is abnormal if and only if $c$ is a singular point of  the endpoint map.
   
\begin{remark}
\label{R_ExistenceSingularCurves} 
The existence of elusive and abnormal curves is closely link to the problem of strong  controllability of $\Delta_M$, that is:
 for any pair of points $(x,y)\in M^2$ there exists a piecewise smooth curve $c:[0,1]\to M$ tangent to $\Delta_M$ and such that $c(0)=x$ and $c(1)=y$. For example, if $\Delta_M$\\
 is finite co-dimensional, there always exists normal curves. But if $\Delta_M$ is not  strongly controllable,  normal  curves could not exist.  For more details the reader can consult \cite{Arg20}.\\
\end{remark}

\begin{warning} 
\label{W_PMPBanachSetting}
From now on until the end of this section, we assume  that $M$ is a Banach manifold. Then for  $x\in M$ and $k\geq 2$,  the set  $C^k_{\Delta_M}(x;[t_1,t_2], M)$ is the set of curves of class $C^k$ $c:[t_1,t_2]\to M$ which are tangent to $\Delta_M$  and with $c(t_1)=x$ has a Banach manifold  structure. 
%(cf. Theorem  \ref{__NormalExtremals})
\end{warning}

As in finite dimension, for any $\nu\geq 0$, we can  consider the associated "control Hamiltonian" $H^\nu:\Delta_M\oplus T^\prime M \to \mathbb{R}$ given by
\[
H^\nu(x,v,p)=<p,v>-\nu \mathbf{\bar{L}}(x,v).
\]
 
Of course we have many problems  to recover a comparable situation as in finite dimension as exposed in  the previous section. The first one is that without the existence of a decomposition $TM=\Delta_M\oplus T$, for $\nu>0$,  we have no natural way to extend $H^\nu$ to  a Hamiltonian $\tilde{H}^\nu$ on $TM\oplus T^\prime M$  such that the restriction of   $d\tilde{H}^\nu$ to the tangent space of $TM$ is identically null. Note that, even if we had  a Lagrangian $\mathbf{L}$ on $TM$,   its restriction $\mathbf{\bar{L}}$ to $\Delta_M$ could not be smooth if $\Delta_M$ is not supplemented.  In any way, we have  a natural  map   
$\tilde{H}^\nu(x,v,p)=<p,v>-\nu \mathbf{L}(x,v)$ on $TM\oplus T^\prime M $ but, in general, its restriction to $\Delta_M\oplus T^\prime M$ is well defined $H^\nu$ but  $H^\nu$ could be not smooth in general.\\
  
\emph{To end this section,  in reference to  the context of Theorem~\ref{T_ImplicitHamiltonianL}, we will propose a framework in which we can define a normal extremal $c$ for the functional $\mathfrak{F}_{\bar{L}}$ which would be coherent with Definition \ref{D_NormalExtremal}.}\\

Assume that $\Delta_M$ is a closed Banach split subbundle of $TM$. Consider $\mathbf{\bar{L}}$ a Lagrangian on $\Delta_M$ and assume that  there exists  Lagrangian $\mathbf{L}:TM\to\mathbb{R} $ such that  $\mathbf{L}_{| \Delta_M}=\mathbf{\bar{L}}$  which is $\Delta_M$-constraint regular Lagrangian. Then $\tilde{H}^\nu(x,v,p)=<p,v>-\nu\mathbf{L}(x,v)$ is an extension of $H^1(x,v,p)=<p,v>-\bar{\mathbf{L}}(x,v)$ whose restriction to\\
$\mathbb{P}=\mathbb{F}\mathbf{L}(\Delta_M)$ is smooth. Therefore,  according to Theorem~\ref{T_ImplicitHamiltonianL} and Theorem~\ref{T_ImplicitHamiltonianPMP}., we introduce:
\begin{definition}\label{D_CharaterizationNormalConvenientc} 
Under the previous assumptions, a curve  $c\in C^\infty ([t_1,t_2],\Delta_M)$ is called a normal extremal if,  in each local coordinates on a chart domain which meets $c([t_1,t_2])$, the curve $t\mapsto (c(t),  p(t)):=\mathbb{F}\mathbf{L}(c(t),\dot{c}(t))$ 
is a solution of the implicit differential system
\begin{equation}\label{eq_MaximumPrincipleConvenient}
\begin{cases}
 \dot{\mathsf{x}}(t)
 =\displaystyle\frac{\partial \mathsf{H}^1}{\partial \mathsf{p}}(\mathsf{x}(t),\mathsf{v}(t), \mathsf{p}(t)) \\
 \dot{\mathsf{p}}(t)
 =-\displaystyle\frac{\partial \mathsf{H}^1}{\partial \mathsf{x}}(\mathsf{x}(t),\mathsf{v}(t), \mathsf{p}(t))\\
 \displaystyle\frac{\partial \mathsf{H}^1}{\partial \mathsf{u}}(\mathsf{x}(t),\mathsf{u}(t), \mathsf{p}(t))=0 
\end{cases}
\end{equation}
 with $p(t_2)\not=0$ 
\end{definition}

\begin{remark}
\label{R_NormalGeodesic} 
In finite dimension, this definition is coherent with the usual one (cf. Theorem~\ref{T_ImplicitHamiltonianPMP}).\\ 
In the Banach setting (cf.\cite{Arg20}), if  the anchored bundle $(\mathcal{H}, M,\xi )$ is $(\Delta_M, M, \operatorname{Id})$, then a normal sub-Riemannian  geodesic satisfies the previous definition according to Proposition ~\ref{P_tEpFL} for $\nu=1$ and $E_0=\Delta_M$.  More generally, note that Definition~7 in \cite{GMV15} of a normal sub-Riemannian geodesic satisfies the previous Definition (cf. \cite{Arg20}).
\end{remark}

\subsection{Characterization of Normal Extremals} 
\label{___CharacterizationOfNormalExtremals}
\emph{In this section, we will present  two  variational "Hamiltonian-Pontryagin principles"  relative to the restriction  to $\Delta_M$ of a $\Delta_M$-constraint $L$. This approach is somewhat different from such results of \cite{YoMa06II} in finite dimension. However, it  allows us to give some characterization of normal extremals}. \\

We again consider a Lagrangian $\mathbf{L}:TM\to \mathbb{R}$ which  is $\Delta_M$-constraint regular.   If $C^\infty([t_1,t_2], \Delta_M) $ is   the set of smooth curves $(c,v):[t_1,t_2]\to \Delta_M$, we denote by $C^\infty_{\Delta_M}([t_1,t_2], M)$ the set of curves of $C^\infty([t_1,t_2], \Delta_M) $ which are tangent to $\Delta_M$. Therefore we have
\begin{equation}
\label{eq_Horizontal}
C^\infty_{\Delta_M}([t_1,t_2], M)=\{(c, v) \in C^\infty([t_1,t_2], \Delta_M)\;:\;\dot{c}=v \}.
\end{equation}
 If $\bar{L}$ is the restriction of $L$ to $ \Delta_M$, naturally, to $\bar{L}$ is associated the functional
\[
\mathfrak{F}_{\mathbf{\bar{L}}}(c)=\displaystyle\int_{t_1}^{t_2}\mathbf{\bar{L}}(c(t),\dot{c}(t))dt
\]
As in the previous section,  a \emph{horizontal  variation} of $c\in C^\infty_{\Delta_M}( [t_1,t_2], M)$ is a  $C^\infty$ map  $\hat{c}:]-\varepsilon,\varepsilon[\times [t_1,t_2]\to M$ such that
 $c_s:=\hat{c}(s,.)\in C^\infty_{\Delta_M}([t_1,t_2], M)$  and  $\hat{c}(0,t)=c(t)$.
  We say that $\hat{c}$ is a horizontal variation with fixed ends if $c_s(t_1)=c(t_1)$ and $c_s(t_2)=c(t_2)$ for all $s$.\\
 $\delta c(t):=\displaystyle\frac{\partial \hat{c}}{\partial{s}}(0,t)$ is called an {\it infinitesimal horizontal  variation} of  $c$.
 
As in the proof of Lemma~\ref{L_DmathcalL}, 1., 
we can show that, for any horizontal variation $\hat{c}$ of $c$,   the map $s\mapsto \mathfrak{F}_{\mathbf{\bar{L}}}(c_s)$ is differentiable for $s=0$. 
  This differential is denoted  $d_c\mathfrak{F}_{\mathbf{\bar{L}}}(\delta c)$.
  
\begin{definition}
\label{D_ExtremalFLDeltaM} 
A critical point with fixed ends of $\mathfrak{F}_{\bar{L}}$  on  $C^\infty_{\Delta_M}([t_1,t_2], M)$  is then a curve $c\in C^\infty_{\Delta_M}([t_1,t_2], M)$ such that 
\begin{equation}
\label{eq_Extremal FLDeltaM}
 d\mathfrak{F}_{\mathbf{\bar{L}}}(\delta c)=0
\end{equation}
for all infinitesimal horizontal variations
  $\delta c$  associated to horizontal variations with fixed ends.
\end{definition}

On the other hand, according to (\ref{eq_Horizontal}) we can consider the functional:
 
\begin{equation}
\label{eq_FL+Constraint}
\widehat{\mathfrak{F}}_{\mathbf{\bar{L}}}=\displaystyle\int_{t_1}^{t_2}\left(\mathbf{\bar{L}}(x(t),v(t))+<p(t),(\dot{x}(t)-v(t))>\right)dt
\end{equation}
defined on curves in  $C^\infty([t_1,t_2],\Delta_M\oplus T^\prime M)$. 
 In the same way as in section \ref{___HamiltonPrinciple}, we can define variation $(\hat{c},\hat{v},\hat{p}):]-\varepsilon,\varepsilon[\times[t_1,t_2]\to C^\infty([t_1,t_2],\Delta_M\oplus T^\prime M)$ of $(c(t), v(t), p(t))$ and so  the associated     infinitesimal variation $(\delta c,\delta v,\delta p)$ given by:
\[
\delta c(t):=\displaystyle\frac{\partial \hat{c}}{\partial{s}}(0,t),\;\; \;\delta v(t):=\displaystyle\frac{\partial \hat{v}}{\partial{s}}(0,t),\;\; \;\delta p(t):=\displaystyle\frac{\partial \hat{p}}{\partial{s}}(0,t).
\]
  Once more, as  in the proof of Lemma~\ref{L_DmathcalL}, 1., we can show that, for any variation $(\hat{c},\hat{v},\hat{p}) $ defined on $]-\varepsilon,\varepsilon[\times [t_1,t_2]$ of $(c,v,p)$ in ${C^\infty([t_1,t_2],\Delta_M\oplus T^\prime M)}$, the map
  $s\mapsto \widehat{\mathfrak{F}}_{\mathbf{\bar{L}}}(c_s,v_s,p_s)$ is differentiable at $s=0$ and we denote $d_c{\mathfrak{F}}_{\mathbf{\bar{L}}}(\delta c, \delta v,\delta p)$ this differential. 
 
\begin{definition}
\label{D_ExtremalFbarL} 
A curve $t \mapsto (c(t),v(t),p(t))$  in  $\Delta_M\oplus T^\prime M$  defined on $[t_1,t_2]$ is a critical point with fixed ends of  $\widehat{\mathfrak{F}}_{\mathbf{\bar{L}}}$
if 
 \begin{equation}
 \label{eq_dWidehatFL}
d \widehat{\mathfrak{F}}_{\mathbf{\bar{L}}}(\delta c,\delta v,\delta p)=0
\end{equation}
for all   infinitesimal variations $(\delta c, \delta v,\delta p) $  associated to variations $(\hat{c},\hat{v},\hat{p})$ such that $\hat{c}(s,t_1)=c(t_1)$ and $\hat{c}(s,t_2)=c(t_2)$ for all $s$. \\
\end{definition}
 
 By the way, we have:
 
\begin{theorem}
\label{T_RelationExtremalFbarL}
 Under the previous considerations, we have the following equivalences:
 \begin{enumerate}
  \item[(1)] 
  There exists a section $p(t)$ of  the pull-back $c^!T^\prime M$ such that $(c(t),\dot{c}(t),p(t))$ is a critical point with fixed ends of  $\widehat{\mathfrak{F}}_{\mathbf{\bar{L}}}$.
  \item[(2)] 
  $c\in C^\infty_{\Delta_M}([t_1,t_2], M)$ is a critical point with fixed ends of $\mathfrak{F}_{\mathbf{\bar{L}}}$.
 \item[(3)] 
 In a coordinates system associated to any chart domain which meets $c([t_1,t_2])$ we have the constraint Euler-Lagrange condition
\begin{equation} 
\label{eq_EulerLagrangeE0}
\left( \frac{\partial \mathsf{L}}{\partial \mathsf{x}}(\mathsf{c}(t) ,\dot{\mathsf{c}}(t))-\frac{d}{dt}\frac{\partial \mathsf{L}}{\partial \mathsf{v}} \right)
(\mathsf{c}(t),\dot{\mathsf{c}}(t))\in \mathbb{E}^0.\\
\end{equation} 
\item[(4)] 
Let $H_\mathbb{P}$ be the Hamiltonian on $\mathbb{P}$ associated to $\mathbf{\bar{L}}$. The curve $c$ belongs to $C^\infty_{\Delta_M}([t_1,t_2], M)$ and   if  we denote by   $p(t)=\mathbb{F}\mathbf{L}_{c(t)}(\dot{c}(t))$ then   $(c(t),\dot{c}(t), p(t))$ is a smooth  integral curve of the implicit Hamiltonian system $(M, {H}_{| T^\prime M},D_{\Delta_M})$, that is, in local coordinates, satisfies the conditions
  \begin{equation}\label{eq_ImplicitHamiltonianinfty}
 \dot{\mathsf{x}}=\displaystyle\frac{\partial {\mathsf{H}}}{\partial \mathsf{p}}(\mathsf{x},\mathsf{v},\mathsf{p})\in \mathbb{E},\;\;\dot{\mathsf{p}}+\displaystyle\frac{\partial {\mathsf{H}}}{\partial \mathsf{x}}(\mathsf{x},\mathsf{v},\mathsf{p})\in \mathbb{E}^0,\;\;
\end{equation} 
     with constraints 
 $\displaystyle\frac{\partial {\mathsf{H}}}{\partial \mathsf{v}}(\mathsf{x},\mathsf{v},\mathsf{p})_{| \mathbb{E}}=0$.\\
\item[(5)]
 If  $(c(t), p(t))=\mathbb{F}\mathbf{L}_{c(t)}(\dot{c}(t))$, then $(c(t),\dot{c}(t), p(t))$ is an integral curve of the implicit Lagrangian system $(M, \Delta_M,\mathfrak{D}\mathbf{L}))$.

\end{enumerate}
\end{theorem}

\begin{corollary}
\label{C_NormalExtremal} 
If any of the Assertions of Theorem~\ref{T_RelationExtremalFbarL} is satisfied with $p(t_2)\not=0$, then
  $c$ is a normal extremal of 
  $\mathfrak{F}_{\mathbf{\bar{L}}}$.
\end{corollary}

 \begin{proof} [Proof of Theorem \ref{T_RelationExtremalFbarL}]
  $(1)\Longleftrightarrow  (3)$. 
   As in the proof of Theorem \ref{T_EulerLagrange-extremal},  we consider a local chart domain $U$ such that $U\cap 
c([t_1,t_2])\not=\emptyset$ and $J$ be the open set in $[t_1,t_2]$ which is the inverse image of $U\cap c([t_1,t_2])$. It is an open  interval whose  closure  is equal to $ \displaystyle\bigcup_{i=1}^N[\tau_1^i,\tau_2^i]$. Choose  any infinitesimal variations $ (\delta c, \delta v,\delta p)$ of $(c,\dot{c}, p)$ associated to a variation $(\hat{c},\hat{v},\hat{p})$  such that  $c_s:\hat{c}(s,.)$,  $v_s:=\hat{v}(s,.), p_s:=\hat{p}(s,.)$  have compact supports in $J$ and where 
$c_s(\tau^i_1)=c(\tau^i_2)$, for all $i \in \{1,\dots, N\}$.\\
By same  arguments as in  the proof of Proposition 3.1 in \cite{YoMa06II}, on the interval $J$  we have: 
\begin{equation}
\label{eq_dFLcvp}
\begin{matrix}
d \widehat{\mathfrak{F}}_{\mathbf{\bar{L}}}\left((\delta c,\delta v,\delta p)_{| \mathbb{J}}\right)\hfill{}\\
\cong \displaystyle\int_\mathbb{J}\left[(\dot{\mathsf{c}}-\mathsf{v})\delta\mathsf{p} +\left(-\dot{\mathsf{p}}+\frac{\partial\bar{\mathsf{L}}}{\partial \mathsf{x}}\right)\delta \mathsf{c}+\left(-\mathsf{p}+\frac{\partial\bar{\mathsf{{L}}}}{\partial \mathsf{v}}\right)\delta \mathsf{v}\right]dt  +\sum_{i=1}^N[p\delta c]_{\tau^i_1}^{\tau^i_2}\hfill{}\\
\end{matrix}
\end{equation}

According to the second member of (\ref{eq_dFLcvp}), note that the maps 

$t\mapsto (\dot{\mathsf{c}}-\mathsf{v})(t)$ takes values in $\mathbb{M}^{\prime\prime}$;

$t\mapsto \left(-\dot{\mathsf{p}}+\dfrac{\partial\bar{\mathsf{L}}}{\partial \mathsf{x}}\right)$ takes values in $\mathbb{M}^\prime$;

$t\mapsto \left(-\mathsf{p}+\dfrac{\partial\bar{\mathsf{{L}}}}{\partial \mathsf{v}}\right)$ takes values in $\mathbb{M}^\prime$.

Assume that the first member of equation (\ref{eq_dFLcvp}) is zero for any  previous choice of variations  $(\delta c, \delta v , \delta p)$.  Since $c_s(\tau^i_1)=c(\tau^i_2)$ for all $i \in \{1,\dots, N\}$, by application of Lemma ~\ref{L_DmathcalL}  to the second member of (\ref{eq_dFLcvp}), we must have 
\begin{equation}
\label{eq_ EulerLagrangeGeneral2}
 \dot{\mathsf{x}}= \mathsf{v}, \;\; \;\;\mathsf{v}\in \mathbb{E}, \;\; \mathsf{p}-\displaystyle\frac{\partial \mathsf{L}}{\partial \mathsf{v}}\in \mathbb{E}^0,\;\; \;\dot{\mathsf{p}}=\displaystyle\frac{\partial\mathsf{L}}{\partial \mathsf{x}}.
  \end{equation}

But since $\dot{\mathsf{p}}=\displaystyle\frac{\partial\mathsf{L}}{\partial \mathsf{x}}$, this implies that along $(\mathsf{c},\dot{\mathsf{c}})$, the equation (\ref{eq_EulerLagrangeE0}) is satisfied.\\

Conversely, under the assumption (3) by the choice of a finite partition $\tau_0=t_1<\tau_0<\cdots<\tau_i<\cdots<\tau_n=t_2$ and a covering $U_1, \dots U_i \dots U_n$ of $c([t_1,t_2])$ by chart domains, such that $c(\tau_i )$ belongs to $U_{i}\cap U_{i+1}$ for $0<i <n$, $c(\tau_0)\in U_0$ and $c(\tau_n)\in U_n$. As in the proof of Theorem~\ref{T_EulerLagrange-extremal},  by application of the condition (\ref{eq_ EulerLagrangeGeneral2})  in local coordinates systems  associate to each $U_i$ to  the second member of (\ref{eq_dFLcvp}) and juxtaposition,  we prove (1).\\
 
  $(2)\Longleftrightarrow  (3)$. Again take the same context of local coordinates as at the beginning of the proof of Theorem  \ref{T_EulerLagrange-extremal}. Then all infinitesimal horizontal variational
  $\delta c$  whose support is contained in $J$ and such that $\delta c(\tau^i_1)=\delta c(\tau^i_2)=0$, we have (cf. (\ref{eq_DFL1}):
  \begin{equation}
\label{eq_dFLcvp2}
\begin{matrix}
d {\mathfrak{F}}_{\mathbf{\bar{L}}}\left((\delta c)_{| \mathbb{J}}\right)\cong \hfill{}\\
  \displaystyle \int_\mathbb{J}\left(\frac{\partial \bar{\mathsf{L}}}{\partial \mathsf{x}}(\mathsf{c}(t) ,\dot{\mathsf{c}}(t))-\frac{d}{dt}\frac{\partial \bar{\mathsf{L}}}{\partial \mathsf{v}}(\mathsf{c}(t),\dot{\mathsf{c}(t}))\right)(\delta c(t))dt
   + \sum_{i=1}^N\left[\left(\frac{\partial \bar{\mathsf{L}}}{\partial \mathsf{v}}(\mathsf{c}(t),\dot{\mathsf{c}}(t))\right)(\delta c)\right]_{\tau_1^i}^{\tau_2^i}\\
  \end{matrix}
   \end{equation} 
Since for $\delta c$,  the associated variation of curves  $c_s$  are tangent to $\Delta_M$, it follows that $\delta c$ is section of $c^!\Delta_M$ and as $\delta c(\tau^i_1)=\delta c(\tau^i_2)=0$,  by application of Lemma  \ref{L_DmathcalL}  to the second member of (\ref{eq_dFLcvp2}), 
we obtain 
\[
\frac{\partial \bar{\mathsf{L}}}{\partial \mathsf{x}}(\mathsf{c}(t) ,\dot{\mathsf{c}}(t))-\frac{d}{dt}\frac{\partial \bar{\mathsf{L}}}{\partial \mathsf{v}}(\mathsf{c}(t),\dot{\mathsf{c}(t}))
\in \mathbb{E}^0.
\]

For the converse, we use same arguments as in the proof of Theorem \ref{T_EulerLagrange-extremal}.\\

 $(3)\Longleftrightarrow  (4)$.  We consider the context of local coordinates. If $c$ belongs to $C^\infty_{\Delta_M}([t_1,t_2], M)$, and $p(t)=\mathbb{F}\mathbf{L}_{c(t)}(\dot{c}(t))$, without any other assumption,  from the definition of $H$, we have
 
 \begin{equation}\label{eq_LinkEulerLagrangeInducedDirac}
\begin{cases}
\dot{\mathsf{c}}(t)=\displaystyle\frac{\partial {\mathsf{H}}}{\partial \mathsf{p}}(\mathsf{c}(t),\dot{\mathsf{c}}(t),\mathsf{p}(t)) \\
\displaystyle\dot{\mathsf{p}}(t)=\frac{d}{dt}\frac{\partial{\mathsf{L}}}{\partial \mathsf{v}}(\mathsf{c}(t) ,\dot{\mathsf{c}}(t)) \\
\displaystyle \frac{\partial{\mathsf{L}}}{\partial \mathsf{x}}(\mathsf{c}(t) ,\dot{\mathsf{c}}(t))=-\frac{\partial\tilde{\mathsf{H}}}{\partial \mathsf{x}}((\mathsf{c}(t),\dot{\mathsf{c}}(t),\mathsf{p}(t))
\end{cases}
\end{equation}

Note that the curve  $c$ belongs to $C^\infty_{\Delta_M}([t_1,t_2], M)$ if and only if   the first condition in (\ref{eq_ImplicitHamiltonianinfty}) is satisfied.
 
According to the two last equations in  (\ref{eq_LinkEulerLagrangeInducedDirac}), the Euler-Lagrange constraint condition  is equivalent to the second condition in (\ref{eq_ImplicitHamiltonianinfty})

The constraint condition  in (\ref{eq_ImplicitHamiltonianinfty}) is equivalent to $\displaystyle\frac{\partial \mathsf{L}}{\partial \mathsf{v}} (\mathsf{c}(t),\dot{\mathsf{c}}(t))=0$, which completes the proof.\\
 
$(4)\Longleftrightarrow  (5)$. Recall that from  (\ref{Eq_ImplicitLlocal}), in our context, the local implicit Lagrangian  system  in local coordinates is characterized by

\begin{equation}
\label{Eq_ImplicitLlocal}
\mathsf{p}=\displaystyle\frac{\partial \mathsf{L}}{\partial \mathsf{v}},\;\;\;\;\; \dot{\mathsf{x}}\in \mathbb{E}\,;\;\;\;\;\dot{\mathsf{x}}=\mathsf{v},\;\; \;\;\;\;\dot{\mathsf{p}}-\displaystyle\frac{\partial \mathsf{L}}{\partial\mathsf{x}} \in \mathbb{E}^0.
\end{equation}
 But from  the definition of $H$ and the properties of $L$, these conditions are equivalent to condition (\ref{eq_ImplicitHamiltonianinfty}).   
\end{proof}

\begin{proof}[Proof of Corollary  \ref{C_NormalExtremal}] 
We can assume that Assertion (4) of Theorem \ref{T_RelationExtremalFbarL} is satisfied. This implies that the assumptions of Proposition \ref{P_tEpFL} are satisfied for $E=\Delta_M$
  and  $\nu=1$. But it  is exactly the context of Definition~\ref{D_CharaterizationNormalConvenientc}. 
which ends the proof.
\end{proof}

\section[Dynamical Systems And Implicit Lagrangians on Direct Limits of Finite Dimensional Manifolds]{Dynamic Systems and Implicit Lagrangians on Direct Limits of Finite Dimensional Manifolds}
\label{__DynamicSystemsImplicitLagrangiansOnDirectLimitsOfFiniteDimensionalManifolds} 
 
\subsection{Some Basic Results}\label{___SomeBasicResults}
Let $M$  be a finite dimensional manifold provided with its  Pontryagin bundle $T^\mathfrak{p}M$ and  Dirac structure $D_M$. If $\phi:N\to M$ is a smooth map, the pull-back $\phi^!(D_M)$ of $D_M$ on $N$ is defined in the following way (cf. \cite{Bur13}): 
\[
(\phi^!(D_M))_x:=\{(X,d\phi^\prime(\beta))\;|\; (d\phi(X),\beta)\in (D_M)_{\phi(x)}\}.
\]

When $\phi:N\to M$ is the inclusion of a submanifold $N$ into $M$, then the pull-back  $\phi^!(D_N)$ is a Dirac structure on $N$ (cf. \cite{Bur13}, Example 5.7).\footnote{Note that this definition is compatible with the definition of compatible Dirac structure in \cite{PeCa24}, Remark 5.3.}

Let $ \left( M_n \right) _{n\in \mathbb{N}}$ be an ascending sequence of finite dimensional manifolds where $M_n$ is submanifold of $M_{n+1}$, each one  provided with its Pontryagin bundle 
$T^\mathfrak{p}M_n=TM_n\oplus T^\prime M_n$ and with a Dirac structure $D_n$ such that $D_n$ is the pull-back of $D_{n+1}$. 
We say that $ \left( M_n, D_n \right) _{n \in \N} $ is an \emph{ascending sequence of finite dimensional Dirac structures}\index{ascending sequence!of finite dimensional Dirac structures}. Then $M=\underrightarrow{\lim}M_n
=\displaystyle\bigcup_{n\in \mathbb{N}} M_n$ has a convenient manifold structure modelled on $\mathbb{R}^\infty$ (the set of finite real sequences) and $TM=\underrightarrow{\lim}TM_n=\displaystyle\bigcup_{n\in \mathbb{N}}TM_n$ and so $T^\prime M=\underleftarrow{\lim}T^\prime M_n$ which is a Fr\'echet bundle over $M$. We denote by $p_n$ and $p^\prime_n$ the canonical projections of $T^{\mathfrak{b}}M_n$ onto $TM_n$ and $T^\prime M_n$ respectively. 
\\
Then we have the following results given in \cite{PeCa24}:
\begin{proposition}\label{P_FiniteDirectLimitDirac} 
Let $(M_n)_{n\in\mathbb{N}}$ be   an ascending  sequence $(M_n,) _{n \in \mathbb{N}}$ of finite dimensional  manifolds $M_n$, each one being provided with the Pontryagin bundle  $T^\mathfrak{p}M_n = TM_n\oplus T^\flat M_n$ and $D_n$ a Dirac structure on $M_n$ which is the pull-back of $D_{n+1}$ for all $n\in\mathbb{N}$.
\begin{enumerate}
\item[1.] 
We  have convenient bundle  projections $p: T^\mathfrak{p} M\to TM$ and $p^\prime: T^\mathfrak{p}M\to T^\prime M$ which are given by $p=\underrightarrow{\lim}p_n$ and $p^\prime=\underleftarrow{\lim}p'_n$.
\item[2.] 
${D}=\underrightarrow{\lim}p_n(D_n)\times \underleftarrow{\lim}p^\prime_n({D}_n)$ is an almost  Dirac structure on $M$.
\item[3.]
A dual pairing $<.,.>$ on $T^\flat M \times TM$ is defined  in the following way:

for $\alpha=\underleftarrow{\lim}\alpha_n\in T_x^\prime M$ and $u=\underrightarrow{\lim}u_n\in T_xM$,  le  $i$ be the smallest integer $j$ such that $\alpha$ belongs to $T'_x M_i$, we have
\begin{center}
 $<\alpha,u>=\underleftarrow{\lim}_{n\geq j}\alpha_n(u_{j})=\underleftarrow{\lim}_{n\geq i}<\alpha_i, u_n>$.
 \end{center}
\item[4.]
The Courant bracket 
\begin{equation}\label{eq_DirctLimitCourantBracket}
[(X,\alpha), (Y,\beta)]_C
=\left([X,Y],L_X\beta-L_Y\alpha-\displaystyle\frac{1}{2}d(i_Y\alpha -i_X\beta)\right)
\end{equation}
is well defined for all sections $(X,\alpha)$ and $(Y,\beta)$ of $T^\mathfrak{p}M$  such that 
$X=\underrightarrow{\lim}X_n$ (resp. $Y=\underrightarrow{\lim}Y_n$) and $\alpha=\underleftarrow{\lim}\alpha_n$ (resp. $\beta=,\underleftarrow{\lim}\beta_n) )$ where $(X_n,\alpha_n)$ (resp. $ (Y_n,\beta_n)$) is a section  of $T^\mathfrak{p}M_n$\footnote{In fact, in general,  any section of  $T^\mathfrak{p}M$  is not a direct  limit  of  ascending sequence of $T^\mathfrak{p} M_n$, this Courant bracket is not defined on all sections of $T^\mathfrak{p}M$  (cf. Remark 5.11 in \cite{PeCa24}) }. This means that we have
\begin{equation}
\label{eq_directLimitCorantSequnce}
[(X,\alpha),(Y,\beta)]_C=\left(\underrightarrow{\lim}[X_n,Y_n],\underleftarrow{\lim} p^\flat_n([(X_n,\alpha_n),(Y_n,\beta_n)]_C)\right)
\end{equation}
Then the almost Dirac structure $D$ on $M$ is not  a Dirac structure on $M$ since the Courant bracket in not defined for all  (local) sections of $T^\mathfrak{p}M$. However, it is "integrable"  in the followings sense:
\begin{description} 
\item[$\bullet$] 
Given any $x\in M$, there exists a smallest integer $i$ such that $x_i$ belongs to $M_i$. Then the  leaf ${L}$ of the characteristic distribution of $D$ through $x$ is the leaf ${L}_i$ through $x_i \in M_i$ of the characteristic distribution of $D_i$ and for $x_n=x_i$ and $ n\geq i$ the leaf through $x_n$ is $L_i$
\item[$\bullet$] 
On each such a  leaf ${L}$, we have a bundle convenient skew symmetric  morphism \footnote{cf. \cite{PeCa24}, Corollary 5.10.} $P_{L}=\underleftarrow{\lim}_{n\geq i}(P_{{L}_n})$ from $T{L}\to T^\prime {L}$ whose kernel is $D\cap T{L}$ and range is $T^\prime {L}\cap T^\flat M$.  The associated $2$-form $\Omega_{{L}}$ is $\underleftarrow{\lim}_{n\geq j} (\Omega_{{L}_n})$ and it  is a strong   pre-symplectic $2$-form on ${L}$. \\
In other words, for $n\geq i$ this means that $(P_{L})_{| T{L}_n}\to T^\prime {L}$ is $P_{{L}_n}=P_{{L}_i} $, $\ker P_{L}=D_n\cap T{L}_n$, and  $\Omega_{{L}}=\Omega_{{L}_n}$.
\end{description}
 \end{enumerate}
 \end{proposition}
 
The following result can be considered as a sort of  converse of Proposition \ref{P_FiniteDirectLimitDirac}  and its  proof is left to the reader:

\begin{proposition}\label{T_ConverseBasicResults}
Under the previous notations we have:
\begin{enumerate}
\item[1.] The graph of $\Omega$ is a  partial Dirac structure $D\subset T^\mathfrak{p}(T^\prime M):=T(T^\prime M)\oplus T^\flat (T^\prime M))$  on $T^\prime M$, and 
  $D_n:=\iota^{!}_n D$ is the Dirac structure which is is the graph of $\Omega_n$.  In particular,   $D=\underrightarrow{\lim}p_n(D_n)\times \underleftarrow{\lim}p^\prime_n({D}_n)$
\item[2.]  The set of local sections $\Gamma \left( T^\mathfrak{p}M_{| U} \right)$ is provided with a Courant bracket (cf. (\ref{eq_CourantTM}) 
which satisfies the property of Proposition \ref{P_FiniteDirectLimitDirac}, Assertion 4.
\end{enumerate}
\end{proposition}

\subsection{Dynamical System Associated to an Implicit Lagrangian on the Direct Limit of Finite Dimensional Manifolds}
\label{DynamicalSystemAssociatedToAnImplicitLagrangianOnTheDirect LimitOfFiniteDimensionalManifolds}

 At first we have the following results on direct limits of implicit non-degenerate constraint Lagrangians:
 
\begin{theorem}\label{T_DirctLimitSequenceDeltaTMn} 
Let us consider the following data:
\begin{description}
\item [--] an increasing sequence of integers $m_n$ such that $0<m_n\leq n$;
\item[--] an ascending sequence $ \left( M_n   \right) $ of manifolds of dimension $n$;
\item[--] an ascending sequence $ \left(  \Delta_{M_n} \right) $ of  subbundles of $TM_n$ of rank $m_n$;
\item[--] on each $M_n$,  a Lagrangian $\mathbf{L}_n:TM_n\to \mathbb{R}$ which is non-degenerate on $\Delta_{M_n}$, for all $n$.
\end{description}
Then we have 
 \begin{enumerate}
\item[(1)] 
On  $M=\underrightarrow{\lim}M_n$, the direct limit  $\mathbf{L}=\underrightarrow{\lim}\mathbf{L}_n$ is a Lagrangian on $TM$ which is non degenerate on $\Delta_{T^\prime M}=\underrightarrow{\lim} \Delta_{T^\prime M_n}$. 
\item[(2)] 
The Legendre transformation $\mathbb{F}\mathbf{L}$  is  $\underrightarrow{\lim} \mathbb{F}\mathbf{L}_n$ and $\mathbb{P}:=\mathbb{F}\mathbf{L}$ is the direct limit 
 $\underrightarrow{\lim} \mathbb{P}_n$ if $\mathbb{P}_n=\mathbb{F}\mathbf{L}_n$.
\item[(3)] 
If $H_n(x_n,v_n,p_n)=<p_n,x_n>_\mathbf{L}(x_n,v_n)$ is the Hamiltonian associated to $\mathbf{L}_n$, the $H=\underrightarrow{\lim}H_n$ is the Hamiltonian associated to $\mathbf{L}$.
\item[(4)] 
If $t\mapsto(c_n(t), v_n(t),p_n(t))$ is an integral curve  the  implicit Lagrangian system  $(M_n, \Delta_{M_n},\mathfrak{D} \mathbf{L}_n)$ then 
\[
t\mapsto(c(t), v(t),p_n(t))
= \left(  \underleftarrow{\lim}c_n(t), \underleftarrow{\lim}v_n(t), \underrightarrow{\lim}p_n(t) \right)
\]
is an integral curve of the    implicit Lagrangian system  $ \left( M, \Delta_{M},\mathfrak{D} \mathbf{L} \right) $ 
\item [(5)]  
For $x_n\in M$, and $k\geq 2$,  consider   the Banach manifold   $C^k_{\Delta_{M_n}}(x_n;[t_1,t_2], M_n)$  of curves of class $C^k$ which are tangent to $\Delta_{M_n}$  where $c(t_1)=x_n$\footnote{cf. ~\ref{W_PMPBanachSetting}.}. 
Then, if $x=\underrightarrow{\lim}x_n$, for $[t_1,t_2]$ small enough,   
\[
C^k_{\Delta_{M_n}}(x;[t_1,t_2], M)=\underrightarrow{\lim}C^k_{\Delta_{M_n}} \left( x_n;[t_1,t_2], M_n \right)
\]
is a convenient manifold and the endpoint map $\operatorname{End}_x:\underleftarrow{\lim}\to M$ which is the direct limit of the maps 
$\operatorname{End}_{x_n}: C^k_{\Delta_{M_n}} \left( x_n;[t_1,t_2], M_n \right) \to M_n$. In particular, the range of $T_c\operatorname{End}_x$ is the direct limit of the range of $T_{c_n}\operatorname{End}_{x_n}$ and this range is closed in $T_{c(t_1)} M$.
 \item[(6)]  
There is no elusive  curves in $C^k_{\Delta_{M}} \left( x;[t_1,t_2], M \right) $. Moreover,
\[
t\mapsto(c(t), v(t),p_n(t))=(\underrightarrow{\lim}c_n(t), \underrightarrow{\lim}v_n(t), \underleftarrow{\lim}p_n(t))
\]
is a normal (resp. abnormal) curve if and only if there exists an integer $n_0$ such that  $t\mapsto(c_n(t), v_n(t),p_n(t))$ is a normal  (resp. abnormal) curve for all $n\geq n_0$.
\item[(7)] 
For any $x=\underrightarrow{\lim} x_n$ there exists an ascending sequence of immersed submanifolds $\mathcal{R}(x_n)$ in $M_n$ such that  
$\mathcal{A}_{\Delta_M} \left( x_n;[t_1,t_2], M_n \right) $ is a submanifold of $\mathcal{A} \left(  x_n;[t_1,t_2],\mathcal{R}(x) \right) $ and  $\operatorname{End}_{x_n}$ takes values in $\mathcal{R}_n(x_n)$ and is a submersion such that 
 \[
 \mathcal{R}(x)=\bigcup_{n\in \mathbb{N}}\mathcal{R}_n(x_n)
 \]
 is an immersed convenient submanifold of $M$ such that any point $y\in \mathcal{R}(x)$ there exists a horizontal curve of class $C^k$ which joins $x$ to $y$ and the   $\operatorname{End}_x$ takes values in  $\mathcal{R}(x)$ and is a submersion.
\end{enumerate}
\end{theorem}

\begin{proof} 
The proofs of Assertions (1), (2)  and (3) come from the Definition of direct limits of maps. The fact that $\mathbf{L}$ is non degenerate comes from the fact that any $x\in M$ belongs to some $M_n$ and so the Hessian of $\mathbf{L}$, in restriction to $\Delta_{M_n}$, is the Hessian of $\mathbf{L}_n$, which implies the results.\\
 
For the proof of  Assertion (4),  it is clear that, if  $\mathcal{D}_n$ is the  differential Dirac operator of $\mathbf{L}_n$, from the definition of such an operator and according to Assertion 1, the definition of direct limit of maps, and from Proposition \ref{T_ConverseBasicResults},  Assertion (1) (relative to $\Omega_n$), we have 
$\mathcal{D}\mathbf{L}
=\underrightarrow{\lim}\mathcal{D}_n\mathbf{L}_n$. Thus,taking into account the definition of an integral curve of an implicit Lagrangian, the result follows from  Proposition~\ref{T_ConverseBasicResults},   (1).\\
 
For the Assertion (5), it is clear that the  sequence  $ \left( C^k_{\Delta_{M_n}} 
\left( x_n;[t_1,t_2], M_n \right) \right) _{n\in \mathbb{N}}$ is an ascending sequence of Banach manifolds, so its direct limit is a convenient manifold (cf. \cite{CaPe23}, Chap.~5) contained in the set $C^k_{\Delta_{M_n}} 
 \left( x;[t_1,t_2], M \right) $. On the other hand, if $c:[t_1,t_2]\to M$ is tangent to $ \Delta_M$, and $c(t_1)=x=\underrightarrow{\lim}{\lim} x_n$, let $n_0$ be the smallest integer such that $\dot{c}(t_1)$ belongs to $\Delta_{M_{n_0}}$.  There exists $\epsilon >0$ such that 
 $c \left( [t_1,\epsilon[ \right) 
  \subset M_{n_0}$. But since $\Delta_{M_{n_0}}$ is closed, $\dot{c}(\varepsilon)$ must also belong to $\Delta_{M_{n_0}}$.  By connectedness argument, it follows that $c$ belongs to $C^k_{\Delta_{M_n}} \left( x_n;[t_1,t_2], M_{n_0} \right) $. Therefore, we have $C^k_{\Delta_{M_n}} \left( x;[t_1,t_2], M \right) =\underrightarrow{\lim}C^k_{\Delta_{M_n}} \left( x_n;[t_1,t_2],  M_n \right) $. 
 The last Property of the Assertion follows from the properties of direct limit of ascending sequences of Banach spaces (cf. \cite{CaPe23}, Chap~5).\\
 
 For assertion (6),  since the range of $T_c\operatorname{End}_x$  is closed in $T_{c(t_1)} M$, either this range is precisely $T_{c(t_1)}M$ or there exists $q\not=0 $ in $T_{c(t_2)}^\prime M$ such that $q$ belongs to the annihilator of the range of $T\operatorname{End}_x(T_c\left(C^k_{\Delta_{M_n}} \left( x;[t_1,t_2], M \right) \right)$, that is either $c$ is a normal curve or is 
 an abnormal curve.  In the first case, this is equivalent to  $\operatorname{End}_{x_n}:$ is a submersion for all $n\geq n_0$ for some $n_0$, that is  $t\mapsto \left( c_n(t), v_n(t),p_n(t) \right) $ is 
 a normal curve for all $n\geq n_0 $ and, in the second case, $t\mapsto \left( c_n(t), v_n(t),p_n(t) \right) $ is an  abnormal curve for all $n\geq n_0$.\\ 
 
 The assertion (7) is a direct consequence of Remark \ref{R_ReachableSet} and properties of direct limits of ascending local charts in finite dimension (cf. \cite{Glo05}). The property of the map $\operatorname{End}_x$ relative to $\mathcal{R}(x)$ is a consequence of properties of Direct limit maps and Remark \ref{R_EndxSubmersionInRx}.
\end{proof}

Conversely, consider an implicit Lagrangian on $\mathbf{L}: TM\to \mathbb{R}$ where $M=\underrightarrow{\lim}M_n$,
 $\Delta_M$ is a subbundle  of $TM$,  and  $\Delta_{T^\prime M}=(Tp_{T^\prime M})^{-1}(\Delta_M)$ is the  associated  closed subbundle of $TT^\prime M$.  According to Definition~\ref{D_ImplicitLagrangian},  let $D_{\Delta_M}$ be the induced partial almost Dirac structure  
where $\mathbb{F}\mathbf{L}$ is the   Legendre transformation of $\mathbf{L}$ and $\mathcal{D}\mathbf{L}$ is its Dirac 
operator. \\
We set $\mathbb{P}=\mathbb{F}\mathbf{L}(\Delta_M)\subset T^\prime M$.

\begin{proposition}
\label{P_DeltaM}
Under the previous context, we have:
\begin{enumerate} \item [1.]  
Let $\mathbf{L}_n:={\iota}_n^\prime\mathbf{L}$ then $\mathbf{L}=\underrightarrow{\lim}\mathbf{L}_n$.
 \item[2.] 
For each $n$, the intersection $ \Delta_M\cap TM_n$ is a subbundle   $\Delta_{M_n}$ of $TM_n$ which  gives rise to an ascending sequence $\left( \Delta_{M_n} \right) $ of vector bundles over the ascending sequences of manifolds $ \left( M_n \right) $ such that $\Delta_M=\underrightarrow{\lim}\Delta_{M_n}$. We also have 
\[
\Delta_{T^\prime M}=\underrightarrow{\lim} \Delta_{T^\prime M_n}.
\]
\item[3.] 
The sequence $ \left( M_n, \Delta_{M_n}, \mathbf{L}_n \right) $, satisfies the assumptions Theorem  \ref{T_DirctLimitSequenceDeltaTMn} and so the implicit Lagrangian system  $ \left( M, \Delta_M,\mathfrak{D} \mathbf{L} \right) $ is linked with the implicit Lagrangian system  $ \left( M_n, \Delta_{M_n},\mathfrak{D}_n \mathbf{L}_n \right) $ as in the conclusion of this Theorem.
\end{enumerate}
\end{proposition}

\begin{proof}
The Assertion 1 is clear. 

For assertion 2,  for any $x=\underrightarrow{\lim} x_n \in M$, according to  \cite{Glo05}, we have an ascending sequence of chart $ \left( U_n,\phi_n \right) _{n \in \N} $ around $x_n$ such that the pair   
$ \left( U=\underleftarrow{\lim} U_n, \phi=\underleftarrow{\lim}\phi_n \right) $ is a chart around $x$ in $M$.  Without loss of generality, we may assume that $\phi_n(U_n)$ is an open set $
\mathbb{U}_n$ of $\mathbb{M}_n \equiv \mathbb{R}^n$ around $0$ and so $\mathbb{U}=\underrightarrow{\lim} \mathbb{U}_n$ is an open set in $\mathbb{R}^\infty$ around $0$. After restriction, if necessary, of each $\mathbb{U}_n$,  we can assume that 
\[
(\Delta_M)_{| \mathbb{U}}=\mathbb{U} \times \mathbb{E}\subset TM_{| \mathbb{U}}=\mathbb{U}\times \mathbb{R}^\infty.
\]
  Since $\mathbb{R}^\infty=\displaystyle\bigcup_{n\in \mathbb{N}}\mathbb{R}^n$, it follows that $E_n:=\mathbb{E}\cap \mathbb{R}^n$ is a vector subspace of $\mathbb{R}^n$. 
Therefore, we have   $ \left( \Delta_M\cap TM_n \right) _{| \mathbb{U}_n}=\mathbb{U}_n\times \mathbb{E}_n$. This implies that $ \Delta_M\cap TM_n$ is a subbundle $\Delta_{M_n}$ of $TM_n$
and so  $\Delta_M=\underrightarrow{\lim}\Delta_{M_n}$. By the way,  we finally obtain  that $\Delta_{T^\prime M}=\underrightarrow{\lim} \Delta_{T^\prime M_n}$.\\
 The assertion 3 is an application of Theorem \ref{T_DirctLimitSequenceDeltaTMn}.
\end{proof} 

\subsection{Direct limit of sub-Riemannian Lie groups}
\label{___DirectLimitOfSubRieamnnianLieGroups}

\subsubsection{Sub-Riemannian Lie groups}
\label{____SubRiemannianLieGroups}
\emph{In this paragraph, we refer to \cite{L-Don21}, Chap.~5 or \cite{ABB19}}.\\

Let ${G}$ be a Lie group of dimension $n$ and $\mathfrak{g}$ its Lie algebra. 
A \emph{left-invariant distribution}\index{left-invariant!distribution} on $G$ is a subbundle $\Delta_G$ on $G$ which satisfies, for all $g$ and $h$ in $G$:
\[
\Delta_G (gh)=TL_g \left( \Delta_G (h) \right) 
\]
In fact, such a distribution is well defined by the subspace $V:=\Delta_G (1_G)$ of $T_{1_G}G\equiv \mathfrak{g}$. 
Let $m$ be the rank of $\Delta_G$.\\
By induction, we define
$$V^{(1)}:=V,\:\; V^{(k)}:=V^{(k-1)}+[V, V^{(k-1)}],\; \forall k>1$$
where $[V, V^{(k-1)}]=\{[X,Y], \; X\in V,\; Y\in  V^{(k-1)}\}$.\\
 Note that by left translation $V^{(k)}$ defines a left invariant subbundle $\Delta_G^{(k)}$ of $TG$. The Lie algebra $\mathsf{Lie}(V)$ is $\displaystyle\bigcup_{k\in \N}V^{(k)}$ and the associated left-invariant  distribution is denoted $\mathsf{Lie}(\Delta_G)$which is a well defined subbundle of $TG$.\\
We then have the following result: 

\begin{proposition}
\label{P_Dichotomie}
If $m<n$, we have the following dichotomy:
\begin{enumerate}
\item[(i)] 
either $V^{(n-1)}=\mathfrak{g}$ and then the induced is called {\it bracket generating} with step less than $n$
\item[(ii)] 
or $V^{(n-1)}\not=\mathfrak{g}$ and in fact there exists a Lie subgroup $H$ of $G$, $H\not=G$, $\mathsf{dim} H<n$ such that  $(\Delta_G)_{| H}$ is contained in $TH$ and is bracket generating in $H$. Moreover if $\bar{k}$ is the smallest integer for which $V^{(k)}= V^{(k-1)}$ then either $V^{(\bar{k})}=\mathfrak{g}$ or $\bar{k}<n$.
\end{enumerate}
\end{proposition}

{\bf From now on until the end of this subsubsection, we will assume that $\Delta_G$ is bracket generating}.\\

A curve $c:[a,b]\to G$ is \emph{horizontal}\index{horizontal curve}  if $\dot{g}(t)$ belongs to $\Delta_G (c(t))$,  for all $t\in [a,b]$. Any  curve $c:[a,b]\to G$ of class $C^k$, $k\geq 2$ is horizontal if and only if there exists a curve $\gamma:[a,b]\to \mathfrak{g}$ of class $C^k$ such that $c$ is solution of the differential equation:

\begin{equation}
\label{eq_Horizontal}
\begin{cases}
& \dot{c}(t) = (L_{c(t)})_*(\gamma(t)) \\
& c(a) = g
\end{cases}
\end{equation}
 
\smallskip
 
A \text{left invariant  sub-Riemannian metric} on $G$ is  a scalar product 
$<.,.>_g$ on $\Delta_G(g)$  which is  left-invariant, i.e., if $u = {L_g}_* \bar{u}$ and $v ={ L_g}_* \bar{v}$
with $\bar{u}, \bar{v}\in V$, we have 
\[
<u,w>_g=<\bar{u},\bar{v}>_{1_G}.
\]
We will set $||v||_g=\sqrt{<v,v>_g}$.

By the way, we can find an orthonormal basis $ \left( e_1,\dots,e_m \right) $ of $V$ and so we  have
\[
\Delta_G(g)=\left \{\sum_{i=1}^m v_i (L_g)_* e_i,\;\; v_i\in \mathbb{R},\; i=1,\dots, m.\right\}
\]
Recall that the length of a curve $c:[0,T]\to G$ is then $\mathcal{L}(c)= \displaystyle \int_0^T  ||\dot{c}(t)|| dt $ and it energy is 
$\mathcal{E}(c)=\displaystyle\frac{1}{2} \int_0^t  ||\dot{c}(t)||^2 dt$. 

 According to the previous notations, the problem of finding the shortest curve connecting two points $g_1$ and $g_2$ of $G$ can then be formulated as the optimal control problem:
\begin{equation}
\label{Eq_MinimizerinV}
\begin{cases}
& \dot{{c}(t)}=\displaystyle\sum_{i=1}^m v_i(t) (L_g)_* e_i,\;\; v_i(t)\in \R,\; i \in \{1,\dots, m\}\\
& \underset{{\footnotesize c(0)=g_1,\; c(T)=g_2}}{\min}\displaystyle \int_0^T  ||\dot{c}(t)||_{c(t)}dt
\end{cases}
\end{equation}

According to our  assumption on bracket generating, there always exists a solution of this problem which is called a minimizer (cf. \cite{ABB19} or \cite{L-Don21}).\\

It is classical that minimizing the length is the same as minimizing the energy and a minimizer $c$ of $\mathcal{E}$ such that $c(0)=g$ is a critical point of the map
$\mathcal{E}$ from the Banach manifold $C^k_{\Delta_G}(g, [0,T], G)$ to $\mathbb{R}$ (cf \cite{ABB19} \cite{L-Don21}).  Therefore, to the optimal problem (\ref{Eq_MinimizerinV}) is naturally associated  the  following  Lagrangian constraint problem:

 We can extend the inner product on $\Delta_G$ to an inner product again denoted $<\;,\;>_g$  and also its associated norm $||\;||_g$.\\
Thus  we consider the Lagrangian $\mathbf{L}(g,v)=\displaystyle\frac{1}{2}||v||^2_g$ on $TG$.  Such a Lagrangian is not degenerate on $\Delta_G$ and we have $\mathbb{F}\mathbf{L}(g,v)=<v,\;>_g$ for $v\in \Delta_G$. Then $\mathbb{P}$is a subbundle of $T^\prime G$ isomorphic to $\Delta_G$. The associated Hamiltonian is $H(g,v,p)=<p,v>-\displaystyle\frac{1}{2}||v||^2_g$.\\
In particular, the restriction of $H$ to $\mathbb{P}$ is
$h:=H_\mathbb{P}(g,v)= \displaystyle\frac{1}{2}||v||^2_g $.\\
The associated induced  implicit Hamiltonian $(M,D_{\Delta_M},H_{| T^\prime M})$ can be written   (cf. (\ref{eq_MHPDeltaM}))
\begin{equation}\label{eq_ImplicitHamiltonianinftyg}
  \dot{{g}}=\displaystyle{\frac{1}{2}\frac{\partial h}{\partial {p}}(g,p)},\;\;\dot{\mathsf{p}}=-\displaystyle\frac{1}{2}\frac{\partial h}{\partial {g}}(g,{p})\;\;\;\;
\end{equation}

\begin{remark} 
A solution of such a system is always a normal curve according to the Pontryagin Maximum Principle (cf. $\S$~ \ref{____PMPFiniteDimensionalContext}).
\end{remark}

\subsubsection{Extremals in Direct Limits of Sub-Riemannian Lie Groups}
\label{____ExtremalsInDirectLimitsOfSubRiemannianLieGroups}

Consider an ascending  sequence  of Lie groups $G_n$ of dimension $n$, provided with an ascending  sequence of left invariant distributions $\Delta_{G_n}$ and a left invariant sub-Riemannian  metric $<\;,\;>^n$. \\

Such  data will be denoted $ \left( G_n, \Delta_{G_n}, <\;,\;;>^n \right) $ and called an \emph{ascending sequence of sub-Riemannian Lie groups}\index{ascending sequence!of sub-Riemannian Lie groups}.\\

From properties of Direct limits (cf \cite{CaPe23}),  it follows that $G=\underrightarrow{\lim}G_n$ is a convenient regular Lie group, $\Delta_G=\underrightarrow{\lim}\Delta_{G_n}$ is a left invariant distribution on $G$ and $<\;,\;>_g:=\underrightarrow{\lim}<\;,\;>^n_{g_n}$ is a left invariant sub-Riemannian metric on $G$ and $\Delta_G$ is bracket generating. \\

By application of Theorem~\ref{T_DirctLimitSequenceDeltaTMn} , normal curves associated to this direct limit structure always exists  and are direct limit of normal curves  associated to the Sub-Riemannian structure on $G_n$ and abnormal curves  are also direct limits of abnormal curves  associated to the sub-Riemannian structure on $G_n$. In particular {\bf there does not exist elusive curves}.\\

As an application of this context, we can consider $G_n=\operatorname{Gl}(n)$.  Such  ascending sequence of sub-Riemannian structures gives rise to a sub-Riemannian structure on the direct limit\footnote{cf. \cite{KrMi97}, 47.7, for basic properties of $\operatorname{Gl}(\infty)$.} $\operatorname{Gl}(\infty)$. The same is true for any ascending sequence of Lie sub-groups of $\operatorname{Gl}(n)$.\\
Conversely, from Proposition~\ref{P_DeltaM}, any sub-Riemannian structure on $\operatorname{Gl}(\infty)$ induced a sub-Riemannian structure on $\operatorname{Gl}(\infty)$.\\
In some sense, we can consider  $\operatorname{Gl}(\infty)$ as a kind of "Classifiant"  for sub-Riemannian structures on  $\operatorname{Gl}(n)$ for  any $n\in \mathbb{N}$.

\subsection{Direct Limits of  Electric  Circuits}
\label{___DirectLimitsOfElectricCircuits}

\subsubsection{Electric circuits} 
\label{____Electric Circuits}

\emph{We first recall the formalism of constraint Lagrangian on LC circuits and  for this section, we refer to  \cite{O-BTCOM13} and \cite{YoMa06I}.}\\ 

For a given electric circuit graph,  Kirchhoff's current law may be written in a manifold  $Q$ of dimension $n$, called the \emph{set of charges}\index{set!of charges},  with a finite set of $1$-forms 
$ \left( \omega^a \right) _{1 \leqslant a \leqslant n-m} $ ($m<n$) as
\[
\left\langle \omega^a,f \right\rangle = 0
\]
where $f =  \left( f^1,\dots,f^n \right) \in T_q Q $ corresponds to the \emph{set of currents}\index{set!of currents} associated to the branches of the circuit and the $1$-forms $\omega^a$ are given, using Einstein's notation, by:
\[
\omega^a = \omega_k^a dq^k~~~~~
a \in \{1,\dots,n-m\},\; k \in \{1,\dots,n\} 
\]
where the coefficients $\omega_k^a$, equal $\pm 1$ or $0$, are determined
by the given circuit.\\
The set of all branch currents $f =  \left( f^1,\dots,f^n \right) $ which satisfy the KCL forms is an $m$ dimensional subspace $\Delta_Q(q)$ of $T_q Q$ defined by:
\[
\Delta_Q(q)
= \left\lbrace
f \in T_q Q:\;
\forall a \in \{1,\dots, n-m\}, 
\left\langle \omega^a,f \right\rangle =0
\right\rbrace
\]
called \emph{the Kirchoff Current Law (KCL) }\index{KCL (Kirchoff Current Law)}.\\
Let $\Delta_Q^{\operatorname{o}}(q) \subset T^\ast_q Q$ be the annihilator of $\Delta_Q(q)$, which  is spanned by $m$   $1$-forms $\omega^1,\dots,\omega^m$ representing the \emph{Kirchoff Voltage Law}(KVL)\index{KVL (Kirchoff Voltage Law}. We will assume that these forms are independent.\\
Let $T_{\pi_Q}: TT^\ast Q \to TQ$ be the tangent map of the cotangent bundle projection $\pi_Q : T^\ast Q \to Q$.\\
Define the distribution $\Delta_{T^\ast Q}$ on $T^\ast Q$ by lifting the constrained distribution $\Delta_Q \subset TQ$:
\[
\Delta_{T^\ast Q} 
:= \left( T \pi_Q \right) ^{-1} (\Delta_Q) \subset T T^\ast Q.
\] 
The annihilator $\Delta_{T^\ast Q}^{\operatorname{o}}$ of $\Delta_{T^\ast Q}$ is given for each $z=(q,p) \in T^\ast Q$ by
\[
\Delta_{T^\ast Q}^{\operatorname{o}}(z)
=
\left\lbrace
\alpha_z\in T^\ast_z T^\ast Q:\;
\forall w_z \in \Delta_{T^\ast Q},\,
\left\langle \alpha_q,w_q \right\rangle = 0
\right\rbrace.
\]
Consider the bundle morphism $\Omega^\flat : T T^\ast M \to T^\ast T^\ast M$  associated to the Liouville $2$-form $\Omega$ (canonical  symplectic structure on $T^\ast Q$). Then a Dirac structure $D_{\Delta_Q}$ on $T^\ast Q$, which is induced from the Kirchoff distribution $\Delta_Q \subset TQ$ can be defined, for each $z=(q,p) \in T^\ast Q$ by
\[
D_{\Delta_Q} (z)
=
\left\lbrace
 \left( v_z,\alpha_z \right) 
 \in T_z T^\ast Q \times T_z^\ast T^\ast Q:\;
 v_z \in \Delta_{T^\ast Q} (z),\;
 \alpha_z - \Omega^\flat(z) (v_z) 
\in \Delta_{T^\ast Q}^{\operatorname{o}}(z) 
\right\rbrace .
\]
For simplicity, from now on, we assume that $Q$ is an open  submanifold of $\mathbb{R}^n$ and we will use the usual coordinates in $\mathbb{R}^n$.  By the way, consider a subbundle $\Delta_Q$ of  $TQ$   and $\Delta_Q^0$ its annihilator.\\
We consider the  Lagrangian given by
\[
 \mathbf{L}(q,v)=\displaystyle\frac{1}{2} \left( v^\top\mathbf{L}v- q^\top \mathbf{C} q \right) \footnote{We can also consider more general Lagrangians as in \cite{O-BTCOM13}.}\]
 where $\mathbf{L}=\mathsf{diag} \left( \mathsf{L}_1,\dots, \mathsf{L}_n \right) $ and $\mathbf{C}=\mathsf{diag} \left( \displaystyle\frac{1}{\mathsf{C}_1},\dots, \frac{1}{\mathsf{C}_n} \right) $ are constant  diagonal matrices. 
 
Therefore,  
$\mathbb{F}\mathbf{L}(q,v)
= \left( q,\displaystyle\frac{\partial \mathbf{L}}{\partial v} \right)
 =(q,\mathbf{ L}v)=(q,p)$.\\     
Again, we set  $\mathbb{P}=\mathbb{F}\mathbf{L}(\Delta_Q)$.
  Thus,  if $\mathbb{E}$ is the typical fibre of $\Delta_Q$, the evolution for the general associated  non holonomic system for a Kirchoff  circuits   in local coordinates is  then:

\begin{equation}
\label{Eq_ImplicitLagrangianLC}
\begin{cases}
\dot{\mathsf{q}}
=\mathsf{v}\\
\dot{\mathsf{p}}+ \mathsf{C}\mathsf{q} \in \mathbb{E}^0 \\
\mathsf{p}
= \mathsf{L}\mathsf{v}\\
\mathsf{v}\in \mathbb{E}
\end{cases}
\end{equation}

We assume that each $\mathsf{L}_i>0$  and $\mathsf{C_i}>0$ for $i \in \{1,\dots, n\}$,  so the Legendre transformation is non degenerate  and so $\mathbb{P}$ is a submanifold of $T^*Q$ (cf. \cite{O-BTCOM13}).

\begin{remark}
\label{R_NormalAbmormal} 
Note that each integral curve of this system with initial condition $f^0:=(f_1^0,\dots, f_n^0)\in Q$ stays in the reachable set $\mathcal{R}(f^0)$ and {\it any integral curve of the previous dynamical system in restriction to the immersed submanifold $\mathcal{R}(f^0)$ is a normal bi-extremal. However, it is an abnormal bi-extremal if and only if $\mathcal{R}(f^0)\not=Q$}.\\
\end{remark}

\subsubsection{Direct Limits of LC Electric Circuits} 
 We begin by fixing some data:
\begin{description}
\item[--]  an increasing sequence of integers $\left( m_n \right) $ such that $0<m_n\leq n$;
\item[--]   an infinite sequence $ \left( \mathsf{L}_n \right) $ and $ \left( \mathsf{C}_n \right) $ of strictly positive number real numbers;
\item[--] an ascending sequence $Q_n$ of  open submanifolds of $\mathbb{R}^n$;
\item[--] an ascending sequence $\Delta_{Q_n}$ of  subbundles of $TQ_n$ of rank $m_n$.
\end{description}
For each $n$, we set $\mathbf{L}_n=\mathsf{diag} \left( \mathsf{L}_1,\dots,\mathsf{L}_n \right) $ and $\mathbf{C}=\mathsf{diag} \left( \displaystyle\frac{1}{\mathsf{C}_1},\dots,\frac{1}{\mathsf{C}_n} \right) $.\\
 
 By the way, for such given data  for each fixed sequence $m_n$,   ascending    sequence $Q_n$ of open  submanifolds of $\mathbb{R}^n$  and ascending sequence $\Delta_{Q_n}$ of rank $m_n$, we set:
$Q:=\underrightarrow{\lim}Q_n$, $\Delta_Q=\underrightarrow{\lim}\Delta_{Q_n}$, $\mathbf{L}=\underrightarrow{\lim}\mathbf{L}_n$. Then, by application of Theorem~\ref{T_DirctLimitSequenceDeltaTMn},  we obtain an  implicit Lagrangian system  $ \left( Q, \Delta_Q,\mathfrak{D} \mathbf{L})  \right) $ which satisfies the conclusions of this Theorem. \\
Conversely, any implicit Lagrangian system  $ \left( Q, \Delta_Q,\mathfrak{D} \mathbf{L}) \right) $ induces on each  $Q_n$ a Lagrangian system  $ \left( Q_n, \Delta_{Q_n},\mathfrak{D}_n \mathbf{L}_n \right) $  from Proposition~\ref{P_DeltaM}.

\section{Dirac Structures on Singular Constraints and Dynamical Systems in the Banach Setting}
\label{__DiracStructuresOnSingularConstraintsInTheBanachSetting}

{\bf In this section, we restrict ourself to the Banach setting.}\\

\emph{The first purpose of this section is to present  a generalization  to a singular distribution on $M$ of  the induced  partial almost Dirac structure associated  to  a closed subbundle $\Delta_M$ of $TM$. More precisely, we   consider an anchored Banach bundle $(E,M,\rho)$
where $\pi:E\to M$ is a  vector bundle whose typical fibre is a reflexive Banach space $\mathbb{E}$ and  an open set $E_0$ of $E$ such that the restriction $\pi_0$ of $\pi$ to $ E_0$ is a locally trivial bundle over $M$ whose  typical fibre $\mathbb{E}_0$ is an open set of the typical 
fibre $\mathbb{E}$ of $E$. We denote by $\rho_0$ the restriction of $\rho $ to $E_0$. In this context, we will  say that 
$\Sigma_M:=\rho(E_0)$ is a \emph{singular non holonomic constraint}. Then, we will  describe dynamical systems associated to a non-degenerate Lagrangian on $E_0$ using a notion of \emph{singular partial almost Dirac structure}, as defined in $\S$~\ref{___InducedPartialAlmostDiracStructures} and an adapted  formalism as developed in $\S$~\ref{__ImplicitLagrangians} and $\S$~\ref{__VariationalApproach}.}  

\subsection{Smooth Singular Distributions on a Manifold}
\label{___SmoothDistribution}

In this paragraph, we recall the definition of a singular smooth distribution on a  manifold and we refer to \cite{CaPe23}, 8.1 and 8.2.
\begin{definition}
\label{D_WeakDistribution}
A  \emph{weak distribution}\index{weak distribution}\index{distribution!weak} on a  manifold $M$  is a map $ \Delta: x \mapsto \Delta_x$  which, for every $x\in M$, $\Delta_x$ has its own Banach structure and such that the inclusion of $\Delta_x$ in $T_xM$ is bounded for all $x\in M$.\\
A weak distribution is \emph{closed}  if each vector space  $\Delta_x$ is a closed    subspace of  in $T_xM$.
\end{definition} 

\begin{example}
\label{Ex_Anchored bundle}
Consider a  bundle  $\pi :E\to M$ and a morphism $\rho: E\to TM$. In this situation, we say that  $(E,M, \rho)$ is an \emph{anchored bundle} and $\rho$ is called the \emph{anchor}\index{anchor}. Then  
$ \Delta_x=\rho_x(E_x)$ is a weak distribution  on $M$ which is closed if and only if the range $\Delta_x$ of $ \rho_x$ is a closed subspace of $T_xM$.
\end{example}

\begin{definition}
\label{D_LocalVectorFieldTangentToAWeakDistribution}
A \emph{local vector field} $Z\in\mathfrak{X}_{\text{loc}}(M)$ is \emph{tangent to the  closed  distribution} $\Delta$, if for any $x\in \operatorname{Dom}(Z)$\footnote{$\operatorname{Dom}(Z)$ denotes the open set on which $Z$ is defined.}, $Z(x)$ belongs to $\Delta_x$.
\end{definition}

The set of local vector fields tangent to $\Delta$ will be denoted by $\mathcal{X}_{\Delta}$.

\begin{definition}
\label{D_DistributionGeneratedByASubset}
We say that a weak closed distribution $\Delta$ is \emph{generated by a subset  $\mathcal{X}\subset \mathfrak{X}_{\text{loc}}(M)$} if, for every $x\in M$, the vector space $\Delta_x$ is the  linear hull  of the set
\[
\left\lbrace 
Y(x),Y \in \mathcal{X},x\in \operatorname{Dom}(Y) 
\right\rbrace.
\]
\end{definition}

\begin{definition} \label{D_DistributionLoctrivial}
Let $\Delta$ be a closed distribution on a Banach manifold $M$.\\
We say that $\Delta$ is \emph{locally trivial}\index{locally trivial distribution}\index{distribution!locally trivial} if for each $x\in M$,  there exist an open neighbourhood $U$ of $x$  and a smooth map $\Theta:U\times {\Delta}_x  \to TM$, called \emph{local trivialization}\index{local trivialization}, such that:
\begin{enumerate}
\item[(i)]
$ \forall y\in U$, $\Theta(\{y\}\times{\Delta}_x)\subset {\Delta}_y$;
\item[(ii)]
$\forall y\in U$,  $\Theta_y\equiv \Theta(y,.):{\Delta}_x\to T_yM$ is a continuous operator  and $\Theta_x:{\Delta}_x\to T_xM$  is the natural inclusion $i_x$. 
\item [(iii)]
$\Theta_y: {\Delta}_x\to {\Delta}_y$  is an isomorphism from ${\Delta}_x$ onto ${\Theta}_y({\Delta}_x)$.\\
A  closed locally trivial distribution will be called {\it a smooth distribution}.
\end{enumerate}
\end{definition}

From \cite{CaPe23}, Proposition 8.1,  
%Chapter 8 p 358  
we have the following  sufficient conditions of local triviality of a closed distribution:
 
\begin{proposition}
\label{P_Uppertrivial}
\label{P_UpperLocalTriviality}
Let $\Delta$ be a closed distribution and assume  that for each $x\in M$, there exists an open neighbourhood $V$ of $x$, a Banach space $F$ and  a smooth map $\Psi:U\times F \to TM$   such that for each $y\in U$:
\begin{enumerate}
\item[(i)]
$ \Psi(\{y\}\times F)= \Delta_y$;
\item[(ii)]
$\Psi_y\equiv \Psi(y,.):F\to T_yM$ is a bounded  operator such that $\Psi_y(F)=\Delta_y$;
\item[(iii)] $\ker \Psi_x$ is split  and we consider a decomposition $F=\ker \Psi_x\oplus E$.
\end{enumerate}
 Then there exists an open neighbourhood $W$ of $x$ such that the restriction   of $\Psi_y$ to $E$  is injective for any $y\in W$,  and then ${\Theta}(y,v)={\Psi}_y\circ  {\Psi}_x^{-1}(v)$ is  a local trivialization of $\Delta$.
\end{proposition}

%Note if $(E, M,\rho)$ is an anchored bundle,  from  Example \ref{Ex_Anchored bundle} and Proposition \ref{P_Uppertrivial} we have:\\
%\begin{corollary}\label{C_anchoredBundle} (cf. \cite{CaPe23}  Corollary 8.1)
%Let $(E,M,\rho)$ be an anchored bundle   such that $\ker \rho_x$  is split and $\rho_x$ has closed range.   Then the distribution  $\mathcal{D}=\mathbf{Im}\rho$ is   a smooth  distribution on $M$.\\
%\end{corollary}

Let $\Delta$ be a smooth distribution on $M$. For any  local trivialization $\Theta:V\times{\Delta}_x\to TM$
and any $v \in {\Delta}_x$, we  consider
\begin{equation}
\label{eq_LocalSectionAssociatedToALowerTrivialization}
X(z)=\Theta(z,v).
\end{equation}
Of course, $X$ is a local vector field on $M$ tangent to $\Delta$  whose domain is $V$. Moreover, the set of all such local vector fields spans $\Delta$.

\begin{definition}
A \emph{(local) lower section for  a  locally trivial closed  distribution} $\Delta$ is  a map of type (\ref{eq_LocalSectionAssociatedToALowerTrivialization}) associated to  some local trivialization $\Theta$.
\end{definition}
 
The set of all  (local) lower sections  of type (\ref{eq_LocalSectionAssociatedToALowerTrivialization}) associated to all  local trivializations will be denoted by $\mathcal{X}^-_{\Delta}$.\\
 
\subsection{Singular Induced Partial Almost Dirac Structures}
\label{___SingularInducedPartialAlmostDiracStructures}
  
In this subsection, we present a generalization of induced  partial almost Dirac structure  associated to to a  regular  distribution presented in section \ref{___InducedPartialAlmostDiracStructures}..\\ 

Very  precisely, consider the following context:\\

\noindent 
Given  an anchored Banach bundle $ (E,M,\rho)$ whose typical fibre $\mathbb{E}$ is reflexive, the fibre $\mathbb{E}_0$ is conic (cf. Remark~\ref{R_conicfiber}).\\
{\bf  We assume that $\ker \rho_x$ is split in $E_x=\pi^{-1}(x)$  and $\operatorname{Range} \rho_x$ is a closed  subspace of $T_xM$}.\\

\noindent
These assumptions are always satisfied if $\rho_x$ is a Fredholm operator for all $x\in M$ (cf. \cite{PeCa24},  Examples A.12 (4)).\\ 
 % be an open an open set of $E$ such that the restriction $\pi_0$ of $\pi$ to $ E_0$ is a locally trivial bundle over $M$ whose  typical fiber $\mathbb{E}_0$ is an open set of the typical 
%fiber $\mathbb{E}$ of $E$. We denote by $\rho_0$ the restriction of $\rho $ to $E_0$. In this context we will  say that {\it $\Sigma_M:=\rho(E_0)$ is a singular non holonomic constraint}.\\
Under these assumptions, note that $\Delta_M$ is a smooth  distribution (cf. Proposition~\ref{P_Uppertrivial}) but is not a subbundle of $TM$ in general.\\  

%For each $x\in M$ %we have  $\Sigma_M\cap(\Delta_M)_x=\rho(\pi_0^{-1}(x))$. In particular when $E_0=E$  is a closed subbundle of $TM$ we will  recover the non holonomic constraint in the sense of section. \ref{___InducedPartialAlmostDiracStructures}. \\
 
\emph{We will now define the induced singular  partial almost Dirac on   $T^\prime M$ associated to $\Delta_M$.}\\
 
But before we need to introduce some notations:
 
\begin{notations}
\label{N_LocalCoordinatesE}
Consider a chart $(U,\phi)$ in $M$ such that $E_{| U}$ is trivializable. If $\mathbb{U}=\phi(U)$, we have  a trivialization $\Psi:E_{| U}\to \mathbb{U}\times \mathbb{E}$. Note that the restriction $\Psi_0$ to $(E_0)_{| U}$ is a trivialization onto $\mathbb{U}\times \mathbb{E}_0$. Again in these local coordinates, for  $(x,v)\in E_{| U}$ the representation $\Psi(x,v)$ will be again  denoted $(x,v)\cong(\mathsf{x},\mathsf{v})$. Since $T\Psi$ gives rise to a trivialization of $T(E_{| U})$ onto $\mathbb{U}\times \mathbb{E}\times  \mathbb{M}\times \mathbb{E}$, the associated  local coordinates will be denoted $(\mathsf{x},\mathsf{v},\dot{\mathsf{x}},\dot{\mathsf{v}})$.  We also have a trivialization $T\phi:TM_{| U}\to \mathbb{U}\times \mathbb{M}$ such that
 $\rho(x,v):=\rho_x(v)\cong \mathsf{r}_\mathsf{x}(\mathsf{v})$. \\
 
 On the other hand, let $\pi ^\prime:E^\prime\to M$ the dual bundle of $\pi:E\to M$. As previously, given  a chart $(U,\phi)$ in $M$ such that $E_{| U}$ is trivializable. Then $E^\prime$ is also 
 trivializable and $(\Psi^*)^{-1}:E^\prime_{| U}\to \mathbb{U}\times \mathbb{E}^\prime$ is a trivialization\footnote{Following the convention in \cite{KrMi97}, the adjoint of a morphism $\Psi$ between two Banach bundles  will be denoted $\Psi^*$.}. The representation $(\Psi^*)^{-1}(x, \mu)$ will be denoted $(\mathsf{x},\mathsf{m})$. Also $T(\Psi^*)^{-1}$ will define a trivialization 
 $T  E^\prime\to \mathbb{U}\times \mathbb{E}^\prime\times \mathbb{M} \times E^\prime $    whose associated local coordinates will be denoted $(\mathsf{x},\mathsf{m}, \dot{\mathsf{x}}, \dot{\mathsf{m}})$. %$T^\prime  E^\prime\to \mathbb{U}\times \mathbb{E}^\prime\times \mathbb{E}^\prime \times E^\prime$
\end{notations}

Given any $x\in M$,  we have $(\Delta_M)_x\cong\{(\mathsf{x},\mathsf{r}_\mathsf{x}(\mathsf{v}),\; \mathsf{v} \in \mathbb{E}\}$. Now for every $p\in T_x^\prime M$ we set 
 
 $(\Delta_{T^\prime M})_{(x, p)}:=(Tp_{T^\prime M})^{-1}((\Delta_M)_x)$ . 
 
% \Deta_{\Sigma_M}:= Tp_{T^\prime M})^{-1}((\Sigma_M)_x)$ .
 
 \noindent Then $(\Delta_{T^\prime M})_{(x, p)}\cong \{(\mathsf{x},\mathsf{p}, \dot{x}=\mathsf{r}_\mathsf{x}(\mathsf{v}), \dot{\mathsf{p}},\; \mathsf{u}\in \mathbb{E},\;  \dot{\mathsf{p}}\in \mathbb{M}^\prime\}$. 
 
Considering the weak canonical symplectic form  $\Omega$ on  $T^\prime M$,  as in $\S$~\ref{___InducedPartialAlmostDiracStructures}, to any $(x,p)\in T^\prime M$, we  associate  the following  vector space  
$(D_{\Delta_M})_{(x,p)}$  of $T^{\mathfrak{p}}_{(x,p)} T^\prime M$: 
\[
(D_{\Delta_M})_{(x,p)}=
\left\lbrace
(\dot{x}, \dot{p},\alpha)\in T^\mathfrak{p}_{(x,p)}(T^\prime M):\;
  \dot{x}=\rho_x(v), v \in E_x,\;\;
   \alpha- \Omega_{(x,p)}^\flat (v)\in (\Delta_{T^\prime M}^0)_{(x,p)}  
\right\rbrace .
\]
From the proof of Propostion~\ref{L_CloseneesOfOmegaDelta}, we obtain:

\begin{proposition}
\label{P_DT'M}      
$(D_{\Delta_M})_{(x,p)}=(\Delta_{T^\prime M})_{(x, p)}\oplus \Omega^\flat((\Delta_{T^\prime M})_{(x,p)})$ and so is a partial linear Dirac structure on $T_{(x,p)}(T^\prime M)$.
\end{proposition}

As in Remark~\ref{R_DOmegaM} (1), we have:
\begin{equation}
\label{eq_SingularDiracLocal}
(D_{\Delta_M})_{(x,p)}\cong \{\left((\mathsf{x},\mathsf{p}, \mathsf{r}_{\mathsf{x}}(\mathsf{u}), \dot{\mathsf{p}}),(\mathsf{x},\mathsf{p}, - \dot{\mathsf{p}}, \mathsf{r}_{\mathsf{x}}(\mathsf{u})\right),\; \mathsf{u}\in \mathbb{E},\; \dot{\mathsf{p}}\in \mathbb{M}^\prime\}
\end{equation}

\begin{definition}
\label{D_SingularDirac}  
The  map $D_{\Delta_M}:(x,p)\mapsto (D_{\Delta_M})_{(x,p)}\subset T_{(x,p)}T^\mathfrak{p}M$ is called \emph{the singular Dirac structure associated to $\Delta_M$}.
\end{definition}

Note that if $E$ is a closed subbunbdle of $TM$,  with the anchor being the inclusion of $E$ into $TM$, then $E=\Delta_M$ and the previous  construction of $D_{\Delta_M}$ is exactly the induced partial almost Dirac structure associated  to $\Delta_M$.\\ By analogy with Definition~\ref{D_DistributionLoctrivial}, we have:
\begin{proposition}
\label{D_SmoothSingularDirac}
For any $x\in M$  and $(x,p)\in T_x^\prime M$ fixed, there exist  an open set $U$ around $x\in M$  and a smooth map
%such that for any and $(x,p)\in T_x^\prime M$  in $T^\prime M$ and a smooth map

$\Theta^\mathfrak{p}: T^\prime M_{| U}\times  (D_{\Delta_M})_{(x,p)}\to T^\mathfrak{p}M$ such that :
\begin{enumerate}
\item[(i)]
$ \forall (x',p')\in T^\prime M_{| U} ,\;\Theta^\mathfrak{p}(\{(x',p')\}\times(D_{\Delta_M})_{(x,p)})\subset (D_{\Delta_M})_{(x',p')}$;
\item[(ii)]
For each $(x',p')\in  T^\prime M_{| U}$, 
 $$\Theta_{(x',p')}^\mathfrak{p}:= \Theta^\mathfrak{p}((x',p'),.):(D_{\Delta_M})_{(x,p)}\to T_{(x',p')}^\mathfrak{p}M$$
  is a continuous operator  and $\Theta^\mathfrak{p}_{(x,p)}:(D_{\Delta_M})_{(x,p)}\to T_{(x,p)}^\mathfrak{p}M$  is the natural inclusion   
\item [(iii)]
$\Theta_{(x',p')}^\mathfrak{p}:(D_{\Delta_M})_{(x,p)}\to (D_{\Delta_M})_{(x',p')}$  is an isomorphism  onto its range.\\
%A  closed locally trivial distribution will be called {\it a smooth distribution}.
\end{enumerate}
\end{proposition}  

\begin{definition}\label{D_SmoothAlmostPartialDirac} 
By abuse of vocabulary, we  will say that  the induced partial almost Dirac $D_{\Delta_M}$ is  a smooth distribution in $T^\mathfrak{p} M$.
\end{definition}

\begin{proof} 
We first consider the context of Definition~\ref{D_DistributionLoctrivial} that is the existence of a smooth map $\Theta: U\times (\Delta_M)_x\to TM$ which satisfies Properties (i) and (ii).

 At first from (i),  for any $x'\in U$ and $(x',p')\in T^\prime M_{| U}$, we have  
\begin{equation}
\label{eq_ThetaT1}
[(Tp_{T^\prime M})^{-1}(\Theta_{x'}((\Delta_M)_x))](x',p')\subset (\Delta_{T^\prime M})_{(x',p')}
\end{equation}
On the other hand, as the problem is local, without loss of generality,  we may assume that $M=U\subset \mathbb{M}$ and so we can use  notations used in  $\S$~\ref{___CanonicalMaps}  and Notations \ref{N_LocalCoordinatesE}.\\
For $(x,p)\in T^\prime M$ fixed,  we define $ \widehat{ \Theta}:T^\prime M_{| U} \times (\Delta_{T^\prime. M})_{(x,p)}\to T_{(x',p')}(T^\prime M)$ by
 
$\widehat{\Theta}((\mathsf{x}', \mathsf{p}'), (\dot{\mathsf{x}},\dot{\mathsf{p}}))=((\mathsf{x}',\mathsf{p}'),(\Theta(\mathsf{x}', \dot{\mathsf{x}}), \dot{\mathsf{p}}))$.\\
Then, for any $(\mathsf{x}',\mathsf{p}')\in T^\prime M$, 
$\widehat{\Theta}_{(\mathsf{x}',\mathsf{p}')}:=\widehat{\Theta}((\mathsf{x}', \mathsf{p}'), (\dot{\mathsf{x}},.) $ is a continuous operator and  from (\ref{eq_ThetaT1}), we have: 

\begin{equation}
\label{eq_T2i}
\widehat{\Theta}_{(\mathsf{x}',\mathsf{p}')}((\Delta_{T^\prime. M})_{(x,p)})\subset (\Delta_{T^\prime M})_{(x',p')}
\end{equation}
\begin{equation}
\label{eq_T2ii}
\widehat{\Theta}_{(\mathsf{x},\mathsf{p})} =\operatorname{Id}_{(\Delta_{T^\prime. M})_{(x,p)} }
\end{equation}
\begin{equation}
\label{eq_ThetaT2iii}
\widehat{\Theta}_{(\mathsf{x}',\mathsf{p}')} \textrm{ is an isomorphism } (\Delta_{T^\prime. M})_{(x,p)}\to \widehat{\Theta}_{(\mathsf{x}',\mathsf{p}')}( (\Delta_{T^\prime. M})_{(x,p)})\\
\end{equation}
 %Since  $\Delta_M$ is smooth distribution,  from Definition \ref.{D_DistributionLoctrivial}  for any $x\in M$ there exist a local tricialization $\Theta:U\times (\Delta_M)\to $TM$ such that $\Theta(y, (\Delta_M)_y^\parallell=\Delta_y:=\subset( \Delta_M)_y$ which  define a subbundle of $\Delta_M^\parallel$ of $TM_{| U}$. By construction of the singular partial almost Dirac structure, 

Now, from Proposition \ref{P_DT'M}, since     $(D_{\Delta_M})_{(x,p)}=(\Delta_{T^\prime M})_{(x, p)}\oplus \Omega^\flat((\Delta_{T^\prime M})_{(x,p)})$,  from (\ref{eq_ThetaT1}), we obtain:
\begin{equation}\label{eq_Tp1}
\widehat{\Theta}_{(x,p)}(\Delta_{T^\prime M})_{(x,p)})\oplus \Omega^\flat(\widehat{\Theta}_{(x,p)}(\Delta_{T^\prime M})_{(x,p)})\subset (\Delta_{T^\prime M})_{(x',p')}\oplus \Omega^\flat((\Delta_{T^\prime M})_{(x',p')})
\end{equation}

Thus, we can consider the map:

$\Theta^\mathfrak{p}: T^\prime M_{| U}\times  (D_{\Delta_M})_{(x,p)}\to T^\mathfrak{p}M$ defined by
\[
\Theta^\mathfrak{p}(((\mathsf{x}, \mathsf{p} ), (\dot{\mathsf{x}},\dot{\mathsf{p}}), \Omega_{(x,p)}^\flat(\dot{\mathsf{x}},\dot{\mathsf{p}})) =\widehat{\Theta}_{(x,p)}(\dot{\mathsf{x}},\dot{\mathsf{p}}),\Omega_{(x,p)}^\flat\circ \widehat{\Theta}_{(x,p)}(\dot{\mathsf{x}},
 \dot{\mathsf{p}}).
 \]
According to this Definition, using the relations (\ref{eq_T2i}), (\ref{eq_T2ii}) and (\ref{eq_ThetaT2iii}), we obtain the announced results.
\end{proof}
 
\subsection{Lagrangians and Associated Integrability Domains with Singular Constraints}
\label{LagrangiansAndAssociatedIntegrabilityDomainsWithSingularConstraints}
As announced in  the introduction  of this section, we consider an open set $E_0$ of $E$ such that the restriction $\pi_0$ of $\pi$ to $ E_0$ is a locally trivial bundle over $M$ whose  typical fibre $\mathbb{E}_0$ is a conic  open set of the typical 
fibre $\mathbb{E}$ of $E$. We denote by $\rho_0$ the restriction of $\rho $ to $E_0$.  In this way, $\Sigma_M:=\rho(E_0)$ is a singular non holonomic constraint.\\

Consider a smooth map $\mathbf{L}:E_0\to \mathbb{R}$. As in $\S$~\ref{__ImplicitLagrangians}, for $(x,v)\in E_0$, since  $\pi_0^{-1}(x)$ is conic, we can define 
%!o $v$ <- $w$
\begin{equation}
\label{eq_SingularLegendreTransform}
(\mathbb{F}\mathbf{L})_x(v).w
=
\displaystyle\frac{d}{ds} \left( \mathbf{L}(x,v+sw) \right)_{| s=0}
\end{equation}
where $v$ and $w$ belong to $\pi_0^{-1}(x)$ and $\mathbb{F}\mathbf{L}_x(v).w$ is the derivative of $\mathbf{L}$ along the vertical fibre of $TE_0\to E_0$  at point $(x,v)$, in the direction $w \in E_x$. Since $E_0$ is an open conic subbundle of $E$, 
by the way,  for $x\in M $, $\mathbb{F}\mathbf({L})_x$ defines  a continuous  linear map from $E_x \to E_x^\prime $\footnote{$E_x^\prime$ is the fibre of the dual bundle of $E^\prime $ of $E$.}. Since the map $\mathbf{L}$ is smooth on $E_0$ which is open in $E$ with conic fibre, 
 it follows that the induced map 
$\mathbb{F}\mathbf{L}:E_0\to E^\prime$ is a smooth map fibre preserving.\\
In local coordinates in $E$ and $E^\prime$, we have:
\[
\mathbb{F}\mathbf{L}\cong(\mathsf{x}, \mathsf{v})
\mapsto 
 \left( \mathsf{x},\mathsf{p}=\displaystyle\frac{\partial \mathsf{L}}{\partial \mathsf{v}} \right).
\]
  
Since $\mathbb{E}$ is reflexive, there no obstruction  that  $\mathbb{F}\mathbf{L}$ would be a local diffeomorphism. If it is true for any $(x,u)\in E_0$, we will say that \emph{$\mathbf{L}$ is regular}.\\
  
On the other hand, we consider the pull back $E\times_M T^\prime M$  of $E$ over $p_{T^\prime M}:T^\prime M\to M$. Then we have a  bundle morphism $\widetilde{p}_{T^\prime M}: E\times_M T^\prime M\to E$ which is  a fibrewise isomorphism and a projection $\widetilde{\pi}_{T^\prime M}: E\times_M T^\prime M\to T^\prime M$ and we have  the following commutative diagram:
\begin{equation}
\label{eq_PullBackE}
\xymatrix{
 E\times_M T^\prime M\ar[d]^{\widetilde{\pi}}\ar[r]^{\;\;\;\;\;\;\;\widetilde{p}_{T^\prime M}}&E\ar[d]^{\pi}\\
T^\prime M\ar[r]^{p_{T^\prime M}} &M}
\end{equation}
Since $E_0$ is an open subbundle of $E$, we  have the same diagram in restriction to $E_0$: 
 
\begin{equation}
\label{eq_PullBackE0}
\xymatrix{
 E_0\times_M T^\prime M\ar[d]^{\widetilde{\pi}_0}\ar[r]^{\;\;\;\;\;\;\;\;\widetilde{p}_{T^\prime M}}&E_0\ar[d]^{\pi}\\
T^\prime M\ar[r]^{p_{T^\prime M}} &M}
\end{equation}

\begin{lemma}
\label{L_ImFL}  
If $\mathbf{L}: E_0\to \mathbb{R}$ is regular, for any $x\in M$,  there exists a chart $(U,\phi)$ around $x$ and a trivialization $\Phi= E_{| U}\to \phi(U)\times \mathbb{E}$ such that, in the associated local coordinates, 
$(\mathsf{z},\mathsf{v})\mapsto \displaystyle\frac{\partial^2 \mathsf{L}}{\partial \mathsf{u}^2}(\mathsf{z},\mathsf{v})(\;,\;)$  is a  smooth field of bilinear forms on  $\mathbb{E}$, which is strongly non degenerated at $(\mathsf{x}\mathsf{,u})$.
\end{lemma}

\begin{proof}  
The result is a consequence of the following equivalent assertions:\\

(1)  In local coordinates, for fixed 
$x\in M$,  there exists a chart $(U,\phi)$ around $x$ such that 
$(\mathsf{z},\mathsf{v})\mapsto \displaystyle\frac{\partial^2 \mathsf{L}}{\partial \mathsf{u}^2}(\mathsf{z},\mathsf{v})(\;,\;)$  is a  smooth field of bilinear forms defined  on   an open set  around $(\mathsf{x},\mathsf{u})$ in  $\mathbb{U}\times \mathbb{E}$ where $\mathbb{U}=\phi(U)$. \\

Under the context (1), after shrinking 
$\mathbb{U}$, if necessary, there exists an open set $\mathbb{V}$ in $\mathbb{E}_0$ such that  $(\mathsf{z},\mathsf{v})\mapsto  \displaystyle\frac{\partial^2 \mathsf{L}}{\partial \mathsf{u}^2}(\mathsf{z},\mathsf{v})(\;,\;)$ is a field of  strong non-degenerated bilinear forms on $\mathbb{U}\times \mathbb{V}$ if and only if 
$(\mathsf{z},\mathsf{v})\mapsto \displaystyle\frac{\partial^2 \mathsf{L}}{\partial \mathsf{u}^2}(\mathsf{z},\mathsf{v})(\;,\;)$  is a field of strong non degenerated bilinear forms if and only if\\
$\displaystyle\frac{\partial^2 \mathsf{L}}{\partial \mathsf{u}^2}(\mathsf{x},\mathsf{u})(\;,\;)$ is strong non-degenerated.\\

Under context (1)  therefore,  in local coordinates, to  $T\mathbb{F}\mathbf{L}:TE_0\to T^\prime E^\prime$  is associated the matrix:
\begin{equation}
\label{eq_JacoobFL}
\begin{pmatrix}
			\mathsf{Id}&0\\
\displaystyle\frac{\partial^2 \mathsf{L}}{\partial \mathsf{x} \partial \mathsf{u}}&\displaystyle\frac{\partial^2 \mathsf{L}}{\partial \mathsf{u}^2}\\
			\end{pmatrix}
\end{equation}
Thus $T_{(x,u)}\mathbb{F}\mathbf{L}:T_{(x,u)}E_0\to T_{(x,u)}^\prime E^\prime$ is an isomorphism if and only if $\displaystyle\frac{\partial^2 \mathsf{L}}{\partial \mathsf{u}^2}(\mathsf{x,\mathsf{u}})(\;,\;)$  is a  smooth field of bilinear forms on  $\mathbb{E}$, which is strongly non degenerated .\\
  
From  the  inversion function theorem,  the restriction of $\mathbb{F}\mathbf{L}$ to ${E_0}_{| U}$ is a local diffeomorphism around $(x,u)$ if and only if  $T_{(x,u)}\mathbb{F}\mathbf{L}:T_{(x,u)}E_0\to T_{(x,u)}^\prime E^\prime$ is an isomorphism.
\end{proof}
  
\begin{proposition} 
\label{P_PDirac} 
Assume that $\mathbf{L}$ is regular. 
\begin{enumerate}
\item[1.]  
The set 
$\mathbb{P}:=\{(z, v, p) \in E_0\times T^\prime M\;:\; \rho^*_z(p)={\mathbb{F}\mathbf{L}}(z, v)\}$ is a split Banach  immersed submanifold of    $E_0\times _M T^\prime M   
$ modelled on $\mathbb{M}\times \mathbb{M}^\prime$. Moreover, the restriction $\Theta_\mathbf{L}$  of $\widetilde{\pi}_0$ to $\mathbb{P}$ is a local diffeomorphism from $\mathbb{P}$ into  $T^\prime  M$. It follows that  the restriction $\Theta_{\mathbf{L}}$ of $\widetilde{\pi}_0$ to $\mathbb{P}$ is  a surjective local diffeomorphism of into $T^\prime M$.
%and we denote by $\mathbb{P}$ the range of  $\widetilde{\mathbb{F}}\mathbf{L}$ in 
\item[2.] 
Let $\widetilde{\Omega}$ be the pull-back $\widetilde{\pi}_0^*\Omega$. Then $\widetilde{\Omega}$ is a strong symplectic $2$-form on $E_0\times_M T^\prime M$\footnote{cf. \cite{PeCa24}, 4.1.} whose kernel is $\ker T\widetilde{\pi}_0$ and whose supplement is  $\ker T\widetilde{p}_{T^\prime M}$ 
%and $\widetilde{\Omega}^\flat(T(E_0\times_M T^\prime M)$ is the pullback of $T^\flat(T (T^\prime M))$ over  $T\widetilde{\pi}_0$ and so
and the graph of $\widetilde{\Omega}$ is a partial Dirac structure $D_{\widetilde{\Omega}}$ 
on $E_0\times_M T^\prime M$. In particular,   the Dirac structure $D_\mathbb{P}$ induced on $\mathbb{P}$  is the graph of the (weak) symplectic $2$-form $\Omega_\mathbb{P}$ induced by $\widetilde{\Omega}$ on $\mathbb{P}$. In fact, $
\Theta_\mathbf{L}^*(\Omega)$  is a weak symplectic form on $\mathbb{P}$ whose associated Dirac structure  $D_{\Omega_{\mathbb{P}}}$ is also 
the pull-back $\Theta_\mathbf{L}^!(D_{\Delta_M})$ on $\mathbb{P}$.
%\item[3.]  Consider the Hamiiltonian $H$ on $E_0\times_M\times T^\prime M$ associted to $\mathbb{L}$  defined by
%$$H(x,u, \xi)=< \xi,\rho_x(u)>-\mathbb{L}(x,u)$$
 % Then the restriction $H_\mathbb{P}$  of $H$ to $\mathbb{P}$ is smooth and
\end{enumerate}
\end{proposition}

\begin{proof}
1. Consider a chart $(U,\phi)$ on $M$ such that $E_{U}$ is trivial and consider a trivialization $\Phi: E_{U}\to \phi(U)\times \mathbb{E}\subset \mathbb{M}\times\mathbb{E}$. In particular, $\Phi:{E_0}_{| U}\to\phi(U)\times \mathbb{E}_0$ is also a trivialization. Then $T^\prime 
M_{| U}$ is also trivializable  \textit{via} the adjoint map $T^*\phi$ of $T\phi$. In the same way, we get a trivialization $\widetilde{\Phi}: {TE^\prime}_{| U}\to  \phi(U)\times \mathbb{E}^\prime\times \mathbb{M}\times \mathbb{E}^\prime$. Now since the first part is a local property, we may assume that $M=\phi(U)$ and ${E}_0=
U\times {E}_0$ and so ${E}_0\times _M T^\prime M= U\times \mathbb{E}_0\times \mathbb{M}^\prime$. In this situation we have
\begin{equation}
\label{locK}
\mathbb{P}=\left\{(z,v,p) \; :\;\rho_z^*(p)-\displaystyle\frac{\partial  \mathbf{L}}{\partial u}(z,v)=0\right\}.
\end{equation}
From Lemma~\ref{L_ImFL}, around  $(x,u,p_0)\in \mathbb{P}$ (cf. (\ref{eq_JacoobFL})),  we can apply the implicit function Theorem in the Banach setting: 
 after shrinking $U$ if necessary, there exists an open neighbourhood   $U\times V$ around $(x,u)\in U\times \mathbb{E}_0$  and a neighbourhood $ V'$ of $\xi$ in $
 \mathbb{E}^\prime$ such that $\mathbb{F}\mathbf{L}$ is a diffeomorphism from $U\times V$ to $U\times V'$ if $\mathbb{F}\mathbf{L}(x,u)= (x,\xi)$. So  if 
 \begin{equation}\|abel{eq_W}
  W=\{(z,p )\;:\; \rho_z^*(p)\in  U\times V'\}
  \end{equation}
   then $W$ is an open set in   $U\times \mathbb{M}^\prime$ and so, locally $\mathbb{P}$ is defines as follows
\begin{equation}\label{eq_Local equationP}
\{(z,\mathbb{F}\mathsf{L}_{| U\times V}^{-1}(\rho_z^*( p)), p,\;(z,p)\in U\times V\}
\end{equation}
Therefore, $\mathbb{P}$ is a closed immersed submanifold in $E_0\times_M T^\prime M$ modelled on  $ \mathbb{M}\times\mathbb{M}^\prime$ and whose tangent bundle is transverse to $\ker T\widetilde{\pi}_0$ and so is a split  immersed submanifold of $E_0\times_M T^\prime M$. The last property is clear.\\

2. Consider the closed $2$-form $\widetilde{\pi}_0^*(\Omega)$ on $\mathbb{P}$. Taking into account  the local context of the first part, it follows that the the kernel of  $\widetilde{\pi}
_0^*(\Omega)$ is $\ker T\widetilde{\pi}_0$. This implies that the closed $2$-form $\widetilde{\Omega}$  is a strong pre-symplectic form.  On the other hand,  $\widetilde{\Omega}
^\flat(T(E_0\times_M T^\prime M)$ is the pullback of $T^\flat(T (T^\prime M))$ over  $T\widetilde{\pi}_0$. Thus  the graph of $\widetilde{\Omega}$ is a partial Dirac structure  
$D_{\widetilde{\Omega}}$ (cf. \cite{PeCa24}, Example~4.17).  
Moreover, from Assertion 1, the tangent bundle of $\mathbb{P}$ is transverse to the kernel of  $\widetilde{\pi}_0^*(\Omega)$. Therefore,  the induced form $\Omega_{\mathbb{P}}$ on 
$\mathbb{P}$ is a weak symplectic form and its graph is a  partial almost Dirac structure for the same argument as for $\widetilde{\Omega}$.  Now, from  
\cite{PeCa24}, Lemma-Definition~5.1,  the pull-back $\Theta_\mathbf{L}^!(D_{\Delta_M})$ on $\mathbb{P}$ is the graph of $\widetilde{\Omega}$.
\end{proof}

Taking into account the context of Proposition \ref{P_PDirac}, we set 
\[
\widetilde{M}=E_0\times_M T^\prime M \text{~~and~~} 
\widetilde{z}=(z,v,p)\in \widetilde{M}.
\] 
By analogy with $\S$~\ref{__ImplicitHamiltonianSystems}, we introduce:
\begin{definition}
\label{D_WidetildeMH} 
Let $H:\widetilde{M}\to \mathbb{R}$ a Hamiltonian function. 
\begin{enumerate} 
\item 
The \emph{implicit Hamiltonian system}  associated  to $(\widetilde{M}, D_{\widetilde{\Omega}}, H)$
 is the dynamical system characterized by the differential inclusion
 \begin{equation}
 \label{Eq_WidetildeMH}
\left( \dot{\widetilde{z}}(t), d_{\widetilde{z}(t)} H\in   D_{\widetilde{\Omega}}\right)
\end{equation}
\item 
A smooth curve $\widetilde{c}:]\varepsilon,\varepsilon[\to \widetilde{M}$ such that $ \left( \dot{\widetilde{c}}(t),d_{\widetilde{c}(t)}H \right) $ belongs to $(D_{\widetilde{\Omega}})_{\widetilde{c}(t)}$ is called an \emph{integral curve} of the implicit Hamiltonian system (\ref{Eq_WidetildeMH}) through $\widetilde{x}=\widetilde{c}(0)$.
\end{enumerate}
\end{definition}

\begin{remark}
\label{eq_IntegralCurveImplicit} ${}$
\begin{enumerate}
\item[1.] 
Using the local context of the proof of Proposition \ref{P_PDirac}, in local coordinates, we have
\begin{equation}
\label{Eq_Loc WidetildeOmega}
\widetilde{\Omega}_{\widetilde{z}}
\left( (\dot{x}, \dot{v}, \dot{p}),(\dot{z}', \dot{v}', \dot{p}')
\right)
\cong  \dot{\mathsf{p}}'(\dot{\mathsf{z}})- \dot{\mathsf{p}}(\dot{\mathsf{z}}')
\end{equation}
\item[2.] 
Let $\widetilde{c}:]-\varepsilon,\varepsilon[\to \widetilde{M}$ be an integral curve of the implicit Hamiltonian system (\ref{Eq_WidetildeMH}) through  $\widetilde{z}_0=\widetilde{c}(0)$.  
This means that we have 
\[
\widetilde{\Omega}( \dot{\widetilde{c}}(t),\;)=- d_{\dot{\widetilde{c}}(t)} H.
\]
If we write  $\widetilde{c}:=(z(t), v(t), p(t))$ this implies $d_{\dot{\widetilde{z}}(t)}((\dot{z}(t), \dot{v}(t), \dot{p}(t))=0$ for all $t\in ]-\varepsilon, \varepsilon[$ and so
this implies  $\displaystyle\frac{\partial H}{\partial v}(\widetilde{c}(t))=0 $, for all $t\in ]-\varepsilon, \varepsilon[$.
If we write  $\widetilde{c}:=(z(t), v(t), p(t))$,  its  projection $(z(t), p(t))$ on $T^\prime M$ 
%! of $c(t)$ 
is an integral curve of the singular partial almost Dirac $(D_{\Delta_M})_{(z(t),p(t))}$; in particular, if we have  $\dot{z}(t)=\rho_{z(t)}(v(t))$ on $ ]-\varepsilon, \varepsilon[$, then $z$ is an admissible curve.  
\item[3.] 
Conversely, consider a curve $t\mapsto (z(t), p(t))$ on $]-\varepsilon, \varepsilon[$ which is an integral curve of $(D_{\Delta_M})_{(z(t),p(t))}$. Then $\dot{z}(t)$ belongs to $(\Delta_M)_{z(t)}$  and assume  that $\dot{z}(t)=\rho(z(t), v(t))$. Then $t\mapsto (z(t), v(t), p(t))$ is an integral curve of  the implicit Hamiltonian system (\ref{Eq_WidetildeMH})  if and only if $\displaystyle\frac{\partial H}{\partial v}(\widetilde{c}(t))=0 $, for all $t\in ]-\varepsilon, \varepsilon[$.
\item[4.] 
More generally, assume that $H$ is a Hamiltonian; since $\widetilde{\Omega}$ is a strong pre-symplectic form on $\widetilde{M}$, then the equation
\begin{equation}
\label{eq_HamiltonianField}
\widetilde{\Omega}(X,\;)=-dH
\end{equation}
has a solution if and only if 
 $\displaystyle\frac{\partial H}{\partial v}(z,v,p)=0 $. Any such solution is defined modulo $\ker \widetilde{\Omega}$. For more details of this problem see \cite{PeCa24}, Proposition~4.6.
\item[5.] 
The implicit Hamiltonian system $(\widetilde{M}, D_{\widetilde{\Omega}}, H)$ can be written in the following way in local coordinates:
\begin{equation}
\label{eq_ImplicitTildeMH}
 \begin{cases}
 \dot{\mathsf{z}}(t) =\displaystyle\frac{\partial \mathsf{H}}{\partial  \mathsf{p}}(\mathsf{z}(t),\mathsf{v}(t),\mathsf{p}(t)))\\
 \dot{\mathsf{p}}(t)=-\displaystyle\frac{\partial \mathsf{H}}{\partial  \mathsf{x}} (\mathsf{z}(t),\mathsf{v}(t),\mathsf{p}(t)))\\
 \displaystyle\frac{\partial \mathsf{H}}{\partial  \mathsf{v}} (\mathsf{z}(t),\mathsf{v}(t),\mathsf{p}(t))=0
 \end{cases}\,.
 \end{equation}
  For simplicity we  set $\tilde{z}(t)=(z(t),v(t), p(t))$. Since   $\displaystyle\frac{\partial \mathsf{H}}{\partial  \mathsf{v}} (\tilde{\mathsf{z}}(t))=0$, this implies that 
  $$\begin{matrix}
 d_{\tilde{z}(t)}\mathsf{H}(\dot{z},\dot{v}(t),\dot{p}(t))&=\displaystyle\frac {\partial \mathsf{H}}{\partial \mathsf{x}}(\tilde{\mathsf{z}}(t))\dot{\mathsf{z}}(t)+ \frac{ \partial \mathsf{H}}{\partial \mathsf{p}}(\tilde{\mathsf{z}}(t))\dot{\mathsf{p}}(t)\hfill{}\\
 &=\displaystyle\frac {\partial \mathsf{H}}{\partial \mathsf{x}}(\tilde{\mathsf{z}}(t))\frac {\partial \mathsf{H}}{\partial \mathsf{p}}(\tilde{\mathsf{z}}(t))-\frac{ \partial \mathsf{H}}{\partial \mathsf{p}}(\tilde{\mathsf{z}}(t))\frac{ \partial \mathsf{H}}{\partial \mathsf{x}}(\tilde{\mathsf{z}}(t))\hfill{}\\
  &=0\hfill{}
 \end{matrix}$$
So $H$ is constant along any solution of (\ref{eq_ImplicitTildeMH}).

\item[6.] An integral curve   $\tilde{z}(t)$  of   $(\widetilde{M}, D_{\widetilde{\Omega}}, H)$ is  such that $t\mapsto(z(t),v(t))$ is  tangent to  $\Delta_M$ and satisfies  $\mathbb{F}\mathbf{L}(z(t),v(t))=(x(t),p(t))$ and so is contained in $\mathbb{P}$.
 In particular $\mathbb{P}$ is the integrality domain of $(\widetilde{M}, D_{\widetilde{\Omega}}, H)$. Since the restriction of $\widetilde{\Omega}$ to $\mathbb{P}$ is weak symplectic, for any Hamiltonian $H$ on $\widetilde{M}$, if  the differential of  the restriction $H_{\mathbb{P}}$ of $H$ to $\mathbb{P}$ satisfies the conditions of \cite{PeCa24}, Proposition 4.6, on some open set $U$ in $\mathbb{P}$, then there exists a unique Hamitonian vector field $X_H$ on $\mathbb{P}$. Moreover, $t\mapsto (z(t), v(t), p(t))$ is an  integral curve of the implicit Hamiltonian system (\ref{Eq_WidetildeMH}) on $U$  if and only if it is an integral curve of $X_H$. Thus according to Assertion (5), we have an open set  $W$ in $T'M$ on which 
 $h_W$ is defined \footnote{cf. relation (\ref{eq_hW})} is constant on each integral curve of $X_H$ contained in $U$.
\end{enumerate}
\end{remark}

To the Lagrangian $\mathbf{L}$, we associate the following Hamiltonian on $\widetilde{M}$:
\begin{equation}
\label{eq_HamiltonianL}
H(z,v,p)=<p,\rho_z(v)>-\mathbf{L}(z,v).
 \end{equation}

% Remark~\ref{eq_IntegralCurveImplicit}, justifies the following :

%\begin{definition}\label{D_IntegrabilityDomain}
%$\mathbb{P}$ is called \emph{the integrability domain} for  singular constraint  partial almost Dirac $D_{\Delta_M}$ associated to $\mathbf{L}$.
%\end{definition}

From Remark~\ref{eq_IntegralCurveImplicit}, we have

\begin{proposition}\label{P_HamiltonianVectorField} 
Let $H_\mathbf{L}^\mathbb{P}$ be the restriction of $H_\mathbf{L}$ to $\mathbb{P}$.  Then  the equation
$$i_X\Omega_P=-dH_{\mathbf{L}}^\mathbb{P}$$
has a unique solution $X_\mathbf{L}$ on $\mathbb{P}$. Moreover, any integral curve of the implicit Hamiltonian system $(\widetilde{M}, D_{\widetilde{\Omega}}, H_\mathbf{L})$  through any point of $\mathbb{P}$ is an integral curve of $X_\mathbf{L}$ and so is unique and tangent to $\mathbb{P}$.
\end{proposition}

This Proposition justifies the following Definition:

\begin{definition}\label{D_IntegrabilityDomain} 
The immersed submanifold $\mathbb{P}$ of $\widetilde{M}$ is called \emph{the integrability domain of the implicit Hamiltonian system  $(\widetilde{M}, D_{\widetilde{\Omega}}, H_\mathbf{L})$  associated to $\mathbf{L}$}.
\end{definition}
\begin{proof} 
Recall that $\Omega_\mathbb{P}$ is a weak symplectic form on $\mathbb{P}$.\\
 Now we have $\displaystyle\frac{\partial H_\mathbf{L}}{\partial v}(z,v,p)=\rho_z^*p-\mathbb{F}\mathbf{L}$, it follows that  $\displaystyle\frac{\partial H_\mathbf{L}^\mathbb{P}}{\partial u}\equiv 0$. Thus  $H_\mathbf{L}^\mathbb{P}$ satisfies trivially the assumptions of \cite{PeCa24}, Proposition~4.6 and so the Hamiltonian vector field $X_\mathbf{L}$ is well defined. The last part of the Proposition is then clear.
\end{proof}

\begin{remark}
\label{R_LocEdHamiltonianXLP} 
Taking into account Remark~\ref{eq_IntegralCurveImplicit}, 1.,  the Hamiltonian vector field $X_\mathbf{L}$  can be written
\begin{equation}
\label{eq_ComponentsXL}
X_{\mathbf{L}}(z,v,p)
= \left( \rho(z,v), 0,
- <p,\displaystyle\frac{\partial }{\partial x}(\rho(z,v))>
+\displaystyle\frac{\partial ^2}{\partial x\partial v}( \mathbf{L}(z,v))\right).\\
\end{equation}

Moreover,  according to (\ref{eq_Local equationP}),  since $\mathbb{F}\mathbf{L}$ is  invertible, on an open set $U\times V$, we have some open set $W$ in $T^\prime M$ 
%defined in (\ref{eq_W}) 
on which  the  smooth map  function 
\[
v_W(z,p):=(\mathbb{F}\mathbf{L})^{-1}(\rho_z^*( p)))
\]
is well defined and, for $ (z,p)\in W$, we have: 
\begin{equation}
\label{eq_DefHW}
H_\mathbf{L}^\mathbb{P} (z,v, p)
=<p,\rho(z,v_W(z,p))-\mathbf{L}(z, v_W(z,p))>
:= h_W(z, p).
\end{equation}
Therefore,  for $(z,p)\in W$, we have 
 \begin{equation}
 \label{eq_ComponentsXL}
 \begin{matrix}
X_{\mathbf{L}}(z,v,p)&=\left(\rho(z,v), 0,- <p,\displaystyle \frac{\partial (u_W)}{\partial x}(z,p))>+\displaystyle\frac{\partial ^2}{\partial x\partial u}( \mathbf{L}(z,u_W(z,v)))\right)\\
			&=\left(\displaystyle\frac{\partial h_W}{\partial p} (z, p), 0, -\frac{\partial h_W}{\partial x} (z, p)\right)\hfill{}\\
\end{matrix}	
\end{equation}

Note that we can consider $h_W$ as a Hamiltonian on $W\subset T^\prime M$ and so $X_{\mathbf{L}}$ as a Hamiltonian field relative to the canonical symplectic form $\Omega$ on $T^\prime M$.
  From Remark~\ref{eq_IntegralCurveImplicit}, (6), the value of $h_W$ is constant on any integral curve of $X_\mathbf{L}$.
\end{remark}
  
\subsection{Normal Extremals of a Lagrangian  on Singular Constraints}
\label{___NormalExtremalOfALagrangianOnSingularConstrains}

\emph{In this paragraph, we consider the context of normal extremals of the dynamical system associated to a regular Lagrangian on $E_0$ with singular constraints}.\\
 
Consider a  regular Lagrangian $\mathbf{L}: E_0\to \mathbb{R}$.   For $2\leq k\leq\infty$,  we denote by $C^k([t_1,t_2], E_0) $ the set of  curves $(z,v):[t_1,t_2]\to E_0$ of class $C^k$  and by $C^k_{E_0}([t_1,t_2], M)$  the set of  $\mathbb{E}_0$-admissible curves: $c: [t_1,t_2]\to M$,
that is 
\begin{equation}
\label{eq_admissiblel}
C^k_{E_0}([t_1,t_2], M)
=
\left\lbrace
(c, v) \in C^k
\left( [t_1,t_2], E_0
\right):\;
\dot{c}(t)=\rho_{c(t)}(v(t))
\right\rbrace .
\end{equation}
To $\mathbf{L}$ is associated the functional
\[
\mathfrak{F}_{\mathbf{{L}}}(c)
=
\displaystyle\int_{t_1}^{t_2}\mathbf{{L}}(c(t),v(t))dt
\]
defined on $C^k_{E_0} \left( [t_1,t_2], M \right) $.\\

% By analogy with $\S$~\ref{__NormalExtremals}, 
A \emph{horizontal variation} of $(c,v)\in C^k_{E_0} \left( [t_1,t_2], M \right) $ is a  $C^k$ map  $\hat{c}:]-\varepsilon,\varepsilon[\times [t_1,t_2]\to E_0$ such that
 $c_s:=\hat{c}(s,.)\in C^k_{E_0}\left( [t_1,t_2], M \right) $  and  $\hat{c}(0,t)=c(t)$.\\
We say that $\hat{c}$ is a \emph{horizontal variation with fixed origin (resp. ends} if  for all $s$, $c_s(t_1)=c(t_1)$ (resp.  $c_s(t_1)=c(t_1)$ and $c_s(t_2)=c(t_2)$ for all $s$).\\
 $\delta c(t):=\displaystyle\frac{\partial \hat{c}}{\partial{s}}(0,t)$ is called an \emph{infinitesimal horizontal variation}\index{infinitesimal horizontal variation} of  $c$.\\
Again, for any horizontal variation $\hat{c}$ of $c$,   the map $s\mapsto \mathfrak{F}_{\mathbf{{L}}}(c_s)$ is differentiable for $s=0$. 
  This differential is  again denoted  $d_c\mathfrak{F}_{\mathbf{{L}}}(\delta c)$. 
  
  More precisely, for $k\geq 2$, in this Banach context (cf. Warning~\ref{W_PMPBanachSetting}), we have (cf. \cite{GoPe21}) and \cite{GoPe25}):
  
\begin{theorem}
\label{T_AdmissibleCurves}${}$
\begin{enumerate}
\item 
 The set $C^k([t_1,t_2],M)$ of $C^k$ curves  and the set  $C^k_E( [t_1,t_2], M)$ of curves $\gamma:=(c,v)[t_1,t_2] \to E$ of class $C^k$ which are $E$-admissible 
 has a Banach manifold structure and the natural projection $C^k_E( [t_1,t_2], M) \to C^k([t_1,t_2],M)$ has a structure of Banach bundle.
\item 
For every fixed $x\in M$, the subset  
\[
C^k_E(x, [t_1,t_2], M)=\{((c,v)\in  C^k_E( [t_1,t_2], M)\;: c(t_1)=x\}
\]
is a split submanifold of 
 $C^2_E( [t_1,t_2], M)$  and the map $\operatorname{End}_x: C^k_{E}(x, [t_1,t_2], M)\to M$ is smooth.
\end{enumerate}
\end{theorem}

Note that since $E_0$ is an open subbundle of $E$, this Theorem is also true by replacing $E$ by $E_0$. 
%Therefore, as  in \cite{Arg20}, we have the same situation as in Description  about extremal at the begining of $\S$~\ref{___ ConvenientSetting}.\\

Now, we will show  a  slight global generalization of Proposition 2 in \cite{Arg20} whose proof follows the same argument as in \cite{Arg20}. However for the sake of completeness, we will give a complete proof.

\begin{proposition}
\label{P_tEpFL}
Fix some $c\in C^k_{E_0}(x, [t_1,t_2], M)$. There  exists $q\in T_{c(t_2)}^\prime M $ with $q\not=0$ such that 

 $\nu d_c\mathfrak{F}_\mathbf{L}- T_c^* \operatorname{End}_{c(t_1)}(q)=0$ 
\noindent if and only if there exists a  section  $p$ of class  $C^k$ of $c^!T^\prime M$ such that  $p(t_2)=q$ which satisfies
\begin{equation}
\label{eq_ConditionNormalExtremal}
 \begin{cases}
 \dot{\mathsf{c}}(t) =\displaystyle\frac{\partial \mathsf{H}^\nu}{\partial  \mathsf{p}}(\mathsf{c}(t),\mathsf{v}(t),\mathsf{p}(t)))\\
 \dot{\mathsf{p}}(t)=-\displaystyle\frac{\partial \mathsf{H}^\nu}{\partial  \mathsf{x}} (\mathsf{c}(t),\mathsf{v}(t),\mathsf{p}(t)))\\
 \displaystyle\frac{\partial \mathsf{H}^\nu}{\partial  \mathsf{v}} (\mathsf{c}(t),\mathsf{v}(t),\mathsf{p}(t))=0
 \end{cases}
 \end{equation}
 In this case, either $\gamma(t):=(c(t), v(t)))$ is a critical point  of $\mathfrak{F}_\mathbf{L}$ with fixed ends points and $\nu=1$ or else  a critical point of $\operatorname{End}_x$ with $\nu=0$.
 \end{proposition}
  
\begin{remark}
\label{R_CaseNu=1}
According to  Proposition~\ref{P_tEpFL},  Remark~\ref{eq_IntegralCurveImplicit}, 3. and  Proposition~\ref{P_HamiltonianVectorField}, we have:
 
 $\bullet\;\;$   if $(\gamma,p) $  is a solution of the differential system 
 (\ref{eq_ConditionNormalExtremal}), then  $p=\mathbb{F}\mathbf{L}(\gamma)$ with $p(t_2)\not=0$  for $\nu=1$  and so  $(\gamma, p)$  is an integral curve of  the Hamiltonian vector field $X_\mathbf{L}$ and so is a smooth curve. 
 
 $\bullet\;\;$ Conversely, if $\gamma:=(c,v)$ is an $E_0$-lift of $c$ such that $(\gamma, \mathbb{F}\mathbf{L}(\gamma))$ is an integral curve of $X_\mathbf{L}$ with $p(t_2)\not=0$, then $\gamma$ is a critical   point of $\mathfrak{F}_\mathbf{L}$ with fixed ends points  with $\nu=1$.
 \end{remark}
 
\begin{proof}[Proof of Proposition \ref{P_tEpFL}]  
Consider $c\in C^k_{E_0}(x, [t_1,t_2], M)$ and $\gamma(t):=(c(t),v(t))$  a $E_0$ lift of $c$.\\ 
Let $t\mapsto p(t)\in T^\prime_{c(t)}M$ be  the section of $c^!(T^\prime M)$\footnote{  $c^!(T^\prime M)$ is the  pull-back of $T^\prime M$ over $[t_1,t_2]$ \textit{via} $c$} which satisfies:
\begin{equation}
\label{eq_qH}
\dot{p}(t)
=-\displaystyle\frac{\partial {H}^\nu}{\partial x}(c(t),\dot{c}(t),p(t))
= -<p,\frac{\partial\rho}{\partial x} (\gamma)>-\nu\displaystyle\frac{\partial\mathbf{L}}{\partial{x}}(\gamma)
\end{equation}
with $p(t_2)= q$. Note that  $t\mapsto p(t)$ is at least of class $C^k$.\\
   
On the other hand,  by definition, we have
\begin{equation}
\label{eq_DotC}
   \dot{c}=\rho (\gamma) \textrm{  and so  }
   \delta\dot{c}(t)=\displaystyle\frac{\partial \rho}{\partial{x}}(\gamma)\delta c+\rho(c,\delta v)
\end{equation}
  Now we have 
\begin{equation}
\label{eq_NudF-TEnd1} 
\begin{matrix}
 \nu d_c\mathfrak{F}\mathbf{L}(\delta c)-T_c^* \operatorname{End}_{c(t_1)}(q)\hfill{}\\
  =
\nu\left\{\displaystyle\int_{t_1}^{t_2}\left(\frac{\partial \mathbf{L}}{\partial x}(\gamma)\right)(\delta c)
+\left(<p,\frac{\partial \rho}{\partial v}(\gamma)>\right)(\delta v) dt\right\}- p(t_2)(\delta c(t_2))\\
\end{matrix}
 \end{equation}
 But from (\ref{eq_qH}), we get:
 \begin{equation}
 \label{eq_pLpx}
\nu \displaystyle\frac{\partial \mathbf{L}}{\partial x}(\gamma)(\delta c)
=\left(\dot{p}+ 
<p,\frac{\partial \rho}{\partial x}(\gamma)>\right)(\delta c)
 \end{equation}
 By replacing in the second member of (\ref{eq_NudF-TEnd1}) the term  $\nu\displaystyle\frac{\partial \mathbf{L}}{\partial x}(\gamma)(\delta c)$ by the second member of 
 (\ref{eq_pLpx}),  %\nu\displaystyle\frac{\partial \mathbf{L}}{\partial x}(\gamma)(\delta c))$, 
  this gives rise to the term $\displaystyle\int_{t_1}^{t_2}<\dot{p}(t),\delta c(t)>dt$ in the second member of (\ref{eq_NudF-TEnd1}).\\
On the other hand,  by an integration by parts,  this last term,  taking in account that $\delta c(t_2)=0$ and the value of $\dot{c}$ in (\ref{eq_DotC}),  we obtain:
 \begin{equation}
 \label{eq_IntDotp}
 \begin{matrix}
 \displaystyle\int_{t_1}^{t_1}<\dot{p}(t),\delta c(t)>dt\hfill{}\\
 =p(t_2)(\delta c(t_2))-\displaystyle\int_{t_1}^{t_2}<p(t),\delta\dot{c}(t)> dt\hfill{}\hfill{}\\
 =p(t_2)(\delta c(t_2))-\displaystyle\int_{t_1}^{t_2}
 \left(<p(t),\frac{\partial \rho}{\partial x}(\gamma)>\right)(\delta c(t))+\left(<p(t),\rho \left( c(t),\delta v(t) \right)> \right) dt\hfill{}
 \end{matrix}
 \end{equation}
Finally we obtain:
 \begin{equation}\label{eq_NudF-TEnd2} 
\begin{matrix}
\nu d_c\mathfrak{F}\mathbf{L}(\delta c)-T_c^* \operatorname{End}_{c(t_1)}(q)\hfill{}\\
  =-\displaystyle\int_{t_1}^{t_2}\frac{ H^\nu}{\partial u}(\gamma(t), p(t))\delta v(t)dt\\
\end{matrix}
 \end{equation}
for all variation $\delta v$ with fixed fixed ends. From Lemma~\ref{L_VariationLemma}, it follows that, if $p(t_2)\not=0$,  
\[
\nu d_c\mathfrak{F}\mathbf{L}(\delta c)-T_c^* \operatorname{End}_{c(t_1)}(q)=0 \Longleftrightarrow  
\displaystyle\int_{t_1}^{t_2}\frac{ H\nu}{\partial v}(\gamma(t), p(t))=0
\]
which ends the proof .
\end{proof}
%By analogy with section~\ref{__NormalExtremals}, a \emph{horizontal variation} of $(c,v)\in C^\infty_{E_0} \left( [t_1,t_2], M \right) $ is a  $C^\infty$ map  $\hat{c}:]-\varepsilon,\varepsilon[\times [t_1,t_2]\to E_0$ such that
% $c_s:=\hat{c}(s,.)\in C^\infty_{E_0}\left( [t_1,t_2], M \right) $  and  $\hat{c}(0,t)=c(t)$.\\
%We say that $\hat{c}$ is a \emph{horizontal variation with fixed ends} if $c_s(t_1)=c(t_1)$ and $c_s(t_2)=c(t_2)$ for all $s$.\\
% $\delta c(t):=\displaystyle\frac{\partial \hat{c}}{\partial{s}}(0,t)$ is called an {\it infinitesimal horizontal  variation} of  $c$.\\
%Again, for any horizontal variation $\hat{c}$ of $c$,   the map $s\mapsto \mathfrak{F}_{\mathbf{{L}}}(c_s)$ is differentiable for $s=0$. 
%  This differential is  again denoted  $d_c\mathfrak{F}_{\mathbf{{L}}}(\delta c)$ and we have:
  
\begin{definition}
\label{D_ExtremalFLE0}
A critical point with fixed ends of $\mathfrak{F}_{{L}}$  on  $C^\infty_{E_0}([t_1,t_2], M)$  is then a curve $c\in C^\infty_{E_0}([t_1,t_2], M)$ such that 
 \begin{equation}
 \label{eq_Extremal FLDeltaM}
 d\mathfrak{F}_{\mathbf{{L}}}(\delta c)=0
 \end{equation}
 for all infinitesimal horizontal variations
  $\delta c$  associated to horizontal variations with fixed ends.
\end{definition}

Once more, 
%by analogy with $\S$~\ref{__NormalExtremals}, 
we can note that we have 
\begin{equation}
\label{eq_admissiblel}
C^\infty_{E_0}([t_1,t_2], M)
=
\left\lbrace
(c, v) \in C^\infty 
\left( [t_1,t_2], E_0
\right):\;
\dot{c}(t)-\rho_{c(t)}(v(t))=0
\right\rbrace 
\end{equation}
 and, on  $C^\infty([t_1,t_2]\widetilde{M}$,   we can define the functional:
 
\begin{equation}
\label{eq_FL+Constraint}
\widehat{\mathfrak{F}}_{\mathbf{\bar{L}}}=\displaystyle\int_{t_1}^{t_2}\left(\mathbf{\bar{L}}(x(t),v(t))+<p(t),(\dot{x}(t)-\rho_{x(t)}(v(t)))>\right)dt.
\end{equation}
  For any variation $(\hat{c},\hat{v},\hat{p}) $ defined on $]-\varepsilon,\varepsilon[\times [t_1,t_2]$ of $(c,v,p)$ in $C^\infty([t_1,t_2],\widetilde{M})$, the map $s\mapsto \widehat{\mathfrak{F}}_{\mathbf{\bar{L}}}(c_s,v_s,p_s)$ is differentiable at $s=0$ and we denote $d_c{\mathfrak{F}}_{\mathbf{\bar{L}}}(\delta c, \delta v,\delta p)$ its differential. 
 
\begin{definition}
\label{D_ExtremalFbarL} 
A curve $t \mapsto (c(t),v(t),p(t))$  in  $\widetilde{M}$  defined on $[t_1,t_2]$ is a critical point with fixed ends of  $\widehat{\mathfrak{F}}_{\mathbf{{L}}}$
if 
\begin{equation}
 \label{eq_dWidehatFL}
d \widehat{\mathfrak{F}}_{\mathbf{{L}}}(\delta c,\delta v,\delta p)=0
\end{equation}
for all   infinitesimal variations $(\delta c, \delta v,\delta p) $  associated to variations $(\hat{c},\hat{v},\hat{p})$ such that $\hat{c}(s,t_1)=c(t_1)$ and $\hat{c}(s,t_2)=c(t_2)$ for all $s$. 
\end{definition}
 
 By the way, with an obvious and simple  adaptation of the proof of Theorem~\ref{T_RelationExtremalFbarL}, we also obtain: 
 
\begin{theorem}
\label{T_RelationExtremalFbarLFixedEnds}
Fix some  $c\in C^\infty_{E_0}([t_1,t_2], M)$.\\
Under the previous considerations, we have the following equivalences:
\begin{enumerate}
\item[(1)] 
  There exists a section $p(t)$ of $c^!T^\prime M$ such that $(c(t),v(t),p(t))$ is a critical point with fixed ends of  $\widehat{\mathfrak{F}}_{\mathbf{{L}}}$.
\item[(2)] 
  $c$  is a critical point with fixed ends of $\mathfrak{F}_{\mathbf{{L}}}$.
 \item[(3)] 
 In a coordinates system associated to any chart domain which meets $c([t_1,t_2])$, we have the  Euler-Lagrange equation
\begin{equation} 
\label{eq_EulerLagrangeE0Singuier}
\left( \frac{\partial \mathsf{L}}{\partial \mathsf{x}}-\frac{d}{dt}\frac{\partial \mathsf{L}}{\partial \mathsf{v}} \right)
(\mathsf{c}(t),{\mathsf{v}}(t))=0.\\
\end{equation} 
\item[(4)] 
Let $H_\mathbf{L}$ be the Hamiltonian on $\mathbb{P}$ associated to $\mathbf{{L}}$. There exists a section $p(t)$ of $c^!T^\prime M$
such that    $(c(t),v(t), p(t))$ is a smooth  integral curve of the implicit Hamiltonian system  $(\widetilde{M}, D_{\widetilde{\Omega}}, 
H_\mathbf{L})$, that is, in local coordinates, satisfies the constraint system (\ref{eq_ImplicitTildeMH}).
 % \begin{equation}\label{eq_ImplicitHamiltonianinfty}
   %\begin{cases} 
% \dot{\mathsf{z}}=\displaystyle\frac{\partial {\mathsf{H}_\mathsf{L}}}{\partial \mathsf{p}}(\mathsf{z},\mathsf{v},\mathsf{p})\\
% \dot{\mathsf{p}}=-\displaystyle\frac{\partial {\mathsf{H}_\mathsf{L}}}{\partial \mathsf{x}}(\mathsf{x},\mathsf{v},\mathsf{p})\\
%\displaystyle\frac{\partial {\mathsf{H}}}{\partial \mathsf{v}}(\mathsf{z},\mathsf{v},\mathsf{p})=0
% \end{cases}\,
%\end{equation}
 %\item[(5)]
%The curve $(c(t),\dot{c}(t), p(t))=\mathbb{F}\mathbf{L}_{c(t)}(\dot{c}(t))$ is an integral curve of the implicit Lagrangian system $(M, \Delta_M,\mathfrak{D}\mathbf{L}))$.
\end{enumerate}
\end{theorem}
In reference  to Definition~\ref{D_CharaterizationNormalConvenientc}, it is natural to introduce:

\begin{definition}\label{D_CharaterizationNormalBanach} 
Under the previous assumptions, a curve  $c\in C^\infty_{E_0} ([t_1,t_2],M)$ is called a \emph{normal extremal} if, in each local coordinates on a chart domain which meets $c([t_1,t_2])$, there exists a section $p(t)$ of $c^!T^\prime M$  such that  the curve $t\mapsto (c(t),  p(t)):=\mathbb{F}\mathbf{L}(c(t),v(t))$ 
is a solution of the implicit differential system (\ref{eq_ImplicitTildeMH})
 with $p(t_2)\not=0$
\end{definition}

We then have the corollary:
\begin{corollary}
\label{C_NormalExtremal} 
Let $\mathbf{L}$ be a regular Lagrangian  on $E_0$  and consider a curve $c\in C^\infty_{E_0} ([t_1,t_2],M)$. Assume that $c$ satisfies any of the assumptions of Theorem~\ref{T_RelationExtremalFbarLFixedEnds}  and a section $p(\;)$ of $c^! T^\prime M$ such that $p(t_2)\not=0$. 
Then  $c$ is a normal extremal of $\mathfrak{F}_{\mathbf{\bar{L}}}$ and we have $\mathbb{F}\mathbf{l}(c(t),v(t))=(c(t), p(t))$. 
\end{corollary}
 
\begin{proof}  
We may assume the assumption (4) in Theorem \ref{T_RelationExtremalFbarLFixedEnds} with $p(t_2)\not=0$. This implies that  $(c(t),v(t), p(t))$  satisfies the differential system (\ref{eq_ConditionNormalExtremal}) of Proposition \ref{T_RelationExtremalFbarLFixedEnds} and so $c$ is a normal extremal since $p(t_2)\not=0$. 
\end{proof}

\begin{remark}
\label{R_OtherDefNormal}
According to Corollary~\ref{C_NormalExtremal}  and Theorem~\ref{T_RelationExtremalFbarLFixedEnds},  $c\in C^\infty_{E_0} ([t_1,t_2],M)$ is  a normal extremal if and only if $c$ is a critical point of  $\mathfrak{F}_\mathbf{L}$ and   $\mathbb{F}\mathbf{L}(c(t_2) v(t_2)) \not= (c(t_2),0)$.  
\end{remark}

\section{Normal Geodesics  of a  Conic sub Hilbert-Finsler Structure on a Banach manifold}\label{__NormalGeodesicsConicSubHilbertFinsler}
\textbf{In this section, we restrict ourself to the Banach setting.}\\
\emph{We give an application of the results of $\S$~\ref{__DiracStructuresOnSingularConstraintsInTheBanachSetting}.  and also characterize the normal geodesics of such a metric.}
 
\subsection{Conic Minkowski Norm on a Banach Space}
%\label{Minkowski}
\label{___ConicMinkowsiNorm}

Let $\mathbb{E}$ be Banach space. Recall that   a \emph{conic domain}\index{conic domain} $\mathbb{E}_0$ in $\mathbb{E}$ is an open subset of $\mathbb{E}$ such that if $u\in \mathbb{E}_0$ then $\lambda u$ belongs to $\mathbb{E}_0$ for any $\lambda>0$ and any $u\in \mathbb{E}_0$.\\
 A  \emph{weak conic Minkowski norm}\index{weak conic Minkowski norm} on $\mathbb{E}$  is a map $F$ from a conic domain $\mathbb{E}_0$ in $\mathbb{E}$ into $ [0,\infty[$ with the following properties:
\begin{enumerate}
\item[(i)] 
$ F(u)>0$ for all $u\in \mathbb{E}_0$ with $u\not=0$.
\item[(ii)] 
$F(\lambda u)=\lambda F(u)$ for all $\lambda>0$ and $u\in \mathbb{E}_0$.
\item[(iii)] 
$F$ is smooth on $\mathbb{E}_0\setminus\{0\}$, so that the Hessian of $\displaystyle \frac{1}{2}F^2$ is well defined.
\end{enumerate}

Note that if $0$ belongs to $\mathbb{E}_0$ then $\mathbb{E}_0=\mathbb{E}$.\\
Now if $\mathbb{S}$ is a hypersurface of $\mathbb{E}$ which does not contain $0$, then \[
 \mathbb{E}_0=\{\lambda  u,\; u\in \mathbb{S},\, \lambda>0\}
 \]
is a conic domain in $\mathbb{E}$.\\

The Hessian  $g$ of  $F^2$ is defined, for any $u\in \mathbb{E}_0$, any $v$ and $w$ in $\mathbb{E}$,  by
\begin{equation}
\label{eq_hessian}
g_u(v,w)
=\displaystyle\frac{\partial^2}{\partial s\partial t}
\left( \displaystyle\frac{1}{2}F^2(u+sv+tw) \right) _{| s=t=0}  
\end{equation}

Therefore $g_u$ is positive definite on $\mathbb{E}$. In particular, $g$ is a (weak) Riemannian metric on the Banach manifold $\mathbb{E}_0$ and each $g_u$ defines an inner product on $\mathbb{E}$.\\
From the Definition  of $F$, we have the following properties: (cf \cite{JaSa11} Proposition 2.2)
\begin{proposition}
\label{P_PropF}${}$
\begin{description}
\item[$\bullet$] 
$g_{\lambda  u}=g_u$ for any $\lambda>0$.
\item[$\bullet$] 
$g_u(u,u)=||u||^2:=\mathcal{F}^2(u)$.
\item[$\bullet$] $g_u(u,w)
=
\displaystyle\frac{\partial}{\partial s} \left( F^2(u+sw) \right) _{s=0}$.
\end{description}
\end{proposition}

%Consider  $v,w\in \mathbb{K}$ such that the set $\{tv+(1-t)w,\; \forall t\in ]0,1[\}$ is contained in $\mathbb{K}$. Then we have the strict triangular inequality $||v+w||\leq ||v||+||w ||$ with equality if and only if $v=\l w$. In particular, if $\mathbb{K}$ is convex then $F$ is a convex map.

\begin{definition} 
 A \emph{conic Minkowski norm on $\mathbb{E}$}\index{conic Minkowski norm} is a weak conic Minkowski norm 
 %such that its domain $\mathcal{K}$ is convex 
 if its Hessian $g$ is a strong Riemannian metric on the domain $\mathbb{E}_0$. 
%! \mathbb{K} transformed in \mathcal{K}
\end{definition}

When the domain $\mathbb{E}_0$ contains $0$, then a conic Minkowski norm will be simply called a  \emph{Minkowski norm}\index{Minkowski norm} even it is not a norm on $\mathbb{E}$ in the classical sense.

If we have a conic Minkowski norm on $\mathbb{E}$, this implies that each inner product $g_u$ on $\mathbb{E}$ provides $\mathbb{E}$ with a Hilbert structure. Conversely, if $\mathbb{E}$ is Hilbertizable, consider such an inner product $<\;,\;>$ on $\mathbb{E}$ and denote $F$ its associated norm. For any conic domain in $\mathbb{E}_0$, the restriction $F$ to $\mathbb{E}_0$ is a conic Minkowski metric. More generally, given a non-zero vector $\xi$ in $\mathbb{E}$ such that $F(\xi)<1$, then $\bar{F}(u)=F(u)+<\xi,u>$ in restriction to $\mathbb{E}_0$ is also a conic Minkowski.  We will give more original examples in section~\ref{__ConicSubHilbertFinslerStructure}

\begin{remark}
\label{R_NotNorm}
The term "norm" is not really appropriate in the sense that even if $E_0=\mathbb{E}$, the definition does not imply that a conic Minkowski norm $F$ is a norm (in the usual sense) on $\mathbb{E}$, since, in general, we will have $F(u)\not=F(-u)$  (cf. the previous  example).
\end{remark}

\begin{remark}
%\label{convex}
\label{R_TriangularInequalityConvex}
Consider  $v,w\in \mathbb{E}_0$ such that the set $\{tv+(1-t)w,\; \forall t\in ]0,1[\}$ is contained in $\mathbb{E}_0$. Then we have the strict triangular inequality $||v+w||\leq ||v||+||w ||$ with equality if and only if $v=\lambda w$. In particular, if $\mathbb{E}_0$ is convex, then $F$ is a convex map. This situation occurs in particular for any Minkowski norm.
\end{remark}

\subsection{Conic sub-Hilbert-Finsler Structure}
\label{__ConicSubHilbertFinslerStructure}

Consider a Banach anchored bundle $(E,M,\rho)$ and let $E_0$  be a conic (open)  subbundle of $E$   whose typical fibre $\mathbb{E}_0$ is a convex conic domain in $\mathbb{E}$. For example, consider a hypersurface $S$ in $E$ such that $S$ does not intersect the zero section $0_E$ and $\pi_{|S}$ is a surjective fibration  
with typical fibre a hypersurface  $\mathbb{S}$ of $\mathbb{E}$.
 Then
\[
 E_0=\{(x,\lambda u), (x,u)\in S,\; \lambda>0\}
 \]
 is a conic subbundle in $E$.
\begin{definition}
%\label{subHF}
\label{D_ConicHilbertFinslerMetric}
A \emph{conic Hilbert-Finsler metric}\index{conic Hilbert-Finsler metric} on a bundle $\pi:E\to M$ is a map $\mathcal{F}$ from a conic subbundle $E_0$ of $E$ in $[0,\infty[$ which is continuous and  smooth  on   $E_0\setminus \{0_E\}$ where the restriction $\mathcal{F}_x$ to $(E_0)_x$ is a conic Minkowski norm.  Such a structure is  also called a \emph{conic sub-Hilbert-Finsler  structure} on $M$ when we have an anchor $\rho:E\to TM$. When $E_0=E$ a conic Hilbert-Finsler metric will be simply called a \emph{ Hilbert-Finsler metric} on $E$ or \emph{sub-Hilbert-Finsler structure} on $M$.
\end{definition}

  %If $VE$ denote the vertical bundle of $E$, the hessian of $\displaystyleplaystyle\frac{1}{2}\mathcal{F}_x^2$ gives rise to smooth field  $g$ of symmetric billinear form   on $E_0\setminus \{0_E\}$. In fact $g$ is a  Riemannian metric on $VE$ above $E_0\setminus \{0_E\}$

\begin{example}
\label{Ex1}
Any strong Riemannian metric on $E$ is a conic Hilbert-Finsler metric on $E$ with $\mathcal{F}(x,u)=\sqrt{g_x(u,u)}$ and domain $E_0=E$. Moreover, assume that there exists a global section $X$ of $E$ such that $g(X,X)<1$. Then $\mathcal{F}_x(x,u)=\sqrt{g_x(u,u)}-g_x(X,u)$ is also a conic Hilbert-Finsler metric with domain $E$ which is not a norm in restriction to each fibre.
\end{example}

\begin{example}
\label{Ex2}
A Hilbert-Finsler metric on $E$ is a conic Hilbert-Finsler  on $E_0=E$ such that $\mathcal{F}_x$ is a Minkowski norm on $E_x$ for each $x\in M$. Thus it is a particular case of conic Hilbert-Finsler metric.
\end{example}

\begin{example}
\label{Ex3}
Consider a Hilbert space $\mathbb{H}$ and $<\;,\;>$ its inner product. On $E=M\times \mathbb{H}$ we get a strong Riemannian metric $g_x(u,v)=<u,v>$. Fix some  non zero $\xi\in \mathbb{H}$. For $0<\alpha<\beta<1$ we consider 
$$E_0=\left\{(x,u)\in M\times \mathbb{H}\setminus\{0\}\; \alpha<\displaystyle \frac{|<\xi,u>|}{<\xi,\xi>^{1/2}<u,u>^{1/2}}<\beta\right\}$$
Then $\mathcal{F}(x,u)=\sqrt{g_x(u,u)}$ in restriction to $\mathcal{K}$ is a conic Hilbert-Finsler metric. 
\end{example}

\begin{example}
\label{Ex4}
As in finite dimension, 
 (see \cite{JaSa11}) 
 let $\mathcal{F}_1,\dots,\mathcal{F}_k$ be a conic Hilbert-Finsler metrics on $E$ with domain $\mathcal{K}_1,\dots,\mathcal{K}_k$ respectively.\\
If $\mathbb{E}_0 =\displaystyle \bigcap_{i=1}^k\mathcal{K}_i$ is a conic domain, then $\mathcal{F}_1+\dots+\mathcal{F}_k$ is a conic Hilbert-Finsler metric on $E$ with domain $\mathbb{E}_0$. More generally, 
 any affine combination $\mathcal{F}=\displaystyle\sum_{i=1}^k \lambda_i\mathcal{F}_i$ with $0\leq \lambda_i\leq 1$ and $\displaystyle\sum_{i=1}^k \lambda_i=1$ is also a conic Hilbert Finsler metric with domain $\mathcal{K}$.
 
$ \mathcal{F}(x,u)
=\left[ \left( \displaystyle\sum_{i=1}^k \mathcal{F}_i^{1/q}(x,u) \right) ^q\right]^{1/2}$ for $q\geq 2$ is a conic Hilbert-Finsler metric with domain $E_0$.
\end{example}

\begin{definition} 
A \emph{conic sub-Hilbert-Finsler structure} on $M$ is the data of an anchored bundle $(E,M,\rho)$  and a conic Hilbert-Finsler metric $\mathcal{F}$ on a conic subbundle $E_0$ of $E$. When $E_0=E$, such a structure is called a \emph{sub-Hilbert-Finsler structure}.
\end{definition}

\begin{remark}
%\label{subriem} 
\label{R_SubRiemannianStructure}
(\cite{Arg20})  
A \emph{(singular) sub-Riemannian structure} on $M$ is the data of an anchored bundle $(E,M,\rho)$ and a Riemannian metric $g$ on $E$. Then $\mathcal{F}(x,u)=\sqrt{g_x(u,u)}$ is  a particular case of sub-Hilbert-Finsler structure on $M$. Conversely, a Hilbert-Finsler metric on  on $E$ such that $\displaystyle\frac{1}{2}\mathcal{F}^2$ is smooth and its Hessian only depends on $x$ and not on $u$ gives rise to a (singular) sub-Riemannian structure on $M$. As in the framework of Finsler structure on a finite dimensional manifold, to a Hilbert-Finsler metric on $E$, we can associate a \emph{Cartan tensor $C_\mathcal{F}$}\index{Cartan tensor} defined by:
\[
(C_\mathcal{F})_{(x,u)}(u_1,u_2,u_3)=\frac{\partial^3\mathcal{F}^2}{\partial s_1\partial s_2\partial s_3}\left(u+\sum_{i=1}^3 s_iu_i\right)_{| s_i=0}.
\]
When $\mathcal{F}^2$ is smooth on $E$, $\mathcal{F}^2$ gives rise to a sub-Riemannian structure on $M$ if and only if the Cartan tensor $C_\mathcal{F}$ is identically zero.
\end{remark}
  
Consider a conic Hilbert-Finsler  $E_0$.
 Denote by $\Delta_M=\rho(E)$ the associated distribution on $M$. Then, for any $ x\in M$,  the set $\Sigma_x:=\{\rho_x(u), u\in (E_0)_x\}$ is a conic domain in $(\Delta_M)_x$. On $\Sigma_x$,  we define $F_x(v):=\inf\{\mathcal{F}_x(u),\; \rho_x(u)=v\}$. Then, for 
 $v\in\Sigma_x$,  
 we have
 
 (i) $F_x(v)>0,\; \forall v\not=0$;
 
 (ii) $F_x(\lambda v)=\lambda F_x(v),\;\forall \lambda>0$.\\
Unfortunately, $F_x$ is not in general  a conic Minkowski norm on $(\Delta_M)_x$ with domain $\Sigma_x$ even when the typical fibre $\mathbb{E}_0$ is convex.

\subsection{Extremals and  Geodesics for Admissible Curves in a sub-Hilbert-Finsler Structure}

\subsubsection{Energy Length and Geodesics}

We consider a conic sub-Hilbert-Finsler metric $\mathcal{F}$ on an anchored bundle $(E,M,\rho)$ defined on a conic domain $\mathbb{E}_0$. 

\begin{definition} ${}$
\begin{enumerate} 
\item
Consider a  smooth  curve   
$\gamma:=(c, v): [t_1,t_2]\to E_0$. 
\begin{enumerate}
\item[(i)] 
The \emph{energy}\index{energy} of $\gamma$ is 
 $ \mathcal{E}(\gamma)
 =\displaystyle\int_{t_1}^{t_2}\frac{1}{2} \mathcal{F}^2(c(t), v(t))dt$.
\item[(ii)] 
The \emph{length}\index{length} of $\gamma$ is 
$L(\gamma)=\displaystyle\int_{t_1}^{t_2} \mathcal{F}(c(t),v(t))dt$
\end{enumerate}
\item
 Consider a curve $c\in C^\infty_{E_0} \left( [t_1,t_2], M \right) $.\\
$\gamma:= (c,v):[t_1,t_2]\to E_0 $ is called a \emph{lift}\index{lift} of $c$ if  $c=\pi\circ \gamma$ and $\dot{c}=\rho(\gamma)$.
\begin{enumerate}
\item[(i)]
The \emph{energy}\index{energy} of $c$ is $ \mathcal{E}(c)
=\inf\{ \mathcal{E}(\gamma), \gamma \textrm{ lift of }c  \}$.
\item[(ii)] 
The \emph{length} of $c$ is $L(c)=\inf\{ L(\gamma), \gamma \textrm{  lift of } c\}$.
\end{enumerate}
\end{enumerate}
\end{definition}

%Again a horizontal variation of $c \in C^\infty_{E_0} \left( [t_1,t_2], M \right) $ is a  $C^\infty$ map  $\hat{c}:]-\varepsilon,\varepsilon[\times [t_1,t_2]\to E_0$ such that
% $c_s:=\hat{c}(s,.)\in C^\infty_{E_0}\left( [t_1,t_2], M \right) $  and  $\hat{c}(0,t)=c(t)$.\\
%We say that $\hat{c}$ is a \emph{horizontal variation with fixed ends} if $c_s(t_1)=c(t_1)$ and $c_s(t_2)=c(t_2)$ for all $s$.\\
% $\delta c(t):=\displaystyle\frac{\partial \hat{c}}{\partial{s}}(0,t)$ is called an {\it infinitesimal horizontal  variation} of  $c$.\\
%Again, for any horizontal variation $\hat{c}$ of $c$,   the map $s\mapsto \mathfrak{F}_{\mathbf{{L}}}(c_s)$ is differentiable for $s=0$. 
%  This differential is  again denoted  $d_c\mathfrak{F}_{\mathbf{{L}}}(\delta c)$ 

%\noindent Note  that, if $\gamma:[a,b]\to M$ is a $\mathbb{E}_0$-admissible curve, the quantity $F_{\gamma(t)}(\dot{\gamma}(t))$ is well defined on $[a,b]$. Given any $\mathbb{E}_0$-lift $c:[a,b]\to \mathbb{E}_0$ since we have $F_{\gamma(t)}(\dot{\gamma}(t))\leq \mathcal{F}_{c(t)}(c(t))$ outside a countable of subset of of $[a,b]$.  The integrals $\displaystyle\int_a^b F_{\gamma(t)}(\dot{\gamma}(t))dt$ and   $\displaystyle\frac{1}{2}\int_a^b F^2_{\gamma(t)}(\dot{\gamma}(t))dt$ are well defined.
%However  there is no reason that this value  is exactly $L(c)$ and $\mathcal{E}(c)$ respectively for some $\mathbb{E}_0$-lift $c:[a,b]\to \mathbb{E}_0$.

\begin{definition} 
\label{D_Geodesic}${}$
\begin{enumerate}
\item 
 An $\mathbb{E}_0$-admissible curve $c:[a,b]\to E$ is called \emph{locally trivial} if there exist $t_0\in ]a,b[$ and a sub-interval $[t_1,t_2]$ of $[a,b]$ which contains $t_0$ and such that $c_{|[t_1,t_2]}=0$.
\item 
An $\mathbb{E}_0$-admissible curve $c:[a,b]\to E$ which is not locally trivial is called a \emph{length  locally minimizing} if,  for any $t_0\in [a,b]$, there exists a sub-interval  $[t_1,t_2]$ in $[a,b]$ containing $t_0$ and a neighbourhood $U$ of $c([t_1,t_2])$ such, for any 
$\bar{c}\in  C^\infty_{E_0} \left( [t_1,t_2], U \right) $  with 
$\bar{c}(t_i)=  c(t_i)$, $i \in \{0,1\}$,  we have 
$L \left( c_{| [t_1,t_2]} \right) \leq L(\bar{c})$. 
Such a curve is called a \emph{geodesic}\index{geodesic}. 
%\item a curve $\gamma\in C^\infty([a,b], E_0$ is called locally energy minimizing energy, if  for any $t_0\in [a,b]$, there exists an sub-interval  $[t_1,t_2]$ in $[a,b]$ containing $t_0$ and a neighbourhood $U$ of $c([t_1,t_2])$ such any $c'\in  C^\infty_{E_0} \left( [t_1,t_2], U \right) $  with $c'(t_i)=  c'(t_i)$, $i=0,1$   then $\mathcal{E}(c_{| [t_0,t_1]})\leq \mathcal{E}(c')$.\\
\end{enumerate}
\end{definition}
As classically in Finsler geometry, we have the following link between energy and length for locally minimizing curves:

\begin{proposition}\label{P_CompareLegthEnergy} 
Let  $c:[a,b]\to M$ be an $\mathbb{E}_0$-admissible curve which has a lift  $\gamma:[a,b]\to E\setminus \{0_E\}$. We have the following properties:
\begin{enumerate}

\item[1.] 
For any lift $\gamma=(c,v):[a,b]\to E_0$ of $c$, the length $L(\gamma)$ does not depend on the parametrization of $c$.
\item[2.]  
For  any  lift  $\gamma=(c,v):[a,b]\to E_0 \setminus \{0_E\}$,  the map $\tau$ defined by 
\[
\tau(s)
=a+\displaystyle\int_a^t\mathcal{F}
\left( c(s), \frac{1}{\mathcal{F}(\gamma(s))} v(s) \right) ds
\]
is a well defined  strict increasing smooth diffeomorphism of $[a,b]$ such that 
$\displaystyle\frac{d}{d s}(\tau)=1$.
% It will be called a arc-length parametrization of $c$.
\item[3.] 
The curve $c$ is locally minimizing for the length  if and only if there  exists an $E_0$-lift of $c$ which is energy locally minimizing. 
\end{enumerate}
\end{proposition}

\begin{proof} 
Assertions 1 and 2 follow from the change of variables of integration.

Assertion 3 is a direct consequence of the Holder-Schwarz inequality.\\
Since we use this  argument later,  we give a proof.   We have:  
 % we have $\displaystyle\frac{d}{ds}(\tau(s))=\displaystyle\frac{d}{dt}(c(t))\dot{\tau}(s)$ Therefore  $\gamma'(s)=(c(s),  \dot{\tau}(s) v(s)$ is also a lift of of $c$
$$\left(\int_{t_1}^{t_2}\mathcal{F}(\gamma(t))dt\right)^2\leq (t_2-t_1)\displaystyle\int_{t_0}^{t_1}\mathcal{F}^2(\gamma(t))dt$$ 
with equality if and only if  $\mathcal{F}(c(t))$ is constant. It follows that if $\gamma_{|[t_0,t_1]}$ is a minimum for  $\mathcal{E}$, then it  is also a minimum for the length. In particular, for such $\gamma$, we have 
$L(c_{|[t_0,t_1]})=L(\gamma_{|[t_0,t_1]} )$. Conversely, according to  the Holder-Schwarz inequality, if 
$\gamma_{|[t_0,t_1]}$ is a minimum ,   for the parametrization given in Assertion 1,  there exists a lift $\gamma$ of $c$ for which  this minimum  is 
\[
 L \left( \gamma_{[t_0,t_1]} \right) 
=\displaystyle\frac{2\mathcal{E}(\gamma_{|[t_0,t_1]})}{t_2-t_1}.
\] 
\end{proof}

\begin{remark}
\label{R_MinimalLength} 
Under the proof  of  Proposition \ref{P_CompareLegthEnergy}, for any $t_0\in [a,b]$, there exists a sub-interval $[t_1,t_2]$ in $[a,b]$, which contais $t_0$, 
 a $E_0$-lift $\gamma$ of $c$  and a reparametrization of $c_{| [t_1,t_2]} $ such that $L(c_{| [t_1,t_2]})=L(\gamma_{|[t_0,t_1]})=\displaystyle\frac{2\mathcal{E}(\gamma_{|[t_0,t_1]})}{t_2-t_1}$. In particular, $L(c)$ does not depend on the parametrization of $c$.
\end{remark}
On the other hand, 
%as in $\S$~\ref{NormalExtremalL}, 
we consider:\\ 

$\bullet\;\;$ on the space $C^\infty \left( [t_1,t_2], E_0 \right) $  the functional 
\begin{equation}
\label{eq_FL+ConstraintConic1}
\mathcal{E}(\gamma)
=\displaystyle\int_{t_1}^{t_2}\left( \frac{1}{2} \mathcal{F}^2(\gamma(t))\right)dt.
\end{equation}

$\bullet\;\;$ on the  the set  $C^\infty \left( [t_1,t_2], \widetilde{M} \right) $, the functional 
\begin{equation}
\label{eq_FL+ConstraintConic2}
\widehat{\mathcal{E}}(\gamma, p) 
= \displaystyle\int_{t_1}^{t_2}\left(\frac{1}{2} \mathcal{F}^2(\gamma(t))+<p(t),\dot{x}(t)-\rho_{x(t)}(v(t)) > \right) dt.
\end{equation}

if $\gamma(t)=(c(t), v(t))$  and $(\gamma(.), ,p(.))\in  C^\infty([t_1,t_2], \widetilde{M}) $.
\medskip

Then we have:

\begin{proposition}
\label{P_locMin}
Let $c$ be  locally minimizing such that $c$ has a lift $\gamma:[a,b]\to E_0$.  For any $t_0\in [a, b]$, let $[t_1,t_2]$ be a sub-interval of $[a,b]$ and an open neighbourhood $U$ of $c \left( [t_1,t_2] \right) $ which satisfy the assumptions of Definition~\ref{D_Geodesic}. There exists a  lift $\gamma: [t_1, t_2],\to E_0\setminus \{0_E\}$  of  $c \left( [t_1,t_2] \right) $ such that  $\gamma_{|[t_1,t_2]}$ is a critical point  of  $\mathcal{E}$ on $C^\infty_{E_0}([t_1,t_2], U)$.
%(resp. $\widetilde{\mathcal{E}}$ on $C^\infty([t_1,t_2],\widetilde{M})$). 
\end{proposition} 
 
\begin{proof} (classical)
From Remark \ref{R_MinimalLength}, if $c$ is locally minimizing, there exists  a lift $\gamma$ of $c$   such that   $\gamma$ is locally minimizing for the energy.  Now  since $c$ is not locally trivial,  if $\gamma_{|[t_1,t_2]}$ were not a critical point  $ \mathcal{E}$, for some $t_0\in [a,b[]$   and some sub-interval $[t_1,t_2]$ of $[a,b]$ such that
 $d_{ \gamma_{| [t_1,t_2]}}\mathcal{E}\not=0$ there would exist a variation  $\hat{\gamma}=(\hat{c}, \hat{v})$ on $[t_1,t_2]$ with fixed ends such that $\mathcal{E}(\gamma_s)<\mathcal{E}(\gamma)$ for $s$ small enough and $s>0$ or $s<0$,  which gives rise to a contradiction.\\
Now, from Theorem~\ref{T_RelationExtremalFbarLFixedEnds},  it is also a critical point  of $\widehat{\mathcal{E}}$ on $C^\infty([t_1,t_2],\widetilde{M})$.
\end{proof}

According to Remark~\ref{R_OtherDefNormal}, this Proposition justifies the following Definition:

\begin{definition}
\label{D_NormalGeodesic} 
 A locally minimizing curve $c\in C^\infty_{E_0} ([t_1,t_2],M)$ is called a \emph{normal geodesic}\index{normal geodesic} if, for any $t_0\in [a, b]$, there exists  a sub-interval  $[t_1,t_0]$ of $[a,b]$ and an open neighbourhood $U$  of $c \left( [t_1,t_2] \right) $ which satisfy the assumptions of Definition~\ref{D_Geodesic}, and  a section $p(.)$ of $c^!T^\prime M$ over $[t_1,t_2]$ such that 
$t \mapsto (c(t),v(t),p(t))$ 
is a critical point  of $\widehat{\mathcal{E}}$ on $C^\infty([t_1,t_2], U)$ with $p(t_2)\not=0$.
\end{definition}

\subsubsection{Characterization of Normal Geodesics}
\label{____CharacterizationOfNormalGeodesics}

On $\widetilde{M}$, to the problem of critical point of the energy $\widehat{\mathcal{E}}$ is associated the  Hamiltonian
%!o $z$ <- $x$
\begin{equation}\label{eq_ConicFinslerHamiltonian}
H(x,v,p)
=<p,\rho(x,v))>-\displaystyle\frac{1}{2}\mathcal{F}^2(x,v)
=<p,\rho(x,v)>
-\displaystyle\frac{1}{2}g_{(x,v)}(v,v)
\end{equation}

If $g_{(x,v)}$ denote the Hessian of $\mathcal{F}^2$ at $(x,v)\in E_0$,  according to the elementary properties of $g_{(x,v)}$  in each fibre $(E_0)_x$ (cf. relation~(\ref{eq_hessian})), we have 
$g_{(x,v)}(v,\;)=\displaystyle\frac{\partial \mathcal{F}^2_x}{\partial v}(v,\;)$ which is the Legendre transformation associated 
 to the Lagrangian $\mathcal{F}^2$. Moreover, we also have $\mathcal{F}^2(x,v)=g_{(x,v)}(v,v)$ (cf. Proposition~\ref{P_PropF}).  
This implies that we have 
\begin{description}
\item 
From (\ref{eq_ConicFinslerHamiltonian}), we have:
\begin{equation}
\label{eq_HamiltonianF}
H(x,v,p)=<p,\rho(x,v))>-\displaystyle\frac{1}{2}g_{(x,v)}(v,v)
\end{equation}
\item  
\begin{eqnarray}
\label{dHdu}
\displaystyle\frac{\partial H}{\partial v}(x,v,p)(\;)=p- \displaystyle\frac{1}{2}g_{(x,v)}(v,\;).
\end{eqnarray}
\item  
The manifold $\mathbb{P}$ in $\widetilde{M}$  is 
\begin{equation}
\label{eq_PConic}
\mathbb{P}=\{(x,v,p)\in \widetilde{M}\;:\;\rho_x^*p= g_{(x,v)}(v,\;)\}
\end{equation}
and so $\displaystyle\frac{\partial H}{\partial v}(x,v,p)\equiv 0$ on $\mathbb{P}$.
\end{description}
 
%Now, from Proposition \ref{P_HamiltonianVectorField},  to  $H_\mathbb{P}$ is associated a  Hamiltonian vector field on $\mathbb{P}$ which will be denoted $X_\mathbb{P}$. According to Remark \ref{R_LocEdHamiltonianXL}, we have:
%\begin{equation}\label{eq_ComponentXF}
%X_{\mathbb{P}}(z,v,p)=\left(\rho(z,v), 0,\displaystyle\frac{\partial }{\partial x}(g_{(z,v,)}(v,v))\right).\\
%\end{equation}

%Therefore, according to Proposition \ref}\label{P_locMin} and Definition \refl{D_NormalGeodesic},  if $c$ is a normal geodesic, if  $p(t)=g_{(c(t),v(t))}(v(t),\;)$, then  $(c(t), v(t), p(t))$ is an integral curves of $X_\mathcal{F}$ for any $t_0\in [a,b]$ there exists a sub-nterval $[t_1,t_2]$ of $[a,b]$ around $t_0$ such that $p(t_2)-p(t_&)\not=0$

Now, as in \cite{Arg20}, we have the following characterization of normal  geodesics:

\begin{theorem}\label{T_CharacterizationGeodesic}
Consider an $\mathbb{E}_0$-admissible curve $c:[a,b]\to M$. %if $\gamma(t)=(c(t), v(t)$ is a $E_0$-lift of $c$ and we set $p(t)=g_{\gamma(t))}(v(t),\;)$. 
 Then $c$ is a normal geodesic if and only there exists an $E_0\setminus \{0_E\}$-lift  such that if $p(t)=g_{\gamma(t))}(v(t),\;)$, then  the curve $t\mapsto (c(t), p(t))$ is an integral curve of the Hamiltonian vector field $X_{\dfrac{1}{2}\mathcal{F}^2}$\footnote{cf. Notations of Proposition~\ref{P_HamiltonianVectorField}.} .
\end{theorem}
The proof of Theorem~\ref{T_CharacterizationGeodesic} will proceed in two steps:
\begin{enumerate}
\item[$\bullet$]
The first one for the necessary condition, which is  very easy. 
\item[$\bullet$]
The second one  for the sufficient condition is very much more technical and which  follows, step by step, the proof of property "local minimizing geodesic" in the proof of Theorem 7 of \cite{Arg20}.
\end{enumerate}

\begin{proof}${}$\\
1. \emph{Necessary condition}.\\
Fix a  geodesic $c:[a,b]\to M$. The necessary condition is a local problem; so we can consider  the local upper context. From (\ref{eq_HamiltonianF}) and (\ref{eq_PConic}),  the restriction of  $H_\mathbf{L}$ to $\mathbb{P}$ is 
\begin{equation}
\label{eq_HrestrictP} 
H_\mathbf{L}(x,v,p)=\displaystyle\frac{1}{2}g_{(x,v)}(v,v)
\end{equation}
But, since $c$ is a  geodesic, from Proposition~\ref{P_locMin}, there exists an $E_0$-lift \\
$\gamma:= \left( c(\;), v(\;) \right) $ of $c$ which is a critical point of $\mathcal{E}$.\\
On the other hand, from Theorem~\ref{T_RelationExtremalFbarLFixedEnds},  (4) and Remark~\ref{eq_IntegralCurveImplicit}, (6),  the relation (\ref{eq_HrestrictP}) implies that  
\begin{equation}
\label{eq_gvvconstant}
g_{\gamma(t)}(v(t),v(t))=\text{Cte} 
\end{equation}
for all $t\in [t_1,t_2]$.\\
Since $c$ is not locally trivial, this implies that $g_{\gamma(t)}(v(t),v(t))\not=0$ for all $t\in[t_1,t_2]$  and then  $p(t)=g_{\gamma(t)}(v(t),\;)\not=0$ for all $t\in [t_1,t_2]$, in particular, for $t_2$.  As,  for every  $t\in [a,b]$, $g_{\gamma(t)}(\;,\;)$ is a non degenerate bilinear form on $E_{c(t)}$, this implies $v(t)\not=0$, for all $t\in [a,b]$. Thus, From Corollary~\ref{C_NormalExtremal}, $c$ is a normal geodesic whose $E_0$-lift is contained in $E_0\setminus \{0_E\}$.\\

2. \emph{Sufficient condition}.\\
According to the proof of the corresponding result of Theorem~7 in \cite{Arg20},  we only have to build a closed  $1$-form $\theta$ on a neighbourhood  in restriction to some small interval  $[t_0, t_0+\varepsilon]$\footnote{If $t_0=b$, the same type of  arguments will be true for the interval $[b-\varepsilon, b]$.} (called a \emph{calibration} in \cite{Arg20}) with the following properties:

there exists $\delta>0$  such that we have
  
 \noindent  $\bullet\;$ for $t$ small enough:
$\;\;\;\theta_{\gamma(t)}(\dot{\gamma}(t))=\delta \sqrt{{H}(\gamma(t),p(t))}\leq \delta \sqrt{g_{\gamma(t)}(v(t),v(t))}$.\\
$\bullet\;$ 
for any $(x,v)\in E$ close to $\gamma(t_0)$:
$\;\;\;|\theta_x(\rho_x(v))|\leq \delta\sqrt{ g_{\gamma(t_0)}(v,v)}$. \\

\emph{Then the result will follow by using the same (classical) arguments as the ones on calibration used in \cite{Arg20}. Thus we only have to build $\theta$ in our context and verify at each step that all the arguments used in such a proof in \cite{Arg20} are also valid in this context}. \\
For such a proof we need to precise this local context.

\begin{notations}  
We will apply the context of Remark \ref{R_LocEdHamiltonianXLP} to the context of  "sub-Hilbert-Finsler metric".\\

Consider a  fixed $\mathbb{E}_0$-admissible curve $c:[a,b]\to M$ which has an $E_0\setminus \{0_E\}$ lift  $\gamma(t)=(c(t), v(t))$  and set $p(t)=g_{\gamma(t)}(v(t),\;)$. 
We also fix some $t_0\in [a,b]$,  a sub-interval $[t_0,t_1]$ with $\varepsilon=t_1-t_0$ in $[a,b]$  and some open set $U$ in $M$ of $c \left( [t_0,t_1] \right) $. Finally, let $W$ be an  open neighbourhood  of 
$ \left( c \left( [t_1, t_2] \right) , p \left( [t_1,t_2] \right) \right) $ such that the projection of $W$ on $M$ contains $U$.

As $g_{(x,v)}(v,\;)$ is the Legendre transformation $\mathbb{L}\mathcal{F}^2$ of $\mathcal{F}^2$  and $g_{(x,v)}(\;,\;)$ is a strong Riemannian metric, by assumption,  on the manifold $E_0$,  the Lagrangian $\mathcal{F}^2$ is regular.
 
We denote by $g_{(x,v)}^\flat$ the associated canonical isomorphism from $E_{(z,v)}$ to  its dual $E_{(x,v)}^\prime$\footnote{Recall that $g_{(x,v)}$ is a strong non-degenerate bilinear form on each fibre $E_{(x,v)}$.}. 
 By the way, from Remark~\ref{R_LocEdHamiltonianXLP}, we can choose $[t_1,t_2]$, $U$ and $W$, small enough, such that
\begin{equation}
\label{eq_VWFinser}
v_W(x,p)
:=(g_{(x,v)}^\flat)^{-1} \left( \rho_x^*(p) \right) 
\end{equation}
On  $W$, we have  
$\rho_x^*(p)=g_{(x, v_W(x,p))}(v_W(x,p),\;)$  and from (\ref{eq_HrestrictP}), we obtain 
(cf. Example~\ref{Ex_SubRiemaniannToBeContinued},  relation (\ref{eq_HamitonianSubRiemannian}))
\begin{equation}
\label{eq_hW}
h_W (x,p)
=\displaystyle\frac{1}{2}g_{(x, v_W(x,p))} \left( v_W(x,p),v_W(x,p) \right)
=
\displaystyle\frac{1}{2}\mathcal{F}^2
 \left( x, v_W(x,p) \right) \displaystyle\frac{1}{2}
<p, \rho_x \left( v_W(x,p) \right) >\\
\end{equation}
In fact, $h_W$ is a  Hamiltonian on $W\subset T^\prime M$  and $X_W$ is a vector field  on $W$. Moreover, from relations (\ref{eq_VWFinser}) and (\ref{eq_hW}), $h_W$ is homogeneous of degree 2 in the second variable $p$.\\

Now assume that   $ \left( c(t),p(t) \right) $ is an integral curve of  the Hamiltonian vector field $X_{\frac{1}{2}\mathcal{F}^2}$ where  $p(t)=g_{\gamma(t))}(v(t),\;)$.  Since $v(t)\not=0$ on $[a,b]$, according to Proposition~\ref{P_CompareLegthEnergy}, we have $\mathcal{F}(c(t),v(t))=1$ and so, without loss of generality, we may assume that $[a,b]=[0,1]$.\\

For simplicity, the associated  smooth map $v_W(\;)$  is  denoted $v(\;)$,  the associated  hamiltonian $h_W$ is denoted $h$ and its Hamiltonian vector field is denoted $X_h$.\\
Moreover, without loss of generality, we can identify $U$ with an open set of $\mathbb{M}$ and $W$ with an open set of $T^\prime U\equiv U\times \mathbb{M}^\prime$ which can be chosen of  type $U\times U'$ where $U'$ is an open set of $\mathbb{M}
^\prime $.  Thus, the  Hamiltonian vector field $X_h$ is a vector field on $U\times U'$. \\
 Moreover, we can also assume that $x_0=c(t_0)=0\in U$ and we set $v \left( \gamma(t_0) \right) =(0, p_0) \in U\times U^\prime$.
\end{notations}

{\bf Construction of $\Theta$}${}$\\
According to the previous notations, after restricting $U$ if necessary, we can assume that $X_h$ has a  flow defined on $U\times U'$ which takes the form 
 
 $(t,{x}, p)\mapsto \left( \mathsf{Fl}_1(t,{x} ,p),\mathsf{Fl}_2(t,{x} ,p) \right).$
 
\noindent Note that since $\gamma([0,1]) \subset E_0\setminus \{0_E\}$,  then $v(t)\not=0$, for all $t\in [0,1]$ and so $p_0=g_{\gamma(t_0)} \left( v(t_0), v(t_0) \right) >0$, which implies that $p_0\not=0$. Since  the curve $(c,p)$ is an integral curve of $X_h$ through  $(0,p_0)$, we have $\mathsf{Fl}_1(t,0,p_0)= (c(t), p(t))$ for $t$ small enough and $p(t_0)=p_0$. In this context, let $\mathbb{K}$ be the kernel of $p_0$ in an hyperplane in $\mathbb{M}$ and set $U_0=\mathbb{K}\cap U$.\\ 

As in  the proof of Theorem 7 in \cite {Arg20}, for $\varepsilon >0$ small enough we can define a map $\mathfrak{f}: ]t_0-\varepsilon, t_0+\varepsilon[\times U_0\to U$ after shrinking $U_0$ if necessary, by
 \begin{eqnarray}
 \label{phi}
\mathfrak{f}(t,x)=\mathsf{Fl}_1(t,x,\nu(x)p_0)
\end{eqnarray}
where $\nu(x)= \sqrt{\dfrac{h(0,p_0)}{h(x,p_0)}}$.  But we have seen that $h$ is $2$-homogeneous in $p$, so we have
\begin{eqnarray}
\label{2H}
h(x,\nu(x)p_0)=h(0,p_0),\;\; \forall x\in U_0
\end{eqnarray}

By same arguments as in the proof of Lemma~5 in \cite{Arg20}, after reducing $\varepsilon$, $U_0$ and $U$ if necessary,   we can show that $\mathfrak{f}$ is a diffeomorphism. Note that $t\mapsto \mathfrak{f}(t,x)$ is the projection on $U$ of the integral curve of $X_h$ with initial conditions $(y,\nu(x)p_0)$ at time $t=t_0$. Now for each $x \in U$ we set $(t(x), z({x}))=\mathfrak{f}^{-1}({x})$. Of course the map ${x}\mapsto (t({x}), z({x}))$ is smooth. Now consider the $1$-form $\theta$ on $U$ defined by
$$\theta({x})=\mathsf{Fl}_2(t({x}),z({x}) ,\nu(z({x}))p_0)$$

By construction, we have $\theta(\mathfrak{f}(t,z)):={p}_z(t)$ with initial condition $p_z(t_0)=\nu(z)p_0$ and so, with these notations, we have  $\theta(x)=p_{z(x)}(t(x))$. We set $c_z(t)=\mathfrak{f}(t,z)$ and  so $(c_z,p_z)$ is an integral curve of $X_h$ with initial conditions $c_z(t_0)=z$ and $p_z(t_0)=\nu(z) p_0$.\\

As in the sub-Riemannian context (cf. \cite{Arg20}), we also have:\\

$|\theta({x}) \left( \rho_{{x}}({u}) \right)|
=|<p_{z(x)}(t(x)),\rho_x (u)>|
=|g_{({x},{v}({x},{p}_{z(x)}(t(x))}(v(x,p))_{z(x)}(t({x}))),{u})|$\\
${}\hfill
\leq \sqrt{g_{({x},{v}({x},{p}_{z(x)}(t({x}))}({v}({x},{p}_{z(x)}(t({x})),{v}({x},{p}_{z(x)}(t({x}))}\sqrt{g_{({x},{v}({x},{p}_{z(x)}(t({x}))}({u},{u})}$\\
${}\;\; \leq \sqrt{2 h({x},{p}_{z(x)}(t({x})))}\sqrt{g_{({x},{v}({x},{p}_{z(x)}(t({x}))}(u,u)}.\hfill{}$\\
since   $g_{ \left( x,v(x,p_{z(x)}(t({x}))) \right) }
 \left( v(x,p_{z(x)}(t({x}))),v(x,p_{z(x)}(t({x}) \right)
= 
2 h \left( x,p_{z(x)}(t(x)) \right) $.\\

Now from (\ref{2H}), we have
\[
h(x,\theta(x))
=h \left( x,p_{z(x)}(t(x)) \right) =
h \left( z(x),\nu(z)p_0 \right) 
=h \left( 0,p_0 \right)
\]
Therefore, for $\delta=\sqrt{2h(0,p_0)}$ from the previous inequality we obtain
\[
|\theta({x})(\rho_{{x}}({u})|\leq \delta \sqrt{g_{({x},{u}({x},{p}_{z(x)}(t({x}))}(u,u)}.
\]

\noindent It remains to show that $\theta$ is closed.  Since $\mathfrak{f}$ is a diffeomorphism,  for this purpose, it is sufficient to prove that

\begin{lemma}
\label{L_thetaClosed}
$\mathfrak{f}^*\theta=\delta ^2dt$.
\end{lemma}

This Lemma is precisely Lemma 6 in \cite{Arg20} and its proof is the same type of adaptation step by step of Arguill\`ere's Lemma as in the construction of  calibration $\theta$. Therefore this proof is left to interested readers.

It follows that $\theta$ is closed  and  so the proof Theorem~\ref{T_CharacterizationGeodesic} is completed.
\end{proof}

%\printindex
  
\end{document}